\documentclass[11pt]{amsart}
\usepackage{amsmath,amsfonts,amscd,amssymb}
\theoremstyle{plain}
\csname@addtoreset\endcsname{equation}{section}
\def\theequation{\thesection.\arabic{equation}}
\newtheorem{theorem}[equation]{Theorem}
\newtheorem{conjecture}[equation]{Conjecture}
\newtheorem{proposition}[equation]{Proposition}
\newtheorem{lemma}[equation]{Lemma}
\newtheorem{corollary}[equation]{Corollary}
\theoremstyle{definition}
\newtheorem*{acknowledgements}{Acknowledgements}
\newtheorem{remark}[equation]{Remark}
\newtheorem*{claim}{Claim}
\newtheorem{hypothesis}[equation]{Hypothesis}
\newtheorem{definition}[equation]{Definition}
\newtheorem{notation}[equation]{Notation}
\newtheorem{example}[equation]{Example}
\newtheorem{exercise}[equation]{Exercise}

\def\O{{\mathcal O}}
\def\X{{\mathcal X}}
\def\Xp{{\mathcal X}_p}
\def\F{{\mathbb F}}

\def\N{{\mathbb N}}
\def\Q{{\mathbb Q}}
\def\Qp{{\mathbb Q}_p}
\def\Qpb{\bar{\mathbb Q}_p}
\def\R{{\mathbb R}}
\def\C{{\mathbb C}}
\def\e{{\mathfrak e}}
\def\Z{{\mathbb Z}}

\def\D{{\mathfrak D}}
\def\m{{\mathfrak m}}

\def\triv{{\mathbf 1}}
\def\vtriv#1{{\mathbf 1}_{#1}}
\def\RC{{\mathcal C}}
\def\cH{{S}}
\def\EFC{\mathbf A}
\def\cA{{\mathcal A}}
\def\cO{{\mathcal O}}
\def\Gr#1{\text{\rm gr}_#1}
\def\longinjects{\lhook\joinrel\lar}
\def\injects{\lhook\joinrel\rightarrow}

\def\wAL{\frac{w(A/K,\triv\oplus\chi)}{w(\chi)^{2\dim A}}}

\def\smallmatrix#1#2#3#4{
  \genfrac{(}{.}{0pt}{1}{#1}{#3}
  \genfrac{.}{)}{0pt}{1}{#2}{#4}
}
\def\smallchoice#1#2#3#4{
  \genfrac{\{}{.}{0pt}{1}{#1}{#3}
  \genfrac{.}{\}}{0pt}{1}{#2}{#4}
}

\def\K{{\mathcal K}}
\def\L{{\mathcal L}}
\def\T{{\mathcal T}}
\def\B{{\mathcal B}}

\def\TV{{\mathcal V}}
\def\W{{\mathcal W}}
\def\eps{{\rm(Root)}}
\def\tam{{\rm(Tam)}}

\def\grphi{\Phi}

\def\lara{\langle,\rangle}

\def\blangle{\boldsymbol{\langle}}
\def\brangle{\boldsymbol{\rangle}}
\def\blara{\blangle,\brangle}

\newcommand{\vabove}[2]{\genfrac{}{}{0pt}{}{#1}{#2}}

\def\newmathop#1{\expandafter\gdef\csname #1\endcsname{\mathop{\rm #1}\nolimits}}
\newmathop{Tr}

\newmathop{rk}
\newmathop{ord}
\newmathop{Gal}
\newmathop{BSD}
\newmathop{Hom}
\newmathop{tors}
\newmathop{coker}
\newmathop{Ind}
\newmathop{Res}
\newmathop{Frob}
\let\oldchar\char\newmathop{char}\let\vchar\char\let\char\oldchar
\newmathop{GL}
\newmathop{PGL}
\newmathop{SL}
\newmathop{Aut}
\newmathop{Sel}
\newmathop{End}
\newmathop{Maps}
\newmathop{Re}
\newmathop{Reg}
\newmathop{Sy}
\newmathop{Syl}
\newmathop{sgn}

\def\notsubset{\not\subset}

\let\lar\longrightarrow
\let\iso\cong
\let\tensor\otimes
\let\normal\triangleleft

\def\leftsemidirect{\mathinner
  {\mathrel\triangleright\joinrel\mathrel{\raise 0.8pt\hbox{$\scriptstyle<$}}}}
\let\leftsemidirect\ltimes
\def\rightsemidirect{\mathinner
  {\mathrel{\raise 0.8pt\hbox{$\scriptstyle>$}}\joinrel\mathrel\triangleleft}}
\let\rightsemidirect\rtimes

\def\dlangle{\langle\!\langle}
\def\drangle{\rangle\!\rangle}
\def\dlara{\dlangle,\drangle}

\def\beq{$$\begin{array}{llllllllllllllll}}
\def\eeq{\end{array}$$}
\def\beqn{\begin{equation}\begin{array}{llllllllllllllll}}
\def\eeqn{\end{array}\end{equation}}

\input cyracc.def       
\font\tencyr=wncyr10
\def\sha{\text{\tencyr\cyracc{Sh}}}
\font\eightcyr=wncyr8

\def\smallsha{\text{\eightcyr\cyracc{Sh}}}

\def\f{\varphi}

\def\qequal{\hbox{\rlap{$\,\raise-1pt\hbox{?}$}=}}

\def\tbuildrel#1\over#2{\buildrel\text{\rm\normalsize\smaller[3]#1}\over{#2}}

\def\simf{\>\>{\tbuildrel \ref{l2:}{\rm f}\over\sim}\>\>}
\def\simd{\>\>{\tbuildrel \ref{l2:}{\rm d}\over\sim}\>\>}
\def\simr{\>\>{\tbuildrel \ref{l2:}{\rm r}\over\sim}\>\>}

\def\simef{\>\>{\tbuildrel \ref{GoverD}\over\sim}\>\>}
\def\simef{\>\>{\tbuildrel \ref{GoverD}\over\sim}\>\>}
\def\simudg{\>\>{\tbuildrel \ref{corUDG}\over\sim}\>\>}

\def\simwwe{\>\>{\tbuildrel \ref{exWWe}\over\sim}\>\>}

\def\sequal{=}
\def\sequalr{\>\>{\tbuildrel \ref{l2:}{\rm r}\over=}\>\>}

\def\sequald{\>\>{\tbuildrel \ref{l2:}{\rm d}\over=}\>\>}
\def\sequala{\>\>{\tbuildrel \ref{ladvp}\over=}\>\>}

\def\tlGDI{\raise1.4pt\hbox{$\scriptscriptstyle\!\left\lgroup\right.\!\!$}}
\def\trGDI{\raise1.4pt\hbox{$\scriptscriptstyle\!\left.\right\rgroup\!\!$}}
\def\tGDI#1{\tlGDI#1\trGDI}  
\def\lGDI{\raise0.8pt\hbox{$\scriptstyle\!\left\lgroup\right.\!\!$}}
\def\rGDI{\raise0.8pt\hbox{$\scriptstyle\!\left.\right\rgroup\!\!$}}
\def\GDI#1{\lGDI#1\rGDI}   
\def\blGDI{\raise0.0pt\hbox{$\!\left\lgroup\right.\!\!$}}
\def\brGDI{\raise0.0pt\hbox{$\!\left.\right\rgroup\!\!$}}
\def\bGDI#1{\blGDI#1\brGDI}   

\def\thincdots{\raise1.3pt\hbox{$\scriptscriptstyle
  \>\cdot\>\cdot\>\cdot\>\cdot\hskip0.3pt$}}
\def\tw#1{{\lfloor\frac{#1}{12}\rfloor}}

\def\IZ{\text{\rm I$_0$}}
\def\IZS{\text{\rm I$_0^*$}}
\def\In#1{{\text{\rm I{}$_{#1}$}}}
\def\InS#1{{\text{\rm I{}$_{#1}^*$}}}
\def\II{\text{\rm II}}
\def\IIS{\text{\rm II$^*$}}
\def\III{\text{\rm III}}
\def\IIIS{\text{\rm III$^*$}}
\def\IV{\text{\rm IV}}
\def\IVS{\text{\rm IV$^*$}}

\def\choice#1#2#3#4{{\def\arraystretch{0.7}
\Bigl\{\!\!\begin{array}{ll}
   \scriptstyle #1,\!\!\!&\scriptstyle #2\cr
   \scriptstyle #3,\!\!\!&\scriptstyle #4\end{array}}\!\!\Bigr\}\,}
\def\leftchoice#1#2#3#4{{\def\arraystretch{0.7}
\Bigl\{\!\!\begin{array}{ll}
   \scriptstyle #1,\!\!\!&\scriptstyle #2\cr
   \scriptstyle #3,\!\!\!&\scriptstyle #4\end{array}}}

\def\Step#1#2{\smallskip\par\noindent{\tt Step #1: #2}.\par\smallskip}

\def\CO{C}
\def\Cy{{\rm C}}
\def\Di{{\rm D}}
\def\Vi{{\rm V}}
\def\Qu{{\rm Q}}
\def\Sym{{\rm S}}
\def\Alt{{\rm A}}

\def\delt{\eth}

\def\neron#1{\omega_{#1}^o}

\def\Tau{{\mathbf T}}
\def\Ttp{\Tau_{\Theta,p}}

\def\subgroup{\raise0.5pt\hbox{$\,\scriptstyle<\,$}}
\let\<\subgroup

\def\Fr{{\rm F}}

\def\picbox#1#2#3#4#5{\begingroup\arraycolsep 2pt$\begin{array}{|c|}\hline
   \hfill {\scriptstyle #1}\,\hfill\vline\hfill\,{\scriptstyle #2}\hfill\cr\hline \hfill {\scriptstyle #3}\,\hfill\vline
   \hfill\,\scriptstyle\text{{#4}}{\scriptstyle,#5}\hfill\cr\hline\end{array}$\endgroup}
\def\pic#1(#2,#3)#4#5#6#7#8{\put(#2,#3){\boxdims{\hbox{\picbox{#4}{#5}{#6}{#7}{#8}}}{#1}{#2}{#3}}}
\def\boxdims#1#2#3#4{#1}
\def\addline#1,#2;{}

\begin{document}

\let\introdagger\dagger
\title{Regulator constants and the parity conjecture}
\author{Tim$^\introdagger$ and Vladimir Dokchitser}
\date{March 21, 2009}
\thanks{{\em MSC 2000:} Primary 11G05; Secondary 11G07, 11G10, 11G40, 19A22, 20B99}
\thanks{$^\dagger$Supported by a Royal Society University Research Fellowship}
\address{Robinson College, Cambridge CB3 9AN, United Kingdom}
\email{t.dokchitser@dpmms.cam.ac.uk}
\address{Gonville \& Caius College, Cambridge CB2 1TA, United Kingdom}
\email{v.dokchitser@dpmms.cam.ac.uk}

\begin{abstract}
The $p$-parity conjecture for twists of elliptic curves relates
multiplicities of Artin representations in $p^\infty$-Selmer groups
to root numbers. In this paper we prove this conjecture
for a class of such twists.
For example, if $E/\Q$ is semistable at 2 and 3,
$K/\Q$ is abelian and $K^\infty$ is its maximal pro-$p$ extension,
then the $p$-parity conjecture holds for twists of $E$ by all orthogonal
Artin representations of $\Gal(K^\infty/\Q)$.
We also give analogous results when $K/\Q$ is non-abelian,
the base field is not $\Q$ and $E$ is replaced by an abelian variety.
The heart of the paper is a study of relations between permutation
representations of finite groups, their ``regulator constants'',
and compatibility between local root numbers and local
Tamagawa numbers of abelian varieties in such relations.
\end{abstract}

\llap{.\hskip 10cm} \vskip -0.8cm
\maketitle
\def\introdagger{{}}
\let\oldtocsubsection\tocsubsection
\def\tocsubsection#1#2#3{#1\hbox to 25pt{#2.\hfill} #3}
\tableofcontents
\let\tocsubsection\oldtocsubsection

\section{Introduction}
\label{intro}

The emphasis of this paper is twofold: to study the
interplay between functions on $G$-sets and on
$G$-representations for a finite group $G$, and
to use it to link root numbers
and Tamagawa numbers of abelian varieties.
The main application is the parity conjecture
for classes of twists of elliptic curves and
abelian varieties by Artin representations.

\subsection{Parity conjectures}

Consider an abelian variety $A$ defined over a~num\-ber field $K\!$,
and a Galois extension $F/K\!$.
The Galois group $\Gal(F/K)$~acts~on the $F$-rational points of $A$,
and an extension of the Birch--Swinnerton-Dyer conjecture
relates the multiplicities of complex representations
in $A(F)\tensor\C$ to the
order of vanishing of the corresponding twisted $L$-functions at $s=1$:

\begin{conjecture}[Birch--Swinnerton-Dyer--Deligne--Gross; \cite{TatC,DelV}, \cite{RohV} \S2]
For every complex representation $\tau$ of $\Gal(F/K)$,
$$
  \blangle\tau,A(F)\brangle=\ord_{s=1} L(A,\tau,s).
$$
\end{conjecture}
\noindent
(Here and below $\blangle\tau,*\brangle$ is the usual representation-theoretic
inner product of $\tau$ and the complexification of $*$.)
When $\tau$ is self-dual, the parity of the right-hand side
is forced by the sign in the conjectural functional equation,
the global root number $w(A/K,\tau)$. Thus, we expect
$$
  (-1)^{\blangle \tau, A(F)\brangle}=w(A/K,\tau).
$$
There is an analogous picture for Selmer groups. For a prime $p$, let
$$
  \small
  \Xp(A/K)=(\text{Pontryagin dual of the $p^\infty$-Selmer group of } A/K)\>\tensor\Q_p.
$$
This is a $\Q_p$-vector space whose dimension is simply
the Mordell-Weil rank of $A/K$ plus the number of copies of $\Q_p/\Z_p$
in the Tate-Shafarevich group $\sha(A/K)$.
The conjectural finiteness of $\sha$ then suggests the following
parity statements, which are often much more accessible:

\let\oldequation\theequation
\def\theequation{1.2a}
\begin{conjecture}[$p$-Parity Conjecture]
\label{pparity}
$$
  (-1)^{\dim\Xp(A/K)}=w(A/K).
$$
\end{conjecture}
\def\theequation{1.2b}
\noindent
Similarly, for a self-dual representation $\tau$ of $\Gal(F/K)$, we expect
\begin{conjecture}[$p$-Parity Conjecture for twists]
\label{pparitytwists}
$$
  (-1)^{\blangle\tau, \Xp(A/F)\brangle}=w(A/K,\tau).
$$
\end{conjecture}
\setcounter{equation}{2}
\let\theequation\oldequation

\noindent
When $K\!\subset\!L\!\subset\!F$,
the second statement for
the permutation representation on the set of $K$-embeddings $L\hookrightarrow F$
is equivalent to the first one for $A/L$.
So \ref{pparitytwists} for all orthogonal twists
implies \ref{pparity} over all intermediate fields~of~$F/K$.

The main applications of this paper confirm special cases of
Conjecture~\ref{pparitytwists}. Here are two specific examples:

\begin{theorem}
\label{introthm1}
Let $E/\Q$ be an elliptic curve, semistable at 2 and 3.
Suppose $F/\Q$ is Galois and
the commutator subgroup of $G=\Gal(F/\Q)$ is a $p$-group.
Then the $p$-parity conjecture holds
for twists of $E$ by all orthogonal representations of~$G$ and,
in particular, over all subfields of~$F$.
\end{theorem}

\begin{theorem}
\label{introthm2}
Let $p$ be an odd prime, and suppose $F/K$ is Galois
and $P\normal \Gal(F/K)$ is a $p$-subgroup.
Let $A/K$ be a principally polarised abelian variety
whose primes of unstable reduction are unramified in $F/K$.
If the $p$-parity conjecture holds for $A$ over the subfields
of $F^P/K$, then it holds over all subfields of~$F/K$.
\end{theorem}

\noindent The general results on the $p$-parity conjecture (Theorems \ref{tamroot},
\ref{introthmGPtwist} and~\ref{introthm3}) are given in \S\ref{ssappl}.
But first we introduce
our main tool from group theory,
which may be of independent interest.

\subsection{$G$-sets versus $G$-representations}
\label{ssGG}

Let $G$ be an abstract finite group. Suppose $\phi: H \mapsto \phi(H)$
is a function that associates to every subgroup $H\<G$ a value in some
abelian group $\cA$ (written multiplicatively), and that $\phi$ takes the same value on
conjugate subgroups.
Recall that $H\leftrightarrow G/H$ is a bijection between
subgroups of $G$ up to conjugacy and transitive $G$-sets up to isomorphism.
So $\phi$ extends to a map from all $G$-sets to $\cA$ by the rule
$\phi(X\amalg Y)\!=\!\phi(X)\phi(Y)$. Let us call 
$\phi$ ``representation-theoretic'' if $\phi(X)$
only depends on the representation $\C[X]$.

Alternatively, say that a formal combination of (conjugacy classes of)
subgroups $\Theta=\sum_i n_i H_i$ is a {\em relation between permutation representations
of~$G$}, or simply a {\em $G$-relation}, if
$$
   \bigoplus\nolimits_i \C[G/H_i]^{\oplus n_i} \iso 0,
$$
as a virtual representation, i.e. the character
$\sum_i n_i \chi_{\scriptscriptstyle \C[G/H_i]}$ is zero.
Then for $\phi$ to be representation-theoretic is equivalent to
$\prod_i \phi(H_i)^{n_i}$ being 1 for every such $G$-relation.

For example, $G=\Sym_3$ has a unique relation up to multiples,
$$\Theta=2\Sym_3 + \{1\} - 2\Cy_2 - \Cy_3\>.$$
(i.e. $\triv^{\oplus 2} \oplus \C[\Sym_3]\iso
\C[\Sym_3/\Cy_2]^{\oplus 2} \oplus \C[\Sym_3/\Cy_3]$;
clearly such a relation must~exist: $\Sym_3$ has 4 
subgroups up to conjugacy, but only 3 
irreducible representations.)

In the context of number theory, $G$ may be a Galois group of a number field
$F/\Q$, and $\phi(H)$ some invariant of the intermediate field $F^H$.
For instance, $\phi(H)$ could be the the degree of $F^H$, its discriminant,
class number or Dedekind zeta-function $\zeta_{F^H}(s)$
(with $\cA=\Z,\Q^\times,\Q^\times$
and the group of non-zero meromorphic functions on $\C$, respectively).
Of these four, all but the class number are representation-theoretic,
e.g. $[F^H\!:\!\Q]=\dim\C[G/H]$ and $\zeta_{F^H}(s)=L(\C[G/H],s)$
are visibly functions of $\C[G/H]$.
It follows, for example, that in every $\Sym_3$-extension $F/\Q$,
$$\zeta(s)^2\zeta_F(s)=\zeta_{F^{\Cy_2}}(s)^2\zeta_{F^{\Cy_3}}(s).$$
The class number formula then yields an explicit identity between
the corresponding class numbers and regulators ($\tfrac{h\cdot \Reg}{|\mu|}$
is representation-theoretic).

We are going to study extensively $G$-relations and functions on
$G$-relations, and present techniques for verifying when a function or
a quotient of two such functions is representation-theoretic
(see \S\ref{smachinery}).

\textheight 594pt
\pagebreak
\textheight 584pt

\subsubsection*{Regulator constants}
Of particular interest to us is the function
$$
\D_{\rho}:H\mapsto\det(\tfrac{1}{|H|}\lara|\rho^{H})\in \K^{\times}/\K^{\times2}
$$
that,
for a fixed self-dual $\K G$-representation $\rho$
($\K$ a field)
with a $G$-invariant pairing~$\lara$,
computes the determinant of the matrix representing $\tfrac{1}{|H|}\lara$
on any basis of the $H$-invariants $\smash{\rho^H}\!$.
Its significance will become clear~when~we
dis\-cuss functions coming
from abelian varieties.
The fundamental property
of $\D_{\rho}$ is that if $\dlara$ is another pairing on $\rho$,
then $\D_{\rho}^{\lara}/\D_{\rho}^{\dlara}$
is representation-theoretic. In other
words, for every $G$-relation $\Theta=\sum_i n_i H_i$, the quantity
$$
  \RC_\Theta(\rho) = \prod\nolimits_i \D_\rho(H_i)^{n_i} \in \K^{\times}/\K^{\times2}
$$
is independent of the pairing. Following \cite{Squarity} we call
$\RC_\Theta(\rho)$ the {\em regulator constant\/} of $\rho$.
(Their properties are discussed in \S\ref{s:regconst} and \S\ref{svanishing}.)

\begin{example}
\label{introdihA}
Suppose $G\!=\!\Di_{2p^n}$ is dihedral with $p\ne 2$, and $\K\!=\!\Q$ or $\Q_p$.
The smallest subgroups $\{1\}, \Cy_2, \Cy_p$ and $\Di_{2p}$ form
a $G$-relation 
$$
  \Theta=\{1\} + 2\>\Di_{2p} - \Cy_p - 2\>\Cy_2\>.
$$
The irreducible $\K G$-representations are $\triv$ (trivial),
$\epsilon$ (sign) and $\rho_k$ of dimension $p^k\!-\!p^{k-1}$ for
every $1\le k\le n$; they are all self-dual. An elementary computation
(see Examples \ref{exdihreg}, \ref{exdihreg2}) shows that
$$
  \RC_{\Theta}(\triv)=\RC_{\Theta}(\epsilon)=\RC_{\Theta}(\rho_n)=p, \qquad
  \RC_{\Theta}(\rho_k)=1 \quad (1\le k<n).
$$
\end{example}

\subsection{Main results and applications}
\label{ssappl}

The central result of this paper is the $p$-parity conjecture
for the following twists: 
for a 
group $G$, a prime $p$ and a $G$-relation~$\Theta$, define
$\Ttp$ to be the set of self-dual $\Qpb G$-representations $\tau$ that satisfy
$$
 \blangle\tau,\rho\brangle \equiv \ord_p \RC_\Theta(\rho) \mod 2
$$
for every self-dual $\Q_p G$-representation $\rho$
(computing $\RC_{\Theta}(\rho)$ with $\K=\Q_p$).

\let\oldequation\theequation
\def\theequation{1.6}\refstepcounter{equation}\label{tamroot}

\def\theequation{1.6(a)}
\begin{theorem}
\label{tamroot1}
Let $F/K$ be a Galois extension of number fields.
Suppose $E/K$ is an elliptic curve
whose primes of additive reduction above 2 and~3
have cyclic decomposition groups (e.g. are unramified) in $F/K$.
For every 
~$p$ and every relation $\Theta$ between
permutation representations of $\Gal(F/K)$,
$$
  (-1)^{\blangle\tau,\Xp(E/F)\brangle}=w(E/K,\tau)
  \qquad \text{for all\, $\tau\in \Ttp$}.
$$
\end{theorem}

\def\theequation{1.6(b)}

\begin{theorem}
\label{tamroot2}
Let $F/K$ be a Galois extension of number fields.
Suppose $A/K$ is a principally polarised abelian variety
whose primes of unstable reduction
have cyclic decomposition groups in $F/K$. Let $p$ be a prime, and assume
that either
\begin{itemize}
\item $p\ne 2$, or
\item $p=2$, the principal polarisation is induced by a $K$-rational
      divisor, and $A$ has split semistable reduction at primes $v|2$
      of $K$ which have non-cyclic wild inertia group in $F/K$.
\end{itemize}
For every relation $\Theta$ between permutation representations
of $\Gal(F/K)$,
$$
  (-1)^{\blangle\tau,\Xp(A/F)\brangle}=w(A/K,\tau)
  \qquad \text{for all\, $\tau\in \Ttp$}.
$$
\end{theorem}

\setcounter{equation}{6}
\let\theequation\oldequation

\begin{remark}
In particular, if $p$ is odd, Conjecture \ref{pparitytwists} holds
for $\tau\in\Ttp$ for all semistable principally polarised
abelian varieties over $K$.
\end{remark}

\begin{remark}
In general, the representations in $\Ttp$ simply encode the regulator
constants. For instance, if $\{\rho_j\}$ are 
the irreducible self-dual $\Q_p G$-representations with
$\ord_p \RC_\Theta(\rho_j)$ odd, then
$$
 \bigoplus\nolimits_{j} \>(\text{any $\Qpb$-irreducible constituent of $\rho_j$}) \>\>\in \>\>\Ttp\>.
$$
(This representation is automatically
self-dual by Corollary \ref{triviality}). 
A general element in $\Ttp$ differs from this one by an element of
$\Tau_{0,p}$, a \hbox{representation} for which the $p$-parity conjecture ought
to ``trivially'' hold;
it would be very interesting to have an intrinsic description of $\bigcup_{\Theta}\!\Ttp$,
\hbox{cf.} Remarks~\ref{remT0p},~\ref{intrinsic}.
\end{remark}

\begin{example}[$\Gal(F/K)\!=\!\Di_{2p^n}$, $p$ odd]
\label{introdih0}
Continuing Example \ref{introdihA}, for every $\Qpb$-irreducible
(2-dimen\-sional) constituent $\tau_n$ of $\rho_n$,
$$
  \triv\oplus\epsilon\oplus\tau_n\>\>\in\>\>\Tau_{\Theta,p}\>.
$$
If $A/K$ is an abelian variety that satisfies the assumptions of
Theorem \ref{tamroot}, e.g. $A$ is semistable
at primes that ramify in $F/K$,
the $p$-parity conjecture holds for the twist of $A$ by
$\triv\oplus\epsilon\oplus\tau_n$.
Applying this construction to the $\Di_{2p^k}$-quotients
of $\Di_{2p^n}$,
we deduce the $p$-parity conjecture for the twists of $A$
by~$\triv\oplus\epsilon\oplus\tau$ for every 2-dimensional irreducible
representation $\tau$ of $\Gal(F/K)$.
\end{example}

\noindent
As the $p$-parity conjecture is known to hold for elliptic curves over $\Q$,
and therefore for their quadratic twists as well,
we find 

\begin{corollary}[Parity conjecture in anticyclotomic towers]
\label{anticyc}
Let $E/\Q$ be an elliptic curve, $L$ an imaginary
quadratic field and $p$ an odd prime.
If $p\!=\!3$, assume that either $E$ is semistable at $3$ or that $3$ splits
in $L$.
Then for every layer $L_n$ of the $\Z_p$-anticyclotomic extension of $L$
and every representation~$\tau$ of $\Gal(L_n/\Q)$,
$$
  \ord_{s=1} L(E,\tau,s) \equiv \blangle\tau, \Xp(E/L_n)\brangle \mod 2.
$$
\end{corollary}

In \S\ref{stowers} we generalise Example \ref{introdih0} to other groups
with a large normal $p$-subgroup.
Based on Theorem~\ref{tamroot}, and using purely group-theoretic manipulations
we obtain (see Theorems \ref{thmGPtwist}, \ref{hypellthm})

\begin{theorem}
\label{introthmGPtwist}
Suppose $F/K$ is a Galois extension of number fields and
the commutator subgroup of $G=\Gal(F/K)$ is a $p$-group.
Let $A/K$ be an abelian variety satisfying the hypotheses of
Theorem \ref{tamroot}.
If the $p$-parity conjecture holds for $A$ over $K$
and its quadratic extensions in $F$,
then it holds for all twists of $A$ by orthogonal representations of $G$.
\end{theorem}

When $A\!=\!E$ is an elliptic curve and $K\!=\!\Q$, the
assumption on the \hbox{$p$-parity} conjecture is always satisfied, as we remarked
above (in particular, we get Theorem \ref{introthm1}).
It also holds
for those $E/K$ that admit a rational $p$-isogeny
under mild restrictions on $E$ at primes above $p$;
see Remark \ref{knowncases} for precise statements, an
extension to abelian varieties and a list of references.

\textheight 594pt
\pagebreak
\textheight 584pt

The condition that the commutator of $G$ is a $p$-group is equivalent to
the Sylow $p$-subgroup being normal with an abelian quotient.
In other words, $F$ should be
a $p$-extension of an abelian extension of the ground field.
For instance, the theorem applies when
\begin{itemize}
\item $G$ is abelian (any $p$).
\item $G\iso \Di_{2p^n}$ is dihedral.
\item $G$ is a 2-group and $p=2$.
\item $G$ is an extension of $\Cy_2$ by a $p$-group.
\item $G\iso (\Z/p^n\Z)\rtimes(\Z/p^n\Z)^{\times}$, for instance $F\!=\!\Q(\mu_{p^n},\sqrt[p^n]{m})$, $K\!=\!\Q$.
\item $G\subset\smallmatrix{\,*}{\,*\,}{p*}{\,*\,}$ in $\GL_2(\Z/p^n\Z)$,
      for instance $F\!=\!K(C[p^n])$ for some elliptic curve $C/K$ that admits
      a rational $p$-isogeny.
\end{itemize}

Root numbers and parities of Selmer ranks in the last 3 cases
have recently been studied
by Mazur--Rubin \cite{MR,MR2}, Hachimori--Venjakob \cite{HV} and
one of us (V.) \cite{VD},
Rohrlich \cite{RohS} and Coates--Fukaya--Kato--Sujatha \cite{CFKS};
see also Greenberg's preprint \cite{GreP}.
This kind of extensions arise in non-commutative Iwasawa theory,
where one has a tower $F_\infty=\bigcup F_n$ with $\Gal(F_\infty/K)$ a
\hbox{$p$-adic} Lie group.
The $\Gal(F_n/K)$ all have a ``large'' normal $p$-subgroup
with a fixed ``small'' quotient.
When this quotient is {\em non-abelian\/}, we have a
weaker version of Theorem \ref{introthmGPtwist}
(Theorems \ref{thmGP}, \ref{hypellthm}; cf. also \ref{thmGP2} for $p=2$):

\begin{theorem}
\label{introthm3}
Suppose $F/K$ is Galois and $P\normal \Gal(F/K)$ is a $p$-subgroup
with $p\!\ne\! 2$.
Let $E/K$ be an elliptic curve (resp. principally polarised abelian variety)
whose primes of additive reduction above 2, 3 (resp. all primes of
unstable reduction) have cyclic decomposition groups in $F/K$.
If the \hbox{$p$-parity} conjecture holds for $E$ over the subfields
of $F^P/K$, then it holds over all subfields of~$F/K$.
\end{theorem}

\begin{example}
Let $E/\Q$ be an elliptic curve, semistable at 2 and 3.
Take $p\ne 2$ and $F_n=\Q(E[p^n])$, so
$\Gal(F_n/\Q)\<\GL_2(\Z/p^n\Z)$.
If the $p$-parity conjecture holds over the subfields of
the first layer $\Q(E[p])/\Q$, then it holds over all subfields of $F_n$
for all $n$.
Incidentally, for $p=3$ the ``first layer'' assumption is always satisfied
(see Example \ref{exgl2f3}).

Using the above theorems,
it is also possible to get a lower estimate on
the growth of the $p^\infty$-Selmer group in this tower
by computing root numbers.
For example, if $p\equiv3\mod4$ and $E$ is semistable
and admits a rational $p$-isogeny, then combining Theorem \ref{introthm1}
and \cite{RohS} Cor.~2 shows that
$\dim\Xp(E/F_n)\ge ap^{2n}$ for some $a>0$ and large enough $n$.
\end{example}

Finally, let us point out some of the things that definitely can {\em not\/}
be \hbox{obtained} just from Theorem \ref{tamroot}.
It is tempting to try and prove the \hbox{$p$-parity} conjecture for $A/K$ itself by
finding a clever extension $F/K$ and a $\Gal(F/K)$-relation $\Theta$ with
$\triv\in\Ttp$. However, Theorem \ref{tauthetapprop} shows that
all $\tau\in\Ttp$ are even-dimensional (and have trivial determinant).
So, even assuming finiteness of $\sha$ and
using several primes $p$, one requires at least
one additional twist for which parity is known.
For instance, the \hbox{$p$-parity} conjecture for all elliptic curves over $\Q$
can be proved for odd $p$ by reversing the argument in
\ref{introdih0} and \ref{anticyc}: it is possible to find a suitable
anticyclotomic extension where one knows $p$-parity for the twists
by $\epsilon$ and some 2-dimensional irreducible~$\tau$, whence it is
also true for $\triv$. (This is the argument used in \cite{Squarity}.)

It is also worth mentioning that
if $\rho$ is an irreducible $\Q_pG$-representation which
is either symplectic or of the
form $\sigma\oplus\sigma^*$ over $\bar\Q_p$,
then
$(-1)^{\blangle\tau,\rho\brangle}=1$ for every $\Theta$ and $\tau\in\Ttp$,
so Theorem \ref{tamroot}
yields no information about the parity of such $\rho$ in $\Xp(A/F)$.
Also, the theorem gives no interesting $p$-parity statements
when $p\nmid |G|$ or $G$ has odd order.

For a summary of properties of $\tau\in\Ttp$ and examples
see \S\ref{svanishing}.

\subsection{Regulator constants and parity of Selmer ranks}
\label{sssquality}

To explain our approach to the parity conjecture,
let us first review the method of 
\cite{Squarity,Selfduality} which allows one
to express the Selmer parity in Theorem~\ref{tamroot}
in terms of local invariants of the abelian variety.

Suppose $F/K$ is a Galois extension of number fields.
For simplicity, consider an elliptic curve $E/K$,
and assume for the moment that the Tate-Shafarevich group
$\sha$ is finite. Define the {\em Birch--Swinnerton-Dyer quotient}
$$
  \BSD(E/K)=
  \frac{\Reg_{E/K}\>|\sha(E/K)|}{\sqrt{|\Delta_K|}\>|E(K)_{\text{tors}}|^2}\cdot
  \CO_{E/K},
$$
the conjectural leading term of $L(E/K,s)$ at $s\!=\!1$, see \cite{TatC} \S1.
Here $\Reg$~is the regulator,
$\CO_{E/K}\!=\!\prod_v C_v(E/K_v,\omega)$ the product
of local \hbox{Tamagawa} \hbox{numbers} and periods, and
$\Delta_K$ is the discriminant of $K$
(see \S\ref{ssnotation} for the \hbox{notation}).

Whenever $E_i/K_i$ are elliptic curves (or abelian varieties)
that happen to satisfy $\prod_i L(E_i/K_i,s)^{n_i}=1$,
then $\prod_i \BSD(E_i/K_i)^{n_i}=1$
as predicted by the Birch--Swinnerton-Dyer conjecture%
\footnote{%
  If $\prod L(E_i/F_i,s)=\prod L(E'_j/F'_j,s)$, the corresponding products
  of Weil restrictions to $\Q$ have the same $L$-function, hence isomorphic
  $l$-adic representations (Serre \cite{SerA} \S2.5 Rmk. (3)), and are
  therefore isogenous (Faltings \cite{Fa}). This is sufficient,
  as $\smallsha$ is assumed finite and BSD-quotients are invariant
  under Weil restriction (Milne \cite{MilO} \S1)
  and isogeny (Tate--Milne \cite{MilA} Thm. 7.3).}.
Taking the latter modulo rational squares (to eliminate $\sha$ and
torsion) yields a relation between the regulators and the local terms $\CO$.
It turns out, and has already been exploited in \cite{Isogroot,Squarity},
that this has strong implications for parities of ranks.

As a first example, if $E$ admits a $K$-rational {\em $p$-isogeny\/} $E\to E'$,
then the equality $L(E/K,s)=L(E'/K,s)$ leads to the congruence
$$
  \frac{\CO_{E/K}}{\CO_{E'/K}} \equiv \frac{\Reg_{E'/K}}{\Reg_{E/K}}
    \equiv p^{\rk(E/K)} \mod \Q^{\times2},
$$
where the second step is an elementary computation with height pairings.

As a second example,
if $E/K$ is arbitrary and $E_d/K$ is its
{\em quadratic twist\/} by $d\in K^{\times}$, then
$L(E/K,s)L(E_d/K,s)=L(E/K(\sqrt{d}),s)$,~and
$$
  \biggl|
  \frac{{\raise-2pt\hbox{$\Delta$}\scriptstyle{}^{1/2}_{K(\sqrt{d})}}}{\Delta_K}
  \biggr|\>
  \frac{\CO_{E/K}\CO_{E_d/K}}{\CO_{E/K(\sqrt d)}}
  \equiv \frac{\Reg_{E/K(\sqrt{d})}}{\Reg_{E/K}\Reg_{E_d/K}}
    \equiv 2^{\rk(E/K(\sqrt{d}))} \mod\Q^{\times2}.
$$

The main subject of this paper is another massive source
of identities between $L$-functions,
{\em relations between permutation representations.}
If $F/K$ is a Galois extension with Galois
group $G$, then a $G$-relation
$$
  \Theta:\>\> \bigoplus\nolimits_i \C[G/H_i]^{\oplus n_i}=0 \qquad (H_i\< G, \>n_i\in\Z)
$$
forces the identity $\prod L(E/F^{H_i},s)^{n_i}\!=\!1$ by Artin formalism,
which leads to
$\prod (C_{E/F^{H_i}})^{-n_i} \equiv \prod (\Reg_{E/F^{H_i}})^{n_i} \mod \Q^{\times2}.$
By definition of the regulator,
$$
 \qquad\qquad\qquad \Reg_{E/F^{H_i}}=\det(\tfrac{1}{|H_i|}\lara|\rho^{H_i})
\qquad\quad(=\D_\rho(H_i) \text{ of \S\ref{ssGG}})
$$
where $\rho=E(F)\tensor\Q$ and $\lara$ is the height pairing on $E/F$.
So the multiplicities
$\rk_\sigma(E/F)$
with which
various irreducible $\Q G$-representations $\sigma$ occur in $E(F)\tensor\Q$
satisfy
$$
  \prod_i (C_{E/F^{H_i}})^{-n_i} \equiv
  \prod_i (\Reg_{E/F^{H_i}})^{n_i} \equiv \RC_\Theta(\rho) \equiv \prod_\sigma \RC_\Theta(\sigma)^%
    {\rk_\sigma(E/F)}
  \mod \Q^{\times2}.
$$
In other words, the $p$-parts of the left-hand side determine the parities
of specific ranks:
for any $\tau_p\in\Ttp$,
$$
\prod_i (C_{E/F^{H_i}})^{n_i} \equiv \prod_p \> p^{\,\blangle\tau_p,E(F)\brangle} \mod \Q^{\times2}\>.
$$

The three procedures may be carried out without assuming that $\sha$ is finite,
at the expense of working with Selmer groups rather than Mordell-Weil
groups. In the first two cases, the outcome is
\beq
\dim\X_p(E/K) & \equiv\!\!\! & \ord_p \frac{\CO_{E/K}}{\CO_{E'/K}} \mod 2 & (\text{isogeny}),\cr
\dim\X_2(E/K(\sqrt d))\!\!\! & \equiv\!\!\! & \ord_2
\frac{\CO_{E/K}\CO_{E_d/K}}{\CO_{E/K(\sqrt{d})}}\>
  \raise2pt\hbox{$\Bigl|$}\!
  \tfrac{{\raise-2pt\hbox{$\scriptstyle\Delta$}^{1/2}_{K(\sqrt{d})}}}{\Delta_K}
  \!\raise2pt\hbox{$\Bigr|$}
  \mod 2& (\text{quad. twist}). \cr
\eeq
In the case of $G$-relations, according to \cite{Selfduality} Thms. 1.1, 1.5,
we have

\begin{theorem}
\label{squality}
Let $F/K$ be a Galois extension of number fields with Galois group $G$.
Let $p$ be a prime and
$\Theta=\sum n_i H_i$ a $G$-relation.
For every elliptic curve $E/K$,
the $\Q_p G$-representation $\Xp(E/F)$ is self-dual, and
$$
\blangle\tau,\Xp(E/F)\brangle 
\equiv
\ord_p\prod_i (C_{E/F^{H_i}})^{n_i} \mod 2\>
\qquad \text{for all $\tau\in\Ttp$}.
$$
The same is true for principally polarised abelian varieties $A/K$, except
that when $p=2$ we require that the polarisation comes from a $K$-rational
divisor.
\end{theorem}

\begin{remark}
In contrast to Theorem \ref{tamroot}, this result
has no constraints on the reduction types of the abelian variety.
So it always gives an expression for $\blangle\tau,\Xp(A/F)\brangle$ for
$\tau\in\Ttp$ in terms of local data.
\end{remark}

\begin{example}
\label{introdih1}
As in Example \ref{introdih0}, suppose $\Gal(F/K)=\Di_{2p^n}$.
Then for a faithful 2-dimensional 
representation $\tau$,
the parity of
$\blangle\triv\oplus\epsilon\oplus\tau,\Xp(A/F)\brangle$
is determined by local Tamagawa numbers, as
$\ord_p C_{A/F}/C_{A/F^{\Cy_p}}\mod 2$.
(Mazur and Rubin have another local expression for {\em precisely\/} the same
\hbox{parity}; see \cite{MR} Thm.\ A.)
\end{example}

\subsection{Root numbers and Tamagawa numbers}

We have explained how in three situations
($p$-isogeny, quadratic twist, $G$-relations)
the parity of some Selmer rank can be expressed in terms of
{\em local\/} Tamagawa numbers.
As root numbers are also products of {\em local\/} root numbers,
this suggests a proof of the corresponding case of
the parity conjecture by a place-by-place comparison
(cf. \cite{CFKS,Isogroot} for the isogeny case and \cite{Kra,KT} for
quadratic twists).

There are two subtle points:

First, the local terms do not always agree.
In each case, one needs a good expression for the root numbers,
separating the part that does agree with $C_v$ and an ``error term''
that provably dies after taking the product over all places.
This error term in the isogeny case is an Artin symbol
(\cite{CFKS}~Thm.~2.7, \cite{Isogroot} Thms. 3, 4),
for quadratic twists it is
a Legendre symbol (\cite{KT} p. 307), and in our case
it comes out as 
the local root number $w(\tau)^{2\dim A}$
(see Theorem~\ref{loccompat}).
In fact, for group-theoretic reasons this
contribution is trivial
(Lemma~\ref{detformula} and Theorem~\ref{tauthetapprop}(1)),
so here the local terms {\em do\/} agree.

Second, although the remaining compatibility of
the corrected local root number
and $C_v$ is a genuinely local problem,
they are computed for completely different objects --- for instance
in Example \ref{introdih1} the representation
$\triv\oplus\epsilon\oplus\tau$ bears little resemblance to
$C_{E/F}/C_{E/F^{\Cy_p}}$.
In the isogeny and quadratic twist cases, the proof of this compatibility
in \cite{CFKS,Isogroot,Kra} boils down to brutally working out an explicit
formula for each term separately. That the two formulae then agree comes
out as a miracle.
In our case, for a {\em fixed\/} Galois group $G$ and relation $\Theta$
this strategy works equally brutally,
cf.~\cite{Squarity}~Prop.~3.3 for $G=\smallmatrix1*0* \subset \GL_2(\F_p)$.

The general case occupies \S\ref{sloccompat} and relies on the theory
of $G$-relations and regulator constants from \S\ref{smachinery}. We first
reduce our ``semilocal'' problem (places can split in $F/K$) to one
about abelian varieties over local fields. If now $A/K$ is an abelian
variety over a {\em local\/} field, 
in all cases covered by Theorem \ref{tamroot} there is an explicit
$\lambda=\pm 1$ and a $\Gal(\bar K/K)$-module $V$ such that
$$
  w(A/K,\tau) = w(\tau)^{2\dim A}\,\lambda^{\dim \tau}\,
    (-1)^{\blangle\tau,V\brangle}
$$
for all self-dual $\tau$
(see Table \ref{roottable}).
For instance, when $A$ is semistable, $\lambda=1$ and $V=X(\T^*)\otimes\Q$
is the character group of the toric part of the reduction of the dual
abelian variety (Proposition \ref{rootabtwist}).
The compatibility statement reduces to proving that the function
$\tfrac{\D_V}{C_v}$
to $\Q^{\times}/\Q^{\times2}$ is representation-theoretic
in the sense of \S\ref{ssGG} (cf. Theorem \ref{loccompat}).

Let us note here that the statement of Theorem \ref{squality} is
that
${\D_{\X_p}}/{\prod C_v}$
is representation-theoretic.
Thus $V$ plays the r\^ole of a ``local version'' of the Selmer module $\X_p(A/F)$.
Curiously, $V$ is a rational representation, which is only
conjecturally true of the Selmer module.

\begin{example}
Take $K=\Q_p$ for an odd prime $p$, and $E/K$ an elliptic
curve with non-split multiplicative reduction of type $\In{n}$. In this case
the module $V$ that computes root numbers is the $1$-dimensional unramified
character of order $2$. Let us consider $\D_V$ and $C_v$
($=c_v$, the local Tamagawa number) in the unique $\Cy_2\times\Cy_2$ extension of $\Q_p$:

\begingroup\smaller[2]
$$
\begin{picture}(095,43)
\put(18,40){$2n$}
\put(00,20){$n$}
\put(20.5,20){$2$}
\put(40,20){$2$}
\put(10,00){$1\text{ or }2$}
\put(7,27){\line(1,1){10}}
\put(38,27){\line(-1,1){10}}
\put(7,18){\line(1,-1){10}}
\put(38,18){\line(-1,-1){10}}
\put(22.5,28){\line(0,1){9}}
\put(22.5,17){\line(0,-1){9}}
\put(60,20){$\buildrel{\text{\larger[2] $C_v$}}\over{\hbox to 22pt{\leftarrowfill$\joinrel\mapstochar$}}$}
\end{picture}
\begin{picture}(170,40)
\put(35,42){$\Q_p(\sqrt{u},\sqrt{p})$}
\put(00,20){$\Q_p(\sqrt{u})$}
\put(40,20){$\Q_p(\sqrt{up})$}
\put(84,20){$\Q_p(\sqrt{p})$}
\put(53,00){$\Q_p$}
\put(32,27){\line(2,1){20}}
\put(83,27){\line(-2,1){20}}
\put(32,18){\line(2,-1){20}}
\put(83,18){\line(-2,-1){20}}
\put(57.5,28){\line(0,1){9}}
\put(57.5,16){\line(0,-1){8}}
\put(133,20){$\buildrel{\text{\larger[2] $\!\D_V$}}\over{\hbox to 22pt{$\mapstochar\joinrel{}\!$\rightarrowfill}}$}
\end{picture}
\begin{picture}(50,40)
\put(20,40){$d$}
\put(00,20){$\frac{d}{2}$}
\put(20.5,20){$1$}
\put(40,20){$1$}
\put(20.5,00){$1$}
\put(7,27){\line(1,1){10}}
\put(38,27){\line(-1,1){10}}
\put(7,18){\line(1,-1){10}}
\put(38,18){\line(-1,-1){10}}
\put(22.5,28){\line(0,1){9}}
\put(22.5,17){\line(0,-1){9}}
\end{picture}
$$
\endgroup
Here $\Q_p(\sqrt{u})$ is the quadratic unramified extension of $\Q_p$,
and $E$ has split multiplicative reduction precisely in those fields that
contain it; $d$ is the determinant of a fixed pairing on $V$, used in the
definition of $\D_V$. The group $\Cy_2\times\Cy_2$ has up to multiples
just one relation (see Example \ref{exc2c2}),
$$
  \Theta=\>\{1\} - \Cy_2^a -\Cy_2^b - \Cy_2^c + 2\>\Cy_2\times\Cy_2.
$$
The corresponding values of $C_v$ and $\D_V$ are
$$
 C_v(\Theta)=\frac{2n\cdot(1\text{ or }2)^2}{n\cdot 2\cdot 2}=2\cdot\square
 \quad\qquad
 \D_V(\Theta)=\frac{d\cdot 1^2}{\frac{d}{2}\cdot 1\cdot 1}=2\>,
$$
and so \smash{$\frac{C_v}{\D_V}$} is representation theoretic
(modulo squares!), by inspection.
\end{example}

This example explains our need to understand $G$-relations,
behaviour of functions and $\D_\rho$.
To establish the compatibility of local root numbers and Tamagawa
numbers in arbitrary extensions (Theorem \ref{loccompat}),
even for elliptic curves with non-split multiplicative reduction,
requires the full force of the machinery of \S\ref{smachinery}.

\subsection{Notation}
\label{ssnotation}

\noindent
For an abelian variety $A/K$ we use the following
notation:
\vskip 2pt

\noindent
\hskip-1mm\begin{tabular}{lll}
&$\Xp(A/F)$        & $\Hom_{\scriptscriptstyle{\Z_p}}\!(\varinjlim\Sel_{p^n}(A/F), \Q_p/\Z_p)\tensor\Q_p$, the dual $p^\infty$-Selmer. \cr
&$w(A/K)$          & local root number of $A/K$ for $K$ local, or\cr
&                  & global root number, $\prod_{v} w(A/K_v)$ for $K$ a number field.\cr
&$w(A/K,\!\tau)$     & (local/global) root number for the twist of $A$ by $\tau$, see \cite{RohG}.\cr
&$c_v(A/K)$        & local Tamagawa number of $A$ at a finite place $v$ of $K$\cr
&                  & (when $K$ is local, the subscript $v$ is purely decorational). \cr
&$C_v(A/K,\omega)$ & $c_v(A/K) \!\cdot\! |{\omega}/{\neron{}}|_K^{}$ for $K$ non-Archimedean, where $|\cdot|_K$ is \cr
&   & the normalised absolute value, and $\neron{}$ a N\'eron differential;\cr
&   & $\int_{A(K)} |\omega|$ for $K=\R$;\ \ $2\int_{A(K)} \omega\wedge \bar\omega$ for $K=\C$.\cr
&   & ($K$ local; $\omega$ is a non-zero regular exterior form on $A/K$.)\cr
&$C_{A/K}$          & $\prod_{v} C_v(A/K_v,\omega)$ for any global non-zero regular exterior\cr
&   & form $\omega$; independent of $\omega$ (product formula).\\[2pt]
\end{tabular}

Notation for representations $G\to\GL_n(\K)$, $\K$ a field of characteristic 0: 
$\blara_G$, $\blara$ usual inner product of two characters of representations;
$\K[G]$ regular representation;
$\K[G/H]$ permutation representation of $G$ on the left cosets of $H$;
$\triv$ trivial representation; $\tau^*$ contragredient \hbox{representation};~$\rho^H$~the $H$-invariants of $\rho$. We call $\rho$ self-dual if
$\rho\iso\rho^*$, equivalently $\rho\tensor\bar\K\iso\rho^*\tensor\bar\K$.

For a $\K$-vector space $V$ and a non-degenerate $\K$-bilinear pairing $\lara$ with values
in $\L\supset\K$, we write $\det(\lara|V)\in\L^{\times}/\K^{\times2}$ for
$\det(\langle e_i, e_j\rangle_{i,j})$ in any $\K$-basis $\{e_i\}$ of $V$.

\newpage
\noindent
For functions on subgroups of $G$ we use the following notation:
\smallskip

\noindent
\begin{tabular}{lll}
&$\Theta$     & a $G$-relation $\sum_i n_i H_i$ between permutation representations, \cr
&             & i.e. $\sum_i\C[G/H_i]^{\oplus n_i}=0$, see Definition \ref{grelation}.\cr
&$\RC_\Theta(\rho)$ & regulator constant for a $G$-representation $\rho$,
           see Definition \ref{defregcon}.\cr
&$\f(\Theta)$  & $\prod_i\f(H_i)^{n_i}$ for $\Theta=\sum_i n_i H_i$, see \S\ref{ssfunctions}.\cr
&$\f\sim\psi$ & equivalence relation $\f(\Theta)\!=\!\psi(\Theta)$ for all $G$-relations $\Theta$, see \S\ref{ssfunctions}. \cr
&$\D_\rho$        & $H\mapsto \det(\tfrac1{|H|}\lara|\rho^H)$, see Definition \ref{D_M}.\cr
&$\tGDI{\thincdots}$ & see Definitions \ref{gdidef}, \ref{gdidef2}.
\end{tabular}

\medskip\noindent
$\Di_{2n}$ denotes the dihedral group of order $2n$
(including $\Cy_2\times\Cy_2$ for $n=2$).
Conjugation of subgroups is usually written as a superscript, $H^x=xHx^{-1}$.

By a local field we mean a finite extension of $\Q_p,\R$ or $\F_p((t))$
(the latter will never occur). We write
$e_{M/L}, f_{M/L}$ for the ramification degree and the residue degree
of an extension $M/L$ of local fields,
$\mu_n$ for the set of $n^{\text{th}}$ roots of unity and
$\ord_p$ for the $p$-valuation of a rational or a $p$-adic number.

\vfill

\section{Functions on the Burnside ring} 
\label{smachinery}

This section is dedicated to 
relations between
permutation representations, behaviour of functions on the Burnside ring
with respect to such relations, the issue whether a function is
representation-theoretic, and regulator constants.
As explained in the introduction, the applications we have in mind
relate to elliptic curves and abelian varieties.
On the other hand, the results are
self-contained, purely group-theoretic in nature, and they
may be of independent interest.

Throughout the section $G$ is an abstract finite group.

\vspace{-13pt}

\subsection{Relations between permutation representations}\rlap{\hskip 50cm.}
\par\noindent
\label{ssrelations}

\noindent
Let $G$ be a finite group and $\cH$ the set of
subgroups of $G$ up to conjugacy.
By abuse of notation, for a subgroup $H\< G$ we also write $H$ for its
class in $\cH$.
The {\em Burnside ring\/} of $G$ is the free abelian group $\Z\cH$
(we will not use its multiplicative structure).
The elements of $\cH$ are in one-to-one correspondence
with transitive $G$-sets via $H\mapsto G/H$.
This extends to a correspondence \hbox{between} elements of $\Z\cH$
with non-negative coefficients and finite $G$-sets, under which addition
translates to disjoint union.

The map $H\mapsto\C[G/H] (\,\iso \Ind_H^G\vtriv{H})$ defines a ring
homomorphism from the Burnside ring to the representation ring of $G$.
On the level of $G$-sets, this map is simply $X\mapsto\C[X]$.
Here in \S\ref{ssrelations} we consider its kernel:

\begin{definition}
\label{grelation}
We call an element of
the Burnside ring of $G$
$$
  \Theta = \sum\nolimits_i n_i H_i \qquad (n_i\in\Z,\> H_i\< G)
$$
a {\em relation between permutation representations of $G$\/}
or simply a {\em $G$-relation\/} if
$\oplus_i\C[G/H_i]^{\oplus n_i}=0$ as a virtual representation,
i.e. the character
$\sum_i n_i \chi_{\scriptscriptstyle \C[G/H_i]}$ is zero.
In other words, if $\Theta$ corresponds to a formal difference of
two $G$-sets, we require that they have isomorphic permutation
representations.
\end{definition}

\newpage

\begin{exercise}
A cyclic group $\Cy_n$ has no non-trivial relations.
\end{exercise}

\begin{example}
\label{exc2c2}
The group $G\!=\!\Cy_2\times \Cy_2$ has
five subgroups $\{1\}$, $\Cy_2^a$, $\Cy_2^b$, $\Cy_2^c$, $G$ and four
irreducible representations $\triv, \epsilon^a, \epsilon^b, \epsilon^c$.
Writing out the permutation representations, 
\begingroup
  \smaller[1]
  \let\oldoplus\oplus
  \def\oplus{\!\oldoplus\!}
  \def\iso{\text{\normalsize $\scriptstyle\cong$}}
  $$
    \C[G] \iso \triv\oplus\epsilon^a\oplus\epsilon^b\oplus\epsilon^c, \>\>\>\>
    \C[G/\Cy_2^a] \iso \triv\oplus\epsilon^a, \>\>\>\>
    \C[G/\Cy_2^b] \iso \triv\oplus\epsilon^b, \>\>\>\>
    \C[G/\Cy_2^c] \iso \triv\oplus\epsilon^c, \>\>\>\>
    \triv \iso \triv,
  $$
\endgroup
we see that, up to multiples, there is a unique $G$-relation
$$
  \Theta=\>\{1\} - \Cy_2^a -\Cy_2^b - \Cy_2^c + 2\>G.
$$
\end{example}

\begin{example}
\label{exdih}
Generally, any dihedral group
$G=\Di_{2n}$ with presentation $\langle g,h|h^n\!=\!g^2\!=\!(gh)^2\!=\!1\rangle$
has the relation
$$
  \>\{1\} - \>\langle g\rangle - \>\langle gh\rangle - \langle h\rangle + 2\>G.
$$
When $n$ is odd it can be written as
$\{1\} - 2\>\Cy_2 - \Cy_n + 2\>\Di_{2n}$, and it is
unique up to multiples when $n$ is prime. For $\Di_8$ and $\Di_{12}$,
together with
$$
\begin{array}{ll}
  \Di_8
  \left\{
  \!\!\!
  \begin{array}{l}
  \langle g\rangle \!-\!\langle gh\rangle \!-\!\langle g, h^2\rangle \!+\!\langle gh,h^2\rangle\cr
  \{1\}\!-\!\langle h^2\rangle\!-\!2\langle g\rangle \!+\!2\langle gh,h^2\rangle
  \end{array}
  \right.
  \quad
  \Di_{12}
  \left\{
  \!\!\!
  \begin{array}{l}
  \langle h^3\rangle  \!-\! \langle h\rangle  \!-\!2 \langle g,h^3\rangle   \!+\! 2G\cr
  \langle h^2\rangle \!-\!\langle h\rangle \!-\!\langle gh,h^2\rangle \!-\!\langle g,h^2\rangle \!+\!2G\cr
  \langle g\rangle \!-\!\langle gh\rangle \!+\!\langle gh,h^2\rangle \!-\!\langle g,h^2\rangle
  \end{array}
  \right.
\end{array}
$$
it forms a $\Z$-basis of all $G$-relations (cf. Table \ref{dihtable}
for the lattice of these subgroups.)
\end{example}

\begin{remark}
The number of irreducible $\Q G$-representations coincides with the number
of cyclic subgroups of $G$ up to conjugacy (\cite{SerLi} \S13.1, Cor.~1).
If~$G$ is not cyclic, this is clearly less
than the number of all subgroups up to conjugacy, so the map
$\sum n_i H_i\mapsto \oplus\C[G/H_i]^{\oplus n_i}$
must have a kernel. Hence every non-cyclic group has non-trivial relations.
\end{remark}

\begin{example}[Artin formalism]
\label{artfor}
Let $F/K$ be a Galois extension of number fields with Galois group $G$.
For $H\< G$, the Dedekind $\zeta$-function $\zeta_{F^H}(s)$ agrees with the
$L$-function over $K$ of the Artin representation $\C[G/H]$.
So a $G$-relation $\sum_i n_i H_i$ yields the identity 
$$
  \prod\limits_i \zeta_{F^{H_i}}(s)^{n_i}=1.
$$
Similarly, if $E/K$ is an elliptic curve (or an abelian variety),
$$
  \prod\limits_i L(E/F^{H_i},s)^{n_i}=1.
$$
\end{example}

\begin{notation}
For $D\<G$, define a map $\Res_D$ from the Burnside ring of $G$
to that of $D$, and a map $\Ind_D^G$ in the opposite direction by
$$
\Res_D H =\sum\limits_%
{{\rlap{\raise2pt\hbox{$\scriptstyle\!\!\! x\in H\!\backslash\! G\!/\!D$}}}}
\>\>
D \cap H^{x^{-1}}
 \qquad\text{and}\qquad \Ind_D^G H = H\>.
$$
On the level of representations (i.e. under $H\mapsto \C[G/H]$), these
are the usual restriction and induction.
On the level of $G$-sets, $\Res_D$ simply restricts the action
from $G$ to $D$ (Mackey's decomposition).
\end{notation}

\newpage

\begin{theorem}
\label{l:}
Suppose $D, H_i\< G$, $N\triangleleft G$.
\begin{enumerate}
\item\label{l:sum}
The sum and the difference of two $G$-relations is a $G$-relation.
\item\label{l:sat}
If $\Theta=\sum_i n_i\, H_i$ and $m\Theta$ is a $G$-relation, then
$\Theta$ is a $G$-relation.
\item\label{l:lif} (lifting)
If $H_i\supset N$ and
$\sum_i n_i\, H_iN/N$ is a $G/N$-relation, then
$\sum_i n_i H_i$ is a $G$-relation.
\item\label{l:ind} (induction)
Any $D$-relation is also a $G$-relation;
i.e. if\,\, $\Theta=\sum_i n_i\, H_i$ is a $D$-relation, then
$\Ind_D^G\Theta=\sum_i n_i\, H_i$ is a $G$-relation.
\item\label{l:pro} (projection)
If $\sum_i n_i H_i$ is a $G$-relation, then
$\sum_i n_i H_iN/N$ is a \linebreak\hbox{$G/N$-relation.}
\item\label{l:res} (restriction)
\begingroup
\baselineskip 100pt
If $\Theta=\sum_i n_i H_i$ is a $G$-relation, then its restriction
\hbox{$\Res_D\Theta=\sum_i n_i\>\> \sum\limits_%
{{\rlap{\raise2pt\hbox{\tiny$\scriptstyle\!\!\! x\in H_{\!i}\!\backslash\! G\!/\!D$}}}}
\>\>
D \cap H_i^{x^{-1}}$} is
\smash{\raise0pt\hbox{a $D$-relation.}}
\endgroup
\end{enumerate}
\end{theorem}

\begin{proof}
\eqref{l:sum},\eqref{l:sat},\eqref{l:lif} Clear.
\eqref{l:ind}
Induction is transitive.
\eqref{l:pro} 
This follows from the fact that $\C[G/H_i]^N\iso \C[G/NH_i]$
as a $G$-representation.
(The invariants $\C[G/H]^N$ come from orbits of $N$ on $G/H$,
so this space has a basis $\{\sum_{x\in \Delta }xH\}_\Delta $ with $\Delta $ ranging over
the double cosets $N\backslash G/H$ (=$G/NH$).
As $N$ is normal, $G$ permutes the basis elements,
and this is the same as the action on $G/NH$.)
\eqref{l:res} 
This is a consequence of Mackey's formula,
$\Res_D\C[G/H] \iso \oplus_{\scriptscriptstyle x\in H\!\backslash\! G\!/\!D}
   \C[D/D \cap H^{x^{-1}}]$. 
\end{proof}

Properties (\ref{l:lif}) and (\ref{l:ind}) allow one to
lift relations from quotient groups and induce them from subgroups.
This is not to
suggest that relations can always be built up like that, for instance
dihedral groups have relations while cyclic groups do not.
Here is a case when this does work:

\begin{lemma}
\label{loffchop}
Let $D\normal G$, and suppose that $G$ acts on the Burnside ring of $D$
by conjugation through a quotient of order $n$.
\begin{enumerate}
\item
If $\Theta=\sum_i n_i H_i$ is a $G$-relation with $H_i\subset D$,
then $n\Theta$ is induced from a $D$-relation.
\item
Suppose that $N\normal G$ with $N\subset D$, and that
each subgroup of $G$ either contains $N$ or is contained in $D$.
Then for every $G$-relation $\Theta$, $n\Theta$ is a sum of
a relation induced from $D$ and one lifted from $G/N$.
\end{enumerate}
\end{lemma}

\begin{proof}
(1) Let $G_0$ be the kernel of the action of $G$ on the Burnside ring of~$D$.
As a $G$-relation, we may write $n\Theta$ as
$$
  n\Theta = \sum_i n_i \Bigl(\sum_{g\in G/G_0} gH_ig^{-1}\Bigr).
$$
We claim that in this form it is a $D$-relation.
Indeed, restricting it to $D$,
on the one hand, yields a $D$-relation (Theorem \ref{l:}\eqref{l:res})
and, on the other hand,
multiplies the expression by $[G:D]$.
Hence the expression itself is a \hbox{$D$-relation}.

(2)
If $\Theta$ is a $G$-relation, write it as
$\sum_i n_i H_i + \sum_j n'_j H'_j$ with $H_i\supset N$ and $H'_j\subset D$.
Then
$$
  \Theta = (\sum\nolimits_i n_i H_i + \sum\nolimits_j n'_j NH'_j) + (\sum\nolimits_j n'_j H'_j - \sum\nolimits_j n'_j NH'_j).
$$
The first term is a relation lifted from $G/N$ (Theorem \ref{l:}\eqref{l:pro}),
and the second term is therefore a $G$-relation with
constituents in $D$, so (1) applies.
\end{proof}

\begin{example}
\label{exoff1}
Suppose $G=\Cy_{u2^m}\times \Cy_{2^k}$ with $u$ odd and $k>m$. Set
$$
  G_1=\{1\}\!\times\! \Cy_{2^{k-m}}  \>\>\subset\>\>
  G_2=\Cy_{u2^m}\!\times\! \Cy_{2^{k-1}}  \>\>\subset\>\> G,
  \qquad G/G_1\iso \Cy_{u2^m}\!\times\! \Cy_{2^{m}}.
$$
Every element outside $G_2$ generates a subgroup containing $G_1$,
so every subgroup not in $G_2$ contains $G_1$.
Since every subgroup of $G$ is normal,
Lemma \ref{loffchop}(2) shows that every $G$-relation
is a sum of a relation coming from $G_2$ and one lifted from $G/G_1$.
By induction, the lattice of $G$-relations is generated by ones coming
from subquotients $\Cy_{u2^m}\!\times\!\Cy_{2^t}\,/\,\{1\}\!\times\!\Cy_{2^{t-m}}$
for $m\le t\le k$, all isomorphic to $\Cy_u\!\times\!\Cy_{2^m}\!\times\!\Cy_{2^m}$.
\end{example}

\begin{example}
\label{exoff2}
Suppose $G=\langle x,y\,|\,x^n\!=\!y^{2^k}\!=\!1, yxy^{-1}\!=\!x^{-1} \rangle$,
a semi-direct product of $\Cy_{2^k}$ by $\Cy_n$
for some $k,n\ge 1$. Consider
$$
  G_1=\langle y^2\rangle    \subset
  G_2=\langle x,y^2\rangle  \subset G, \qquad G/G_1\iso \Di_{2n},\>\, G/G_2\iso \Cy_2.
$$
Note that if $x^ay^b\in H\< G$ with $b$ odd, then $(x^ay^b)^2=y^{2b}\in H$
implies that $H\supset G_1$. Equivalently, every
subgroup not contained in $G_2$ contains $G_1$.
By Lemma \ref{loffchop}(2), if $\Theta$ is any $G$-relation, then
$2\Theta$ is a sum of a relation induced from $G_2$ and one lifted from $G/G_1$.
If $4\nmid n$, it is easy to verify that every subgroup of $G_2$ is normal
in $G$, so $\Theta$ itself is already of this form.

Observe that $G_2\iso \Cy_n\!\times\!\Cy_{2^{k-1}}$, whose relations
were discussed in the previous example.
\end{example}

\subsection{Regulator constants}
\label{s:regconst}

Let $\K$ be a field of characteristic 0.
In this section we define regulator constants for self-dual
$\K G$-representations, first introduced in \cite{Squarity} for $\K=\Q$.
(The name ``regulator constant'' comes from regulators of elliptic curves;
see \S\ref{sssquality}.)

\begin{notation}
Suppose $V$ is a $\K$-vector space with a non-degenerate \hbox{$\K$-bilinear}
pairing $\lara$ that takes values in some extension $\L$ of $\K$.
We write $\det(\lara|V)\in\L^{\times}/\K^{\times2}$ for
$\det(\langle e_i, e_j\rangle_{i,j})$ in any $\K$-basis $\{e_i\}$ of $V$.

\end{notation}

\begin{definition}
\label{defregcon}
Let $G$ be a finite group, $\rho$ a self-dual $\K G$-representation,
and $\Theta\!=\!\sum_i n_i H_i$ a $G$-relation.
Pick a $G$-invariant non-degenerate $\K$-bilinear pairing $\lara$ on $\rho$
with values in some extension $\L$ of $\K$,
and define the {\em regulator constant\/}
$$
  \RC_\Theta(\rho)=
  \RC^\K_\Theta(\rho)={\prod_i \det(\tfrac{1}{|H_i|}\lara|\rho^{H_i})^{n_i}}
    \in \K^{\times}/\K^{\times2}.
$$
(This is well-defined, non-zero and independent of $\lara$
by Lemma \ref{invnondeg} and
Theorem \ref{regconst}.
It follows that it lies in $\K^{\times}/\K^{\times2}$ rather than $\L^{\times}/\K^{\times2}$
as the pairing can be chosen to be $\K$-valued.)
\end{definition}

\begin{exercise}
\label{c2c2regconst}
Let $G\!=\!\Cy_2\!\times\!\Cy_2$ and
$\Theta\!=\!\{1\} \!-\! \Cy_2^a \!-\!\Cy_2^b \!-\! \Cy_2^c \!+\! 2\>G$
from \hbox{Example~\ref{exc2c2}}. Then $\RC_\Theta(\chi)\!=\!2$ for all
four 1-dimensional characters~$\chi$~of~$G$.
\end{exercise}

\begin{lemma}
\label{invnondeg}
Suppose $\rho $ is a $\K G$-representation, and $\lara$ a
$G$-invariant $\K$-bilinear non-degenerate pairing on $\rho $.
For every $H\< G$, the restriction of $\lara$ to $\rho ^H$ is non-degenerate.
In other words, $\det(\lara|\rho^H)\ne 0$.
\end{lemma}

\begin{proof}
Consider the projection $P: \rho \to \rho ^H$ given by
$v\mapsto\frac{1}{|H|}\sum_{h\in H} h\cdot v$. Then $\rho =\rho ^H\oplus\ker P$, and
for $v\in \rho ^H, w\in\ker P$,
$$
  \langle v,w \rangle = \tfrac{1}{|H|} \sum_{h\in H} \langle hv,hw \rangle
    = \langle v, P(w) \rangle = 0.
$$
So $\rho ^H$ and $\ker P$ are orthogonal to each other, and the pairing cannot be
degenerate on either of them.
\end{proof}

\begin{lemma}
\label{dimsaddup}
Let $\Theta=\sum_i n_i H_i$ be a $G$-relation and
$\rho$ a $\K G$-representation. Then
$$
  \sum\nolimits_i n_i\dim\rho^{H_i}=0.
$$
\end{lemma}

\begin{proof}
By Frobenius reciprocity,
\smallskip

$
\qquad  \sum n_i\dim\rho^{H_i}=\sum n_i\>\blangle\!\Res_{H_i}\rho,\triv_{H_i}\brangle_{H_i}
  =\sum n_i\blangle\rho,\Ind_{H_i}^{G}\triv_{H_i}\brangle_G
$

$
\qquad \phantom{\sum n_i\dim\rho^{H_i}}
  = \blangle\rho,\oplus(\Ind_{H_i}^{G}\triv_{H_i})^{\oplus n_i}\brangle_G
  = \blangle\rho,0\brangle_G = 0.
$
\end{proof}

\smallskip
\noindent We now prove that regulator constants are independent of the pairing:

\begin{theorem}
\label{regconst}
Let $\Theta=\sum_i n_i H_i$ be a $G$-relation,
$\rho$ a self-dual $\K G$-repre\-sentation,
and $\lara_1, \lara_2$
two non-degenerate $G$-invariant $\K$-bilinear pairings on $\rho$.
Computing the determinants with respect to the same bases of $\rho^{H_i}$
on both sides,
$$
  \prod_i\det(\tfrac{1}{|H_i|}\lara_1|\rho^{H_i})^{n_i}
      =
  \prod_i\det(\tfrac{1}{|H_i|}\lara_2|\rho^{H_i})^{n_i}.
$$
(This is an actual equality, not modulo $\K^{\times 2}$.)
\end{theorem}

\begin{proof}
We may assume $\K$ is algebraically closed.
It is enough to prove the statement for a particular choice
of bases of $\rho^{H_i}$,
as seen from the transformation rule $X\mapsto M^t X M$ for matrices of
bilinear forms under change~of~basis.

If $\rho=\alpha\oplus\beta$ with $\alpha,\beta$ self-dual and
$\Hom_G(\alpha,\beta^*)=0$, then
$\langle a,b\rangle_1=0$ 
for $a\in\alpha$ and $b\in\beta$, and similarly for $\lara_2$.
Since $\rho^H=\alpha^H\oplus\beta^H$,
choosing bases that respect the decomposition reduces the problem to
$\alpha$ and $\beta$ separately. Thus, we may assume that
either $\rho=\tau^{\oplus n}$ with $\tau$ irreducible and self-dual,
or $\rho=\sigma^{\oplus n}\oplus(\sigma^*)^{\oplus n}$
with $\sigma$ irreducible and not self-dual.

In the first case, for each $H_i$ fix a basis of $\tau^{H_i}$ and take the
induced bases of $(\tau^{H_i})^{\oplus n}=\rho^{H_i}$. Let $\lara_\tau$ be
a non-degenerate $G$-invariant bilinear pairing on $\tau$, and let $M_i$
be its matrix on the chosen basis of $\tau^{H_i}$. As $\lara_\tau$ is unique
up to scalar ($\K$ is algebraically closed), the matrix of $\lara_1$ on $\rho^{H_i}$ is
$$
  T(\Lambda,M_i)=
  \begin{pmatrix}
     \lambda_{11} M_i & \lambda_{12} M_i & \ldots & \lambda_{1n} M_i \cr
     \lambda_{21} M_i & \lambda_{22} M_i & \ldots & \lambda_{2n} M_i \cr
     \vdots & \vdots & \ddots & \vdots \cr
     \lambda_{n1} M_i & \lambda_{n2} M_i & \ldots & \lambda_{nn} M_i \cr
  \end{pmatrix} \>,
$$
for some $n\!\times\!n$ matrix $\Lambda=(\lambda_{xy})$ not depending on $H_i$.
Hence
$$
  \det (\tfrac{1}{|H_i|}\lara_1|\rho^{H_i}) =
    (\det\Lambda)^{\dim \tau^{H_i}} (\det \tfrac{1}{|H_i|} M_i)^{n}\>.
$$
The dimensions $\dim \tau^{H_i}$ cancel in $\Theta$ by Lemma \ref{dimsaddup},
so $\prod_i\det (\tfrac{1}{|H_i|}\lara_1|\rho^{H_i})^{n_i}$
does not depend on $\Lambda$, and takes therefore the same value for
$\lara_2$.

The argument in the second case is similar. The matrix of $\lara_1$ on
$\rho^{H_i}$ is of the form
$\smallmatrix{0}{T(\Lambda,M_i)}{T(\Lambda',M'_i)}{0}$ where $M_i$ and $M'_i$
are the matrices of a fixed $G$-invariant non-degenerate
pairing $\lara_\sigma: \sigma\times\sigma^*\to\K$ and its transpose. Again
the contributions $((-1)^n\det\Lambda\det\Lambda')^{\dim\tau^{H_i}}$ cancel in
$\Theta$ and the result follows.
\end{proof}

\begin{corollary}
\label{RCproperties}
Regulator constants are multiplicative in $\Theta$ and $\rho$,
\beq
  \RC_{\Theta_1+\Theta_2}(\rho) = \RC_{\Theta_1}(\rho) \RC_{\Theta_2}(\rho), \cr
  \RC_\Theta(\rho_1\oplus\rho_2) = \RC_\Theta(\rho_1) \RC_\Theta(\rho_2).
\eeq
If $\L\supset\K$, then $\RC^\K_\Theta(\rho)=\RC^\L_\Theta(\rho\tensor\L)$ in $\L^{\times}/\L^{\times2}$.
\end{corollary}

\begin{example}
\label{regperm}
Suppose $\rho=\K[G/D]$ for some subgroup $D$ of $G$.
Take the standard pairing on $\rho$, making the elements of $G/D$
an orthonormal basis.
The space of invariants $\rho^H$ has a basis consisting
of $H$-orbit sums of these basis vectors.
Since $\det(\lara|\rho^H)$ is the product of lengths of these orbits,
$$
  \det(\tfrac1{|H|}\lara|\rho^H)
     =
  \prod_{w\in H\backslash G/D} \tfrac{1}{|H|} \tfrac{|HwD|}{|D|}
     =
  \prod_{w\in H\backslash G/D} \tfrac{1}{|H\cap D^w|} \>,
$$
which yields an elementary formula for the regulator constants of $\rho$.
Note that for many groups, every $\K G$-representation
is a $\Z$-linear combinations of such $\rho$, e.g.
dihedral groups $\Di_{2p^n}$ when $\K=\Q$ or $\Q_p$, or symmetric groups.
\end{example}

\begin{example}
\label{exdihreg}
For an odd prime $p$, the dihedral group $G=\Di_{2p}$
has the relation (cf. Example \ref{exdih})
$$
  \Theta=\{1\} - \>2\Cy_2 - \Cy_p + 2\>G.
$$
For $\K=\Q$ or $\Q_p$, the irreducible $\K G$-representations are
$\triv$, sign $\epsilon$ and $(p\!-\!1)$-dimensional $\rho$.
Writing them as combinations of permutation representations
$\K[G/H]=\Ind_H^G\triv_H$,
we can compute their regulator constants as in Example~\ref{regperm}:
modulo squares,
\beq
  \RC_\Theta(\triv) &=&
  \RC_\Theta(\Ind_G^G\triv) &=&
  (\tfrac{1}{1})^1 (\tfrac{1}{2})^{-2} (\tfrac{1}{p})^{-1} (\tfrac{1}{2p})^2 &
  \equiv p
  \cr
  \RC_\Theta(\triv)\RC_\Theta(\epsilon) &=&
  \RC_\Theta(\Ind_{\Cy_p}^G\triv) &=&
  (\tfrac{1}{1^2})^1 (\tfrac{1}{1})^{-2} (\tfrac{1}{p^2})^{-1} (\tfrac{1}{p})^2 &
  \equiv 1
  \cr
  \RC_\Theta(\triv)\RC_\Theta(\rho) &=&
  \RC_\Theta(\Ind_{\Cy_2}^G\triv) &=&
  (\tfrac{1}{1^p})^1 (\tfrac{1}{2\cdot 1^{(p-1)/2}})^{-2} (\tfrac{1}{1})^{-1} (\tfrac{1}{2})^2 &
  \equiv 1.
  \cr
\eeq
So $\RC_\Theta(\triv)=\RC_\Theta(\epsilon)=\RC_\Theta(\rho)=p$.
\end{example}

\begin{example}[$\Di_{2p^n}$, $p$ odd, $\K=\Q_p$]
\label{exdihreg2}
Generally, suppose $G=\Di_{2p^n}$ with $p$ odd,
and consider the $G$-relations coming from various $\Di_{2p}$ subquotients,
\beq
  \Theta_{n+1-k}=\Cy_{p^{k-1}} - 2\>\Di_{2p^{k-1}} - \Cy_{p^{k}} + 2\>\Di_{2p^{k}} &&
    (1\le k\le n).\cr
\eeq
The irreducible $\Q_pG$-representations are $\triv$,
sign $\epsilon$ and $\rho_k$ of dimension \hbox{$p^k\!-\!p^{k-1}$} for
$1\le k\le n$. A computation as in Example \ref{exdihreg} shows that
$$
  \RC_{\Theta_k}(\triv)=\RC_{\Theta_k}(\epsilon)=\RC_{\Theta_k}(\rho_k)=p,
  \qquad
  \RC_{\Theta_k}(\rho_j)=1\text{\ for $j \ne k$}.
$$
\end{example}

\begin{example}[$\Di_{2p^n}$, $p\!=\!2$, $\K\!=\!\Q_2$]
\label{exdihreg3}
For $G\!=\!\Di_{2^{n+1}}$ consider the \hbox{$G$-relations}
\beq
  \Theta_1 &=& \Cy_{2^{n-1}} - \Di_{2^n}^a - \Di_{2^n}^b
     - \Cy_{2^n} + 2\>G\cr
  \Theta_{n+1-k} &=& \Di_{2^k}^a - \Di_{2^k}^b - \Di_{2^{k+1}}^a + \Di_{2^{k+1}}^b &&
    (1\le k< n).\cr
\eeq
Here
$\Cy_{2^k}\!=\!\langle h^{2^{n-k}} \rangle$,
$\Di_{2^k}^a\!=\!\langle h^{2^{n-k}}, g \rangle$,
$\Di_{2^k}^b\!=\!\langle h^{2^{n-k}}, gh \rangle$ in terms of the generators
given in Example \ref{exdih}. In this case,
the irreducible $\Q_2G$-representations are $\triv$ (trivial),
$\rho_k$ of dimension $2^k\!-\!2^{k-1}$ for $2\le k\le n$, and
one-dimensional characters $\epsilon, \epsilon^a, \epsilon^b$ that factor
through $G/\Cy_{2^n}$, $G/\Di_{2^{n-1}}^a$ and $G/\Di_{2^{n-1}}^b$ respectively.
The regulator constants are
\beq
  \RC_{\Theta_1}(\triv)=\RC_{\Theta_1}(\epsilon)=
    \RC_{\Theta_1}(\epsilon^a)=\RC_{\Theta_1}(\epsilon^b)&=2,\cr
  \RC_{\Theta_k}(\epsilon^a)=\RC_{\Theta_k}(\epsilon^b)=
    \RC_{\Theta_k}(\rho_k)&=2 && (k>1),
\eeq
and trivial on other irreducibles.
\end{example}

\begin{example}
\label{exsl2f3}
Let $G=\SL_2(\F_3)$, which is the semi-direct product of $\Cy_3$
by the quaternion group $\Qu_8$. Denote its complex irreducible
representations by $\triv, \chi, \bar\chi$ (1-dim.),
$\tau, \chi\tau, \bar\chi\tau$ ($\tau$ symplectic 2-dim.)
and $\rho$ (3-dim.). A basis of $G$-relations and their
regulator constants for the $\Q G$-irreducibles are
$$
\begin{array}{l|ccccccc}
            & \triv & \chi\oplus\bar\chi & \rho & \tau^{\oplus 2} & \chi\tau \oplus \bar\chi\tau\cr
\hline
\Cy_4-\Cy_6-\Qu_8+G & 2 & 1 & 2 & 1 & 1 \cr
\Cy_2-3\Cy_4+2\Qu_8 & 2 & 1 & 2 & 1 & 1 \cr
\end{array}
$$
The table stays the same if $\Q$ is replaced by any $\K$
of characteristic 0, except that $\tau$ may become realisable over $\K$,
in which case $\RC^\K_\Theta(\tau)$ and not just $\RC^\K_\Theta(\tau^{\oplus 2})$
makes sense. This regulator constant will still be 1
by Corollary \ref{triviality} below, because $\tau$ is symplectic.
Observe that in the table the representations $V\oplus V^*$ also have
trivial regulator constants, which is true for all groups by the same
corollary.
\end{example}

Let us record a number of situations when the regulator constants
are trivial; other properties are discussed in \S\ref{svanishing}.
The following result, or rather Corollary \ref{triviality},
was motivated by the behaviour of root numbers of elliptic curves
(see Proposition \ref{genparity}).

\begin{theorem}
\label{vanishing}
Suppose $\rho$ is a self-dual $\K G$-representation such that
$\rho\tensor_\K\bar\K$ admits a
non-degenerate alternating $G$-invariant pairing.
Then $\RC_\Theta(\rho)=1$ for every $G$-relation $\Theta$.
\end{theorem}

\begin{proof}
Since $\dim\Hom_G(\triv,\rho\wedge\rho)$ is
the same over $\K$ and $\bar\K$,
there is also a non-degenerate alternating $G$-invariant pairing
$\lara$ on $\rho$ itself.
By Lemma \ref{invnondeg}, its restriction
to $\rho^H$ is non-degenerate (and alternating)
for every subgroup $H$ of $G$. In an appropriate basis for $\rho^H$
this pairing is given by a matrix $\smallmatrix0A{-A^t}0$,
so $\det(\tfrac{1}{|H|}\lara|\rho^H)$ is a square in $\K$.
\end{proof}

\newpage

\begin{corollary}
\label{triviality}
Let $\rho$ be a self-dual $\K G$-representation. Suppose either
\begin{enumerate}
\item $\rho$ is symplectic as a $\bar\K$-representation, or
\item $\rho\tensor\bar \K\iso\tau\oplus\tau$ for some $\bar \K G$-representation $\tau$, or
\item all $\bar\K$-irreducible constituents of $\rho\tensor\bar \K$ are not self-dual.
\end{enumerate}
Then $\RC_\Theta(\rho)=1$ for every $G$-relation $\Theta$.
\end{corollary}

\begin{proof}
It suffices to check that in each case $\rho\tensor_\K\bar\K$ carries a
non-degenerate alternating $G$-invariant pairing. This holds by definition
in case (1). In cases (2) and (3), $\rho\tensor_\K\bar\K$ is of the
form $V\oplus V^*$. Writing $P$ for the matrix of the canonical
map $V\times V^*\to\bar\K$, the pairing $\smallmatrix0P{-P^t}0$
has the required properties.
\end{proof}

\begin{lemma}
\label{mustenter}
Let $\rho$ be a self-dual $\K G$-representation.
Suppose $\Theta=\sum n_i H_i$ is a $G$-relation such that no $\bar\K$-irreducible
constituent of $\rho$ occurs in any of the $\K[G/H_i]$.
Then $\RC_\Theta(\rho)=1$.
\end{lemma}

\begin{proof}
By Frobenius reciprocity,
\hbox{$\dim\rho^{H_i}\!=\!\blangle\triv,\Res_{H_i}\rho\brangle
\!=\!\blangle\K[G/H_i],\rho\brangle\!=\!0$.}
\end{proof}

\begin{remark}
Suppose $R$ is a principal ideal domain whose
field of fractions $\K$ has characteristic coprime to $|G|$.
If $\rho$ is a free $R$-module of finite rank with a $G$-action,
then $\RC_\Theta(\rho)$ may be defined in the same way,
$$
  \RC_\Theta(\rho)=
  {\prod\nolimits_i \det(\tfrac{1}{|H_i|}\lara|\rho^{H_i})^{n_i}}
    \in \K^{\times}/ R^{\times2}\>,
$$
where the determinants are now computed on $R$-bases of $\rho^{H_i}$.
The pairing $\lara$ may take values in any extension of $\K$ as before,
and the class of $\RC_\Theta(\rho)$ in $\K^\times/R^{\times 2}$ is
independent of $\lara$.

For instance, when $R\!=\!\Z$ the group $R^{\times 2}$ is trivial, so
$\RC_{\Theta}$ associates a well-defined rational number to every $\Z G$-lattice.
Also, if $R$ is a discrete valuation ring with 
maximal ideal $\m$ and residue field $R/\m=k$ with $\vchar k\nmid |G|$,
it is not difficult to see that $\RC_\Theta(\rho\tensor\K)$ is in
$R^\times/R^{\times2}$, and
$$
  \RC_\Theta(\rho\tensor\K) = \RC_\Theta(\rho\tensor k) \mod \m.
$$
As every $\K G$-representation admits a $G$-invariant $R$-lattice, we deduce
\end{remark}

\begin{corollary}
\label{mustdivide}
Let $\K\!=\!\Q$ or $\K\!=\!\Q_p$, and let $\rho$ be a $\K G$-representation.
If $p\nmid|G|$, then
$\ord_p\RC_\Theta(\rho)$ is even for every
$G$-relation $\Theta$.
\end{corollary}

\subsection{Functions modulo $G$-relations}
\label{ssfunctions}

We now turn to linear functions
$$
  \f:\>\>
  \text{Burnside ring of $G$}
  \quad\hbox to 3em{\rightarrowfill}\>\>
  \begingroup
  \raise-5pt
  \hbox{\text{\def\arraystretch{0.8}\begin{tabular}{c}abelian group\cr\smaller[3] (written multiplicatively)\end{tabular}}}
  \endgroup
$$
or, equivalently, functions on $G$-sets
that satisfy $\f(X\,\amalg\, Y)\!=\!\f(X)\f(Y)$.
Our main concern is the distinction
between functions that are representation-theoretic
(i.e. only depend on $\C[X]$) and those that are not.
We say that
\begin{itemize}
\item
$\f$ is {\em trivial\/} on an element $\Psi$ of the Burnside ring of $G$
     if $\f(\Psi)=1$.
\item
$\f\sim \f'$ if $\f/\f'$ is trivial on all $G$-relations.
\end{itemize}
So, $\f$ is representation-theoretic in the sense of \S\ref{ssGG}
if and only if $\f\sim 1$.

\newpage

\begin{exercise}
\label{exHHfun}
For a constant $\lambda$, the function
$H\mapsto \lambda^{[G:H]}\>(=\lambda^{\dim\C[G/H]})$~to $\R^{\times}$ is trivial on $G$-relations.
On the other hand, $H\mapsto |H|$ in general is not.
\end{exercise}

\begin{example}
\label{constsim1}
The constant function $\f: H\mapsto\lambda$ is trivial on $G$-relations:
$$
  \f\bigl(\sum n_i H_i\bigr) = \prod \lambda^{n_i} =
    \prod \lambda^{n_i\blangle \triv, \C[G/H_i] \brangle}
    = \lambda^{\blangle \triv, \bigoplus \C[G/H_i]^{\oplus n_i} \brangle}
    = \lambda^0 = 1.
$$
\end{example}

\begin{example}
A cyclic group has no relations, so $\f\sim 1$ for every $\f$.
\end{example}

\begin{example}
\label{exellfun}
If $E/K$ is an elliptic curve and $G=\Gal(F/K)$, then
$$
  {\mathbf L}: H \longmapsto L(E/F^H,s)
$$
is a function with values in the multiplicative group of meromorphic
functions on $\Re s\!>\!\tfrac 32$. By Artin formalism,
${\mathbf L}\sim 1$ (Example \ref{artfor}).
As explained in \S\ref{sssquality}, the Birch--Swinnerton-Dyer conjecture
implies that
$$
C:\>   H \longmapsto C_{E/F^H}    \quad \text{ and } \quad
\Reg:\>  H \longmapsto \Reg_{E/F^H}
$$
satisfy $C\cdot\Reg\sim 1$ as functions to $\R^{\times}/\Q^{\times2}$.
\end{example}

\begin{definition}
\label{gdidef}
If $D\< G$, we say a linear function on the Burnside ring of $G$ is {\em local\/}
(or {\em $D$-local}) if its value on any $G$-set only depends on the $D$-set
structure.
Since $G/H=\coprod_{\scriptscriptstyle x\in H\!\backslash\! G\!/\!D} D/(H^{x^{-1}}\!\cap D)$ as a $D$-set,
this is equivalent to the following:
there is a linear function $\f_D$ on the Burnside ring of $D$ such that
$$
  \f(H) = \f_D(\Res_D H) \qquad
    \rlap{$\displaystyle
    \bigl(
    = \prod_{\scriptscriptstyle x\in H\backslash G/D} \f_D(H^{x^{-1}} \cap D)
    \>\>\bigr).$}
    \qquad\qquad
$$
\vskip-4pt
\noindent
In this case we write $\f=\tGDI{D,\f_D}_G$, or simply
$$
  \f=\GDI{D,\f_D}.
$$
\end{definition}

\begin{example}
\label{exlocfun}
Such functions arise naturally in a number-theoretic setting.
Suppose $F/K$ is a Galois extension of number fields and $v$ a place of~$K$.
Write $G=\Gal(F/K)$ and $D\!=\!\Gal(F_z/K_v)$ for the local Galois group
at $v$ (more precisely, a fixed decomposition group at $v$, so $D\<G$).
Under Galois correspondence, $\f_D$ associates something to every
extension of $K_v$, in which case $\f=\tGDI{D,\f_D}$ simply means
$$
  \f(L) = \prod_{\text{places } w|v \text{ in $L$}} \f_D(L_w)\>.
$$
(The double cosets $HxD$ correspond to the places $w$ of \smash{$L\!=\!F^H$} above $v$,
and $H \cap D^x$ are their decomposition groups in $\Gal(F/L)\!=\!H$.)
Typical local functions are those counting primes $w$ above~$v$ in $F^H$,
or primes with a given residue field \smash{$\F_q$}:
\beq
H \mapsto \lambda^{\text{\#\{$w$ above $v$ in $F^H$\}}} &
  (\,=\tGDI{D,\lambda}\>)\cr
H \mapsto \lambda^{\text{\#\{$w$ above $v$ in $F^H$ with $k(F^H_w)\iso\F_q$\}}}&
\bigl(\>=\GDI{D,H\mapsto\smallchoice{\lambda,\>}{\text{$k(F_w^H)\iso\F_q$}}{1,\>}{\text{else}\hfill}}\>\bigr),
\eeq
where $k(\cdot)$ denotes residue field.
Another example is the function
$$
  H \mapsto \prod\nolimits_{w|v} c_w(A/F^H)\>,
$$
\vskip-2pt
\noindent
that for an abelian variety $A/K$ computes the product of local Tamagawa
\newpage\noindent
numbers in $F^H$ above $v$.

Let $I\normal D$ be the inertia subgroup. If a place $w$ of $F^H$
corresponds to a double coset $HxD$, its decomposition and inertia groups in
$F/F^H$ are $H\cap D^x$ and $H\cap I^x$, respectively. Its
ramification and residue 
degree over~$v$ are $e_w\!=\!\tfrac{|I|}{|H\cap I^x|}$ and
$f_w\!=\tfrac{[D\,:\,I]}{[H\cap D^x\,:\,H\cap I^x]}$  
(the order of Frobenius in $F/K$ divided by that in $F/F^H$).
Many of the local functions that we will consider in \S\ref{sloccompat}
can be expressed through $e$ and $f$, which motivates the following
definition.
\end{example}

\begin{definition}
\label{gdidef2}
Suppose $I\triangleleft D\< G$ with $D/I$ cyclic,
and $\psi(e,f)$ is a function of two variables $e,f\in\N$.
Define
$$
  \GDI{D,I,\psi}: \quad
  H \longmapsto \prod_{x\in H\backslash G/D}
    \psi\bigl(\tfrac{|I|}{|H\cap I^x|},\tfrac{[D\,:\,I]}{[H\cap D^x\,:\,H\cap I^x]}\bigr)\>,
$$
the product being taken over any set of representatives of the double cosets.
This is a $D$-local function on the Burnside ring of $G$, to be precise
$$
\GDI{D,I,\psi}=
\GDI{D,U\mapsto\psi\bigl(\tfrac{|I|}{|U\cap I|},\tfrac{|D|}{|UI|}\bigr)}.
$$
\end{definition}

\begin{theorem}
\label{l2:}
Let $I\triangleleft D\< G$ with $D/I$ cyclic. Then
\begin{enumerate}
\item[($\ell$)]\label{l2:D}(Localisation)
If $\f=\tGDI{D,\f_D}$,
and $\f_D$ is trivial on $D$-relations, then
$\f$ is trivial on $G$-relations.
\item[(q)]\label{l2:Q}(Quotient)
If $N\triangleleft G$ and $\f(H)=\f_{G/N}(HN/N)$ for some function
$\f_{G/N}$ on the Burnside ring of $G/N$ which is trivial on $G/N$-relations,
then $\f$ is trivial on $G$-relations.
\item[(t)]\label{l2:T}(Transitivity)
If $D_1\< D_2\< G$, $\f=\tGDI{D_2,\f_2}_G$ and $\f_2=\tGDI{D_1,\f_1}_{D_2}$,
then $\f=\tGDI{D_1,\f_1}_G$.
\item[(f)]\label{l3:f} (Functions of $f$)
If $\psi(e,f)$ does not depend on $e$, then $\tGDI{D,I,\psi}\sim 1$.
\item[(r)]\label{l3:r} (Renaming)
If $I_0\< I$ is normal in $D$ with cyclic quotient, and $\psi(e,f)$ is
a function of the product $ef$, then $\tGDI{D,I,\psi}\sequal\tGDI{D,I_0,\psi}$.
\item[(d)]\label{l3:d} (Descent)
If $I\< D_0\< D$ and $\psi(e,f)=\psi(e,f/m)^m$
whenever $m$ divides $f$ and $[D:D_0]$,
then $\tGDI{D,I,\psi}\sequal\tGDI{D_0,I,\psi}$.
\end{enumerate}
\end{theorem}

\begin{proof}
($\ell$) and (q) follow from Theorem \ref{l:}\eqref{l:res}
and \eqref{l:pro}, respectively.
(t) is immediate from the $G$-set interpretation
in Definition \ref{gdidef}.
(f) follows from properties ($\ell$) and (q), and that
the cyclic group $D/I$ has no non-trivial relations.
(r), (d) follow from the definitions.
\end{proof}

\begin{example}
\label{vovian}
Property (f) shows that the two functions
$H\mapsto \lambda^{\#\cdots}$ counting primes in Example \ref{exlocfun}
are representation-theoretic, in other words they cancel in relations.
For instance, suppose that $F/K$ is a Galois extension of number fields and
$\sum_i H_i - \sum_j H_j'$ is a $\Gal(F/K)$-relation. Writing
\smash{$L_i=F^{H_i}$} and \smash{$L_j'=F^{H_j'}$},
$$
 \sum\nolimits_{i} \#\{\text{real places of } L_i\}
 = \sum\nolimits_{j} \#\{\text{real places of } L_j'\}\>.
$$
The same is true for complex places, or primes above a fixed prime $v$
of $K$ with a given residue degree over $v$. (This does not work when
counting primes with a given ramification degree instead.)
\end{example}

\newpage

\begin{example}
\label{exWWe}
If $W\< G$ is a subgroup of odd order,
then $\tGDI{W,W,e}\sim 1$ as functions to $\Q^{\times}/\Q^{\times 2}$.
For $W$ is solvable by the Feit-Thompson theorem, and
picking $W_0\normal W$ of prime index $p$,
$$
  \GDI{W,W,e} \sequalr \GDI{W,W_0,ef} \simf \GDI{W,W_0,e} \sequald \GDI{W_0,W_0,e}.
$$
The assertion follows by induction. (The inductive step fails for $p=2$,
e.g.
for $G=\Cy_2\!\times\!\Cy_2$ the function $\tGDI{G,G,e}$ does not
cancel in the relation of Example \ref{exc2c2}.)
\end{example}

Finally, we record a variant of the ``descent'' criterion of
Theorem \ref{l2:}d.

\begin{lemma}[Refined $p$-descent]\label{ladvp}
Let $N\normal G$ be of prime index $p$. Suppose $\phi, \psi$ are functions
on the Burnside rings of $G$ and $N$ respectively. Then $\phi=\tGDI{N,\psi}$ if and only if
\begin{itemize}
\item $\phi(H)=\psi(H\cap N)$ if $H\not\subset N$, and
\item $\phi(H)=\prod_{x\in G/N}\psi(xHx^{-1})$ if $H\subset N$, for any
choice of representatives.
\end{itemize}
\end{lemma}
\begin{proof}
The double cosets $H\backslash G/N$ are precisely the left cosets $G/N$
for \hbox{$H\subset N$,} and there is a unique double coset otherwise.
\end{proof}

\subsection{$\D_{\rho}$ and $\Ttp$}\rlap{\hskip 50cm.}
\par\noindent
\label{svanishing}

\noindent
We now introduce the function $\D_\rho$ that computes regulator constants,
and use it to study their properties.
Once again, $\K$ is any field of characteristic zero.
At the end of the section, we
reformulate these results for $\K=\Q_p$ in terms of the
sets $\Ttp$ of \S\ref{ssappl}. 

\begin{definition}
\label{D_M}
For a self-dual $\K G$-representation $\rho $ with a non-degenerate
$\K$-valued $G$-invariant bilinear pairing $\lara$, define
$$
  \D_\rho :\>  H \longmapsto \det(\tfrac1{|H|}\lara|\rho ^H) \> \in \K^{\times}/\K^{\times2}.
$$
By $G$-invariance of the pairing, $\D_\rho (H) = \D_\rho (xHx^{-1})$,
so this is indeed a function on the Burnside ring.
If $\Theta$ is a $G$-relation, then by definition of regulator constants
$$\D_\rho (\Theta)=\RC_\Theta(\rho ).$$
Up to $\sim$, this function is independent of the pairing:
if $\D'_\rho $ is defined in the same way but with a different pairing on $\rho $,
then \hbox{$\D_\rho\!\sim\!\D'_\rho$} by Theorem~\ref{regconst}.
In particular, $\D_{\rho\oplus\rho'}\sim\D_{\rho}\D_{\rho'}$.
\end{definition}

\begin{example}
\label{exDrho}
$\D_\triv$ is the function $H\mapsto |H|\in\K^{\times}/\K^{\times2}$,
and $\D_{\K[G]}$ is the constant function $H\mapsto 1$ (cf. Example \ref{regperm}).
\end{example}

\begin{remark}
The function $\D_\rho$ is representation-theoretic (i.e. $\sim 1$)
if and only if $\rho$ has trivial regulator constants in all $G$-relations.
For example, this happens  if $\rho$ carries a non-degenerate
$G$-invariant alternating form
(Theorem~\ref{vanishing}).
On the other hand,
$\RC_\Theta(\triv)$ need not be trivial, so $\D_\triv\not\sim 1$ in general
(cf. \ref{exdihreg}--\ref{exsl2f3}, \ref{exHHfun}).
\end{remark}

\begin{lemma}
\label{corUDG}
If $D\< G$ and $\rho$ is a self-dual $\K D$-representation,
then $$\D_{\Ind_D^G \rho}\sim \tGDI{D,\D_\rho}$$
as functions to $\K^{\times}/\K^{\times2}$.
\end{lemma}

\begin{proof}
Pick a $D$-invariant non-degenerate $\K$-bilinear pairing $\lara$ on $\rho$.
For a $D$-set $X$, define a pairing $(,)$ on $\Hom(X,\rho)$ by
$$
  (f_1,f_2)=\tfrac{1}{|D|}\sum_{x\in X}\langle f_1(x),f_2(x)\rangle,
    \qquad f_1,f_2\in \Hom(X,\rho).
$$
If $X=D/U$, the pairing $(,)$ on $\Hom_D(X,\rho)\subset\Hom(X,\rho)$ agrees
with $\tfrac{1}{|U|}\lara$ on $\rho^U\subset \rho$ under the identification
$\Hom_D(D/U,\rho)=\rho^U$ given by $f\mapsto f(1)$. So for
a general $D$-set $X=\coprod_i D/U_i$,
$$
  \textstyle
  \D_\rho\bigl(\sum_i U_i\bigr) = \det\bigl( (,) \bigm| \Hom_D(X,\rho)    \bigr).
$$
Applying this to $X=G/H$, we have
\beq
  \GDI{D,\D_\rho}(H)&=&\D_\rho(\Res_D H) = \det\bigl( (,) \bigm| \Hom_D(G/H,\rho)  \bigr)
  \cr
    &=& \det\bigl( \tfrac{1}{|H|}(,) \bigm| [\Hom_D(G,\rho)]^H  \bigr)
    = \D_{\Ind_D^G \rho}(H),
\eeq
where the last equality uses that $(,)$ is in fact a $G$-invariant pairing
on $\Hom_D(G,\rho)=\Ind_D^G \rho$.
\end{proof}

\begin{corollary}
\label{GoverD}
Let $I\triangleleft D\< G$ with $D/I$ cyclic.
As functions to $\K^{\times}/\K^{\times2}$,
$$
  \D_{\K[G/D]}
  \simudg \GDI{D,\D_\triv}
  \sim \GDI{D,H\mapsto\tfrac{1}{|H|}}
  \simf \GDI{D,H\mapsto\tfrac{|D|}{|H|}}
   = \GDI{D,I,ef}.
$$
\end{corollary}

\smallskip

Regulator constants behave as follows under lifting,
induction and restriction of relations (cf. Theorem \ref{l:}).

\begin{proposition}
\label{proprc}
Let $\rho$ be a self-dual $\K G$-representation.
\begin{enumerate}
\item
Suppose $G=\tilde G/N$ and $\Theta$ is a $G$-relation.
Lifting $\Theta$ to a $\tilde G$-relation~$\tilde\Theta$,
we have 
$\RC_{\tilde\Theta}(\rho)=\RC_\Theta(\rho)$.
\item If $G\<U$ and $\Theta$ is a $U$-relation, then
$\RC_\Theta(\Ind_G^U\rho)=\RC_{\Res_G\Theta}(\rho)$.
\item If $D\<G$ and $\Theta$ is a $D$-relation, then
$\RC_\Theta(\Res_D\rho)=\RC_{\Ind_D^G\Theta}(\rho)$.
\end{enumerate}
\end{proposition}

\begin{proof}
(1) The left- and the right-hand side are the same expression
up to a factor of $\prod_i |N|^{n_i\dim\rho^{H_i}}$,
if we write $\Theta=\sum_i n_i H_i$. This factor equals 1
by Lemma \ref{dimsaddup}.
(2) This is a reformulation of Lemma \ref{corUDG}.
(3) Clear.
\end{proof}

In view of Example \ref{exDrho}, the regular representation has
trivial regulator constants. Generally, we have

\begin{lemma}
\label{regreg}
If $H\< G$ is cyclic, $\RC_\Theta(\K[G/H])=1$ for every $G$-relation $\Theta$.
\end{lemma}

\begin{proof}
For $H$ cyclic, $\D_{\K[G/H]}\simef \tGDI{H,1,f} \simf 1$.
\end{proof}

\smallskip

\begin{theorem}
\label{oddorder}
If $G$ has odd order, then $\RC_\Theta(\rho)=1$
for every $G$-relation $\Theta$ and every self-dual
$\K G$-representation $\rho$.
\end{theorem}

\begin{proof}
The only self-dual irreducible $\bar\K G$-representation is the trivial one.
(Their number coincides with the number
of self-inverse conjugacy classes of $G$, but these have odd order and so,
except for the trivial class, have no self-inverse elements.)
By Corollary \ref{triviality}(3), if $\rho$ does not contain $\triv$,
then $\RC_\Theta(\rho)=1$. But then $\RC_\Theta(\triv)=\RC_\Theta(\K[G])$,
which is $1$ by Lemma~\ref{regreg}.
\end{proof}

\begin{corollary}
$\RC_\Theta(\K[G/H])=1$ if $H\<G$ has odd order.
\end{corollary}

\begin{proof}
$\RC_\Theta(\K[G/H])\,{\buildrel{\ref{proprc}(2)}\over{=}}\,\RC_{\Res_H\Theta}(\triv_H)=1$.
\end{proof}

\begin{lemma}
\label{regregp}
Suppose $\K\!=\!\Q$ or $\K\!=\!\Q_p$, and $N\normal H\< G$ with $H/N$ cyclic.
If $p\nmid |N|$,
then $\ord_p\RC_\Theta(\K[G/H])$ is even for every $G$-relation $\Theta$.
\end{lemma}

\begin{proof}
$\D_{\K[G/H]}\simef \tGDI{H,N,ef}\simf \tGDI{H,N,e}$ and the values of $e$ are divisors
of $|N|$. So $\RC_\Theta(\K[G/H]) = \D_{\K[G/H]}(\Theta)$ has even
$p$-valuation for $p\nmid|N|$.
\end{proof}

\medskip

\subsubsection*{Reformulation for $\K=\Q_p$}
We now define the sets of representations $\Ttp$ that
encode those representations whose
regulator constants are ``$p$-adicially non-trivial''.
It is for these twists that we prove the $p$-parity conjecture
(Theorem \ref{tamroot}). So let us also restate the properties of
regulator constants in this language ($\Ttp$-ese?)

\begin{definition}
\label{tauthetap}
Suppose $\K=\Q_p$. 
For a $G$-relation $\Theta$ define
$\Ttp$ to be the set of self-dual $\Qpb G$-representations
$\tau$ that satisfy
$$
 \blangle\tau,\rho\brangle \equiv \ord_p \RC_\Theta(\rho) \mod 2
$$
for every self-dual $\Q_p G$-representation $\rho$.
\end{definition}

\begin{remark}
\label{remT0p}
For instance, $\Ttp$ contains representations of the form
$$
    \bigoplus_{\vabove{\text{$\rho$ self-dual $\Q_p$-irr.}}{\ord_p\RC_\Theta(\rho)\text{ odd}}}
    \!\!\!\!\!
    (\text{any $\Qpb$-irreducible constituent of $\rho$}).
$$
These are indeed self-dual, since $\RC_\Theta(\sigma\oplus\sigma^*)=1$
by Corollary \ref{triviality}.
Note also that these particular representations
have no symplectic constituents or those with even Schur index,
by the same corollary.

A general element of $\Ttp$ differs from this one by something
in $\Tau_{0,p}$, in other words by a self-dual (virtual)
$\Qpb$-representation whose inner product with any self-dual
$\Qp$-representation is even. Concretely, $\Tau_{0,p}$ is generated by
representations of the form $\sigma\oplus\sigma^*$ (in particular
$\sigma^{\oplus 2}$ for self-dual $\sigma$), and irreducible self-dual
$\sigma$ with either an even number of $\Gal(\Qpb/\Q_p)$-conjugates or
even Schur index over $\Q_p$.

In the context of the $p$-parity conjecture, the elements of $\Tau_{0,p}$
correspond to twists $\tau$ for which the conjecture should ``trivially''
hold. The parity of $\blangle\tau,\X_p(A/F)\brangle$ is even, and we
expect $w(A,\tau)=1$ for these $\tau$. (This is indeed the case whenever
we have an explicit formula for this root number.)
\end{remark}

\begin{example}
\label{exTtp1}
If $G\!=\!\Cy_n$ is cyclic, $\Theta\!=\!0$ is the only $G$-relation.
The set $\Tau_{0,p}$ consists of $\Z$-linear combinations of
$\triv^{\oplus 2}$, sign$^{\oplus 2}$ for $n$ even, and
$\chi\oplus\chi^*$ for all the remaining $1$-dimensionals $\chi$.
\end{example}

\begin{example}
\label{exTtp2}
If $G\!=\!\Di_{2p}$ with $p$ odd, and $\Theta\!=\!\{1\} \!-\! 2\Cy_2 \!-\! \Cy_p \!+\! 2\,G$
of Example \ref{exdihreg}, then $\tau\in\Ttp$
if and only if $\tau$ contains an odd number of trivial representations,
an odd number of sign representations, and
in total an odd number of $2$-dimensional $\Qpb$-irreducibles.
\end{example}

\begin{example}
\label{exTtp3}
If $G\!=\!\Di_{2p^n}$, and $\tau$ any 2-dimensional $\Qpb G$-representation,
Examples \ref{exdihreg2}, \ref{exdihreg3} show that $\tau\oplus\triv\oplus\det\tau$
lies in $\Ttp$ for some $G$-relation~$\Theta$.
\end{example}

\begin{example}
\label{exTtp4}
If $G=\Alt_5$, its complex- (or $\Qpb$-) irreducible representations
are $\triv$, 3-dimensional $\tau_1,\tau_2$, 4-dimensional $\chi$
and 5-dimensional $\pi$. Here
$$
 \triv\!\oplus\!\pi\in\Tau_{\Theta,2}
 \>\qquad\triv\!\oplus\!\chi\!\oplus\!\pi\in\Tau_{\Theta',3}
 \>\qquad\triv\!\oplus\!\tau_i\!\oplus\!\chi\in\Tau_{\Theta'',5}\>,
$$
for some $G$-relations $\Theta,\Theta',\Theta''$ (see \cite{Squarity} Ex. 2.19).
\end{example}

\begin{theorem}[Properties of $\Ttp$]
\label{tauthetapprop}
Let $\Theta$ be a $G$-relation and $\tau\in\Ttp$.
\begin{enumerate}
\item
$\tau$ has even dimension and trivial determinant.
\item
$\tau\oplus\tau'\in\Tau_{\Theta+\Theta',p}\,$
for $\tau'\in\Tau_{\Theta',p}$.
\item
$\tilde\tau\in\Tau_{\tilde\Theta,p}$ whenever
$G=\tilde G/N$, and $\tilde\tau$, $\tilde\Theta$ are lifts of $\tau$
and $\Theta$ to $\tilde G$.

\item
If $D\<G$, then
$\Res_D\tau\in\Tau_{{\Res_D\Theta},p}$.
\item
If $G\<U$, then
$\Ind_G^U\tau\in\Tau_{{\Ind_G^U\Theta},p}$.
\item $\blangle\tau,\Q_p[G/H]\brangle$ is even whenever
$H\<G$ is cyclic, has odd order or contains a normal subgroup $N\normal H$
with $H/N$ cyclic and $p\nmid |N|$.

\item If $|G|$ is odd or coprime to $p$, then $\Ttp=\Tau_{0,p}$.
\end{enumerate}
\end{theorem}

\begin{proof}
\noindent (2) Clear.

\noindent (3) Proposition \ref{proprc}(1) and Lemma \ref{mustenter}.

\noindent (4) Take any self-dual $\Q_p D$-representation $\rho$. Then modulo 2,
$$
\begin{array}{lclcll}
    \blangle\Res_D\tau,\rho\brangle= 
    \blangle\tau,\Ind_D^G\rho\brangle\equiv 
    \ord_p\RC_{\Theta}(\Ind_D^G\rho)
    \smash{{\tbuildrel\ref{proprc}(2)\over\equiv}}
    \ord_p\RC_{\Res_D\Theta}(\rho).
\end{array}
$$
\noindent (5) Same computation, using Proposition \ref{proprc}(3).

\noindent (6), (7) Reformulation of \ref{mustdivide} and \ref{regreg}--\ref{regregp}.

\noindent (1)
$\dim\tau=\blangle\tau,\Q_p[G]\brangle$ is even by (6).
Now $\det\tau(g)=1$ for all $g\in G$ if and only if
$\det\Res_H\tau=\triv$ for all cyclic $H\<G$.
So (4) reduces the problem to cyclic groups, where it is clear
(see Example \ref{exTtp1}).
\end{proof}

\begin{corollary}
$\triv\not\in\Ttp$ and $\,\triv\!\oplus\!\epsilon\not\in\Ttp$ for any
1-dimensional $\epsilon\not\iso\triv$.
\end{corollary}

\begin{remark}\label{intrinsic}
In view of Theorems \ref{squality} and \ref{tamroot} we may call
$\Tau_p=\bigcup_{\Theta} \Ttp$ the space of ``$p$-computable'' twists.
This set of representations is canonically associated to a finite group
$G$ and a prime number $p$. It
behaves well under restriction and induction, and is
closed under direct sums and
tensor product with permutation representations (since
 $\tau\tensor \Ind_H^G\triv =
\Ind_H^G(\Res_H\tau)$ lies in~$\Tau_{\Ind_H^G\!\Res_H\!\Theta,p}$).
It would be very nice to have an intrinsic description~of~$\Tau_p$.
\end{remark}

\section{Root numbers and Tamagawa numbers}
\label{sloccompat}

The aim of this section is to establish the following statement
about compatibility of local root numbers and local Tamagawa numbers.
The proof will occupy all of \S\S\ref{sscompatnot}--\ref{ss-case4}.
But first, we will explain how together with Theorem \ref{squality}
it implies Theorem \ref{tamroot}, the central result of this paper
on the $p$-parity conjecture.
In fact, we expect the theorem below to hold for all
principally polarised abelian varieties,
and this would imply that the restrictions on the reduction of $A$ in
Theorem \ref{tamroot} could be removed.

\begin{notation}
Let $K$ be a local field of characteristic zero, $F/K$ a Galois
extension, and $A/K$ an abelian variety.
For $H\<\Gal(F/K)$ write (cf. \S\ref{ssnotation})
$$
  C_v(H) = C_v(A/F^H) = C_v(A/F^H,\neron{})
$$
for any exterior form $\neron{}$ on $A/K$, minimal if $K$ is non-Archimedean.
(We insist on minimality only for convenience: Theorem \ref{loccompat}
below holds for any choice of $\omega$
because \hbox{$C_v(\cdot,\omega)\sim C_v(\cdot,\neron{})$},
cf. proof of Corollary \ref{propglobcompat}.)
\end{notation}

\begin{theorem}[Existence of $\TV$]
\label{loccompat}
Let $K$ be a local field of characteristic zero, $F/K$ a Galois extension
with Galois group $D$ and $A/K$ a principally polarised abelian
variety. Assume that either
\begin{itemize}
\item[(1)]
$D$ is cyclic,
\item[(2)]
$A=E$ is an elliptic curve with semistable reduction,
\item[(3)]
$A=E$ is an elliptic curve with additive reduction and
$K$ has residue characteristic $l>3$, or
\item[(4)]
$A/K$ has semistable reduction.
\end{itemize}
Then there is a $\Q D$-module $\TV $ such that
\begin{itemize}
\item[\eps]
$\tfrac{w(A,\tau)}{w(\tau)^{2\dim A}}=(-1)^{\blangle\tau,\TV \brangle}$
      for all self-dual representations $\tau$ of $D$, and
\item[\tam] $C_v \sim \D_\TV$ as functions on the Burnside ring of $D$.
      Equivalently, for every $D$-relation $\Theta=\sum_i n_i H_i$,
$$
  \prod\nolimits_i C_v(A/F^{H_i})^{n_i} \equiv
  \RC_\Theta(\TV )\mod \Q^{\times2}.
$$
\end{itemize}
In the following exceptional subcase of (4) we only claim \tam{} up to
multiples of 2:
\begin{itemize}
\item[(4ex)] $A/K$ is semistable, $K$ has residue characteristic 2,
the wild inertia group of $F/K$ is non-cyclic and $A/K$ does
not acquire split semistable reduction over any odd degree extension.
\end{itemize}
\end{theorem}

In the setting of the theorem,
let $\Theta$ be a $D$-relation and $p$ a prime number,
odd in case (4ex). For any $\tau\in\Ttp$,
we obtain a chain of equalities:
\beq
  \tfrac{w(A/K,\tau)}{w(\tau)^{2\dim A}}
  \!\!\! &{\tbuildrel\eps\over=}\!\!\!&
    (-1)^{\blangle\tau, \TV \brangle}\!\!\! &=&
    (-1)^{\ord_p\RC_\Theta(\TV \tensor\Q_p)}\!\!\! &{\tbuildrel\tam\over=}\!\!\!&
    (-1)^{\ord_p C_{v}(\Theta)}
\eeq
(note that $C_v(\Theta)\in\Q^{\times}/\Q^{\times2}$ even for $K=\R,\C$
by property \tam, so $\ord_p C_{v}(\Theta)$ makes sense).
By the determinant formula  
$w(\tau)^2=1$,
as $\tau$ is self-dual and has trivial determinant
(Theorem \ref{tauthetapprop}(1), Lemma \ref{detformula}).
Thus,

\begin{corollary}[Local compatibility]
\label{proploccompat}
Suppose $F/K$ and $A/K$ are as in Theorem \ref{loccompat}.
Let $\Theta$ be a $D$-relation and $p$ a prime number, odd
in case~{\rm (4ex)}. Then for every $\tau\in\Ttp$,
$$
  w(A/K,\tau) = (-1)^{\ord_p C_v(\Theta)}.
$$
\end{corollary}

Now let us deduce Theorem \ref{tamroot}.
Suppose $F/K$ is a Galois extension of {\em number fields},
$A/K$ an abelian variety and $v$ a place of $K$. Fix a
non-zero regular
exterior form $\omega$ on $A/K$, and define functions
on the Burnside ring of $\Gal(F/K)$ by
$$
  C_{w|v}:\>H \mapsto \prod_{w|v} C_w(A/F^H,\omega)\quad\qquad
  C:\>H\mapsto C_{A/F^H} \>\>\bigl(=\prod_v C_{w|v}(H)\bigr)\>,
$$
the first product taken over the places of $F^H$ above $v$.

\begin{corollary}
\label{propglobcompat}
Let $F/K$ be a Galois extension of number fields,
$A/K$ an abelian variety, and fix
a place $z$ of $F$ above
a place $v$ of $K$.
Suppose $A/K_v$, $F_z/K_v$
satisfy the assumptions of Theorem \ref{loccompat}, and let
$p$ be a prime number, odd in case {\rm(4ex)}.
Then for every $\Gal(F/K)$-relation $\Theta$ and $\tau\in\Ttp$,
$$
  w(A/K_v,\Res_{\Gal(F_z/K_v)}\tau) = (-1)^{\ord_p C_{w|v}(\Theta)}.
$$
If the assumptions hold at all places $v$ of $K$, then
$$
  w(A/K,\tau) = (-1)^{\ord_p C(\Theta)}.
$$
\end{corollary}

\begin{proof}
Write $D=\Gal(F_z/K_v)\<\Gal(F/K)$ for the decomposition group of~$z$,
and $I$ for its inertia subgroup.
First note that $C_{w|v}(\Theta)$ is independent of the choice of the
exterior form $\omega$: if $\omega=\alpha\omega'$, then
$$
 \frac{\prod_{w|v} C_w(A/F^H,\omega)}{\prod_{w|v} C_w(A/F^H,\omega')} =
 \prod_{w|v} |\alpha|_{F^H_w} 
 = \GDI{D,I,|\alpha|_{K_v}^f}(H)\>,
$$
and this function is trivial on $\Theta$ by Theorem \ref{l2:}f.
So we may assume that $\omega$ is minimal at $v$.

Now
$\Res_D\tau\in\Tau_{\Res_D\Theta,p}$ by Theorem \ref{tauthetapprop}(4),
so
\beq
w(A/K_v,\Res_D\tau)\!\! &=\!\!\!& (-1)^{\ord_p C_v(\Res_D\Theta)} &(\text{Corollary \ref{proploccompat}})\cr
&=\!\!\!&(-1)^{\ord_p C_{w|v}(\Theta)} &(\text{as } C_{w|v}\!=\!\tGDI{D,C_v}, \text{cf. \ref{gdidef}, \ref{exlocfun}}).
\eeq
For the last claim, take the product over all places.
\end{proof}

\begin{proof}[Proof of Theorem \ref{tamroot}]

The abelian variety $A/K$
satisfies the hypothesis of Corollary \ref{propglobcompat}
at all places of $K$. So
for every $\tau\in\Ttp$
$$
  w(A/K,\tau)
  \>\>{\tbuildrel\ref{propglobcompat}\over=}\>\>
  (-1)^{\ord_p C(\Theta)}
  \>\>{\tbuildrel\ref{squality}\over=}\>\>
  (-1)^{\blangle\tau,\Xp(A/F)\brangle}.
$$
\end{proof}

\subsection{Setup}
\label{sscompatnot}

In the remainder of \S\ref{sloccompat} we prove Theorem \ref{loccompat}.
Let $A/K$ and $F/K$ be as in the theorem, in particular $K$ is again
{\em local}.
We split cases (1)-(3) into subcases
and define an extension $L$ of $K$ as follows:

\begin{notation}
\label{hypFK}
\noindent\begingroup
\begin{itemize}\itemindent -6pt
\item[(1)] $D$ is cyclic.
\begin{itemize}
  \item[(1-)] $|D|$ is odd; $L=K$.
  \item[(1+)] $|D|$ is even; $L$ is the
        unique quadratic extension of $K$ inside $F$.
\end{itemize}
\item[(2)] $A=E$ is an elliptic curve with semistable reduction.
\begin{itemize}
  \item[(2G)] $E$ has good reduction; $L=K$.
  \item[(2S)] $E$ has split multiplicative reduction; $L=K$.
  \item[(2NS)] $E$ has non-split multiplicative reduction;
                       $L/K$ is quadratic unramified.
\end{itemize}
\item[(3)] $A=E$ is an elliptic curve with additive reduction,
           $K$ has residue characteristic $l>3$.
           Write $\Delta_E$ and $c_6$ for the standard invariants
           of some model of $E/K$ and $\e=\frac{12}{\gcd(12,\ord\Delta_E)}$.
\begin{itemize}
  \item[(3C)] $E$ has potentially good reduction, $\mu_\e\subset K$; \\
       $L=K(\sqrt[\e]{\Delta_E}),$ a cyclic extension of $K$.
  \item[(3D)] $E$ has potentially good reduction, $\mu_\e\notsubset K$;\\
       $L=K(\mu_\e,\sqrt[\e]{\Delta_E})$, a dihedral extension of $K$.
  \item[(3M)] $E$ has potentially multiplicative reduction;
        $L=K(\sqrt{-c_6})$.
\end{itemize}
\item[(4)] $A/K$ has semistable reduction; $L$ is the smallest unramified extension of $K$ where
$A$ acquires split semistable reduction.
\end{itemize}
\end{notation}\endgroup

\noindent
We remind the reader that (4) has a subcase (4ex), see Theorem \ref{loccompat}.
Note that (1) includes Archimedean places, and
in (2)-(4) $L$ is a minimal Galois extension of $K$ where
$A$ acquires split semistable reduction (cf. Lemma~\ref{lemcv}).

In view of Lemma \ref{containL} below, we may and will henceforth assume

\begin{hypothesis}
$F$ contains $L$.
\end{hypothesis}

\begin{notation}
Henceforth write
\beq
D &= & \Gal(F/K),\cr
D'&= & \Gal(F/L)\normal D,\cr
I &= & \text{Inertia subgroup of $D$},\cr
W &= & \text{Wild inertia subgroup of $I$}.\cr
\eeq
We work extensively with
functions from the Burnside ring of~$D$ to~$\Q^{\times}/\Q^{\times 2}$.
For brevity, $\tGDI{\cdots}$
stands for $\tGDI{\cdots}_D$ in \S\S\ref{ss-case1}--\ref{ss-case4}
(see Definitions \ref{gdidef}, \ref{gdidef2}).
\end{notation}

\begin{lemma}
\label{containL}
Suppose $\TV $ is a $\Q G$-module and $M/K$ a Galois extension
contained in $F$.
If $\TV $ satisfies \eps{} and the $p$-part of \tam{} of Theorem~\ref{loccompat},
then $\W=\TV ^{\Gal(F/M)}$ satisfies the same conditions for the extension~$M/K$.
\end{lemma}

\begin{proof}
The irreducible constituents of $\W$ are precisely those of $\TV $
that factor through $\Gal(M/K)$, so $\W$ clearly satisfies \eps.
If $\Theta$ is a $\Gal(M/K)$-relation and $\tilde\Theta$ is its lift to $G$
(as in Theorem \ref{l:}\eqref{l:lif}),
then $\RC_{\tilde\Theta}(\TV \ominus \W)=1$ by Lemma~\ref{mustenter}.
So $\W$ satisfies \tam.
\end{proof}

The choice of $\TV $ is forced by formulae for the local
root numbers $w(A/K,\tau)$.
For a self-dual representation $\tau$ of $\Gal(F/K)$, we claim that
$$
  w(A/K,\tau) = w(\tau)^{2\dim A}\,\lambda^{\dim \tau}\,
    (-1)^{\blangle\tau,V\brangle}
$$
with $\lambda=\pm1$ and the representation $V$ of $D/D'$ given in Table \ref{roottable}.
Here \linebreak \noindent \vskip-13pt
\stepcounter{equation}
\begin{table}[h]
$$
\begin{array}{|c|c|c|c|}
\hline
\text{Case} & D/D' & \lambda   & V   \cr
\hline
\vphantom{\int^X}
\text{(1-)}    & 1     & w(A/K)  & 0 \cr
\vphantom{a_{\int_X}}
\text{(1+)}    & \Cy_2 & w(A/K)  & \chi^{\oplus b} \cr
\noalign{\hbox to 221pt{\leaders\hbox to 0.3em {\hss\smash{\tiny.}\hss}\hfill}}
\vphantom{\int^X}
\text{(2G)}    & 1      & 1                          & 0           \cr
\text{(2S)}    & 1      & 1                          & \triv       \cr
\vphantom{a_{\int_X}}
\text{(2NS)}   & \Cy_2  & 1                          & \eta        \cr
\noalign{\hbox to 221pt{\leaders\hbox to 0.3em {\hss\smash{\tiny.}\hss}\hfill}}
\vphantom{\int^X}
\text{(3C)}    & \Cy_2, \Cy_3, \Cy_4, \Cy_6 & \epsilon & 0 \vphantom{\int^{X}}   \cr
\text{(3D)}    & \Di_6, \Di_8, \Di_{12}   & -\epsilon & \triv\oplus\eta\oplus\sigma\cr
\vphantom{a_{\int_X}}
\text{(3M)}    & \Cy_2                    & w(\chi)^2  & \chi  \cr
\noalign{\hbox to 221pt{\leaders\hbox to 0.3em {\hss\smash{\tiny.}\hss}\hfill}}
\vphantom{\int^X}
\text{(4)}     & \text{cyclic}            & 1 & X(\T^*)\!\tensor\!\Q  \cr
\hline
\end{array}
$$
\caption{Root numbers}
\label{roottable}
\end{table}

\vskip-10pt \noindent
$\triv, \chi, \eta$ and $\sigma$ are the trivial character,
the non-trivial character of order 2, the unramified quadratic character
and the unique faithful 2-dimensional representation of $D/D'$.
The exponent $b$ is defined by $(-1)^b=\wAL$,
$\epsilon$ is as in \cite{RohG} Thm. 2 (we will not need it explicitly)
and $X(\T^*)$ is the character group of the torus in the Raynaud
parametrisation of the dual abelian variety $A^t$, see \S\ref{ss-case4}.

Granting the claim, define
$$
  \TV = V \>\>\oplus\>\>\left\{
  \begin{array}{lll}
     0, & \lambda = 1,\cr
     \Q[D], & \lambda = -1.\cr
  \end{array}
  \right.
$$
By Frobenius reciprocity, $\blangle\tau,\Q[D]\brangle=\dim\tau$,
so $\TV$ satisfies property \eps{} of Theorem \ref{loccompat}.
Moreover, $\D_{\Q[D]}\sim 1$ by Lemma \ref{regreg}, so $\D_\TV\sim\D_V$.

It remains to prove the following

\begin{proposition}
\label{propcvdv}
In each of the cases (1)--(4) and $V, \lambda$ as in Table \ref{roottable},
we have $\D_V\sim C_v$ (up to multiples of 2 in case {\rm(4ex)})
and
$$
  w(A/K,\tau) = w(\tau)^{2\dim A}\,\lambda^{\dim \tau}\,
    (-1)^{\blangle\tau,V\brangle}
$$
for every self-dual representation $\tau$ of $\Gal(F/K)$.
\end{proposition}

The proof is a case-by-case analysis and will occupy \S\ref{ss-case1}--\S\ref{ss-case4}.

\subsection{Case (1): Cyclic decomposition group}
\label{ss-case1}

\begin{lemma}
Su\smash{ppose $D=\Gal(F/K)$ is cyclic. Then}
\beq
w(A/K,\tau) = w(\tau)^{2\dim A} w(A/K)^{\dim\tau}
   &\>\>\>\text{if}&2\nmid[F\!:\!K],\cr
w(A/K,\tau) = w(\tau)^{2\dim A} w(A/K)^{\dim\tau}
   \bigl(\wAL\bigr)^{\blangle\tau,\chi\brangle}
   &\>\>\>\text{if}&2\,|\,[F\!:\!K].\cr
\eeq
Moreover, $w(A/K)$ and \smash{$\wAL$} are $\pm 1$.
\end{lemma}

\pagebreak

\begin{proof}
For the last claim, local root numbers of abelian varieties are $\pm 1$,
see e.g. \cite{SabR} \S1.1. So $w(A/K)=\pm 1$, and the same holds for
the quadratic twist of $A$ by $\chi$.
By the determinant formula, 
$w(\chi)^2=\pm 1$ as well (Lemma \ref{detformula}).

By Lemma \ref{detformula},
$w(A/K,\rho\oplus\rho^*)=1$ for every representation $\rho$, so it suffices
to check the formulae for $\tau=\triv$ and for
$\tau=\chi$ when $[F:K]$ is even.
But this is clear as $w(\triv)=1$.
\end{proof}

As $D$ is cyclic and has therefore no relations, we trivially have
\hbox{$C_v\!\sim\! 1\!\sim\!\D_V$}.
This proves Proposition \ref{propcvdv} in Case (1).

\subsection{Case (2): Semistable elliptic curves}
\label{ss-case2}

The root number formula follows from \cite{RohG} Thm. 2
and the determinant formula (Lemma \ref{detformula}).

We now prove that $\D_V\sim C_v$.
Note that the differential $\omega$ remains minimal in all
extensions of $K$, so $C_v(E/F^H)=c_v(E/F^H)$ for all $H\<D$.
By Tate's algorithm (\cite{Sil2}, IV.9), in terms of $e=e_{F^H/K}$ and $f=f_{F^H/K}$
these Tamagawa numbers are:
$$
\begin{array}{c|ccc}
  \text{Reduction of $E/K$} &
  \text{Good} & \text{Split $\In{n}$} & \text{Non-split $\In{n}$} \cr
\hline
     c_v(E/F^H)
     & 1 & ne & \biggl\{\smaller[2]
     \begin{array}{ll}
       1,  & 2\nmid f,2\nmid ne,\cr
       2,  & 2\nmid f,2\,|\,ne,\cr
       ne, & 2\,|\,f.\cr
     \end{array}
     \cr
\end{array}
$$

\subsubsection{(Case 2G) $E$ has good reduction}

$\D_V=C_v=1$.

\subsubsection{(Case 2S) $E$ has split multiplicative reduction}
\label{sssplit}

If $E/K$ has type~I$_n$, 
$$
  \D_V=\D_{\Q[D/D]} \simef
  \GDI{D,I,ef} \simf \GDI{D,I,ne} =  C_v.
$$

\subsubsection{(Case 2NS) $E$ has nonsplit multiplicative reduction}
\label{ssnonsplit}

If $E/K$ has type~I$_n$, then
$$
  \D_V=\D_\eta\sim\D_{\Q[D/D']}/\D_{\Q[D/D]}
  \simef { \GDI{D',I,ef} }\,/{ \GDI{D,I,ef} }
$$
$$
  \sequala { \bGDI{D,I,\choice{ef}{2\nmid f}{(ef/2)^2}{2|f}}}\,/{ \GDI{D,I,ef} }
  = { \bGDI{D,I,\choice{1}{2\nmid f}{ef}{2|f}}}
$$
$$
  \simf { \bGDI{D,I,\choice{1}{2\nmid f}{en}{2|f}}}
  =  C_v \cdot
       { \bGDI{D,I,\choice{2}{2\nmid f, 2|ne}{1}{\text{else}}}}.
$$
It remains to show that the last factor is $\sim 1$.
If $n$ is even, it is a function of~$f$ and therefore $\sim 1$ by
Theorem \ref{l2:}f.

Now suppose $n$ is odd. Then
$$
  { \bGDI{D,I,\choice{2}{2\nmid f, 2|e}{1}{\text{else}}}}
    \sequald
  { \bGDI{I,I,\choice{2}{2|e}{1}{2\nmid e}}}
    \sequalr
  { \bGDI{I,W,\choice{2}{2|ef}{1}{2\nmid ef}}}.
$$
If $v\nmid 2$, then $W$ has odd order, so this is a function of $f$,
hence $\sim 1$ again.
Henceforth assume $v|2$, so $W$ is a 2-group and $[I:W]$ is odd. Then
$$
  { \bGDI{I,W,\choice{2}{2|ef}{1}{2\nmid ef}}}
    \sequald
  { \bGDI{W,W,\choice{2}{2|e}{1}{2\nmid e}}}
    \simf
  { \bGDI{W,W,\choice{1}{2|e}{2}{2\nmid e}}}.
$$
By Theorem \ref{l2:}$\ell$, it suffices to prove that
the function on subgroups of~$W$
$$
  \f_W: H \longmapsto \choice{1}{H\ne W}{2}{H=W}
$$
is trivial on $W$-relations.

Let $W/\Phi=\bar W$ be the maximal exponent 2 quotient of $W$
(so $\Phi\normal W$ is its Frattini subgroup.)
Since proper subgroups of $W$ cannot have full image in~$\bar W$,
we have $\f_W(H)=\f_{\bar W}(\bar H)$.
By Theorem \ref{l2:}q, it is enough to verify that $\f_{\bar W}$
is trivial on $\bar W$-relations. But every $\bar W$-relation has an even
number of terms with $H=\bar W$ (only these have $\C[\bar W/H]$ of odd
dimension), so this is clear.

\subsection{Case (3): Elliptic curves with additive reduction}
\label{ss-case3}

We now come to the truly painful case of additive reduction.
Thus $l\ne 2,3$ is a fixed prime, $K$ a finite extension of $\Q_l$ and
$E/K$ has additive reduction. We write $q$ for the size of the
residue field of $K$ and
$\delta$ for the valuation of the minimal discriminant of $E/K$.
The asserted root number formula again comes from \cite{RohG} Thm. 2,
and it remains to show $\D_V\sim C_v$.

Decompose the functions
$$
  \D_V = a\cdot d, \quad C_v = c_v\cdot \omega
$$
with
\beq
  a(H)&=&\det(\lara|V^H) &&&&
  d(H)&=&|H|^{-\dim V^H}\\[2pt]
  c_v(H)&=&c_v(E/F^H)&&&&
  \omega(H)&=&\Bigl|\frac{\neron{{\scriptscriptstyle E/K}}}{{\neron{E/F^{\text{\smaller[3]$\scriptscriptstyle H$}}}}}\Bigr|_{\scriptscriptstyle{F^H}}.\cr
\eeq
These are well-defined on conjugacy classes of subgroups of $D$ and take
values in $\Q^{\times}/\Q^{\times2}$. The pairing $\lara$ on $V$ may be chosen
arbitrarily, and we picked one that seemed natural to give explicit values
of $a$ in Tables~\ref{dihtable},~\ref{c6c2table} (write $V$ as a sum
of permutation modules and use Example \ref{regperm}).

The function $\omega$ may be expressed in terms of $e=e_{F^H/K}$
and $f=f_{F^H/K}$ as follows. The minimal discriminant of $E/K$ has
valuation $\delta e$ over $F^H$, so
$\omega(H)=q^{\lfloor (\delta e-\delta_H)/12\rfloor f}$
where $\delta_H$ is the valuation of the minimal discriminant of $E/F^H$
(cf. \cite{Sil1} Table III.1.2). If $E$ has potentially good reduction,
then $0\!\le\!\delta_H\!<\!12$. If the reduction is potentially
multiplicative of type \InS{n} (so $\delta=6+n$),
it becomes $\InS{ne}$ over $F^H$ if $e$ is odd ($\delta_H=6+ne$),
and $\In{ne}$ if $e$ is even ($\delta_H=ne$). It follows easily that
$$
  \omega = \biggl\{ \begin{array}{ll}
    \tGDI{D,I,q^{\tw{\delta e}f}}, & \text{if $E$ has potentially good reduction}\cr
    \tGDI{D,I,q^{{{\lfloor\frac e2\rfloor}}f}}, & \text{if $E$ has potentially multiplicative reduction.}
  \end{array}
$$

\subsubsection{Reduction to 2-power residue degree}

\begin{lemma}
\label{cvunch}
If $K''/K'$ is a subextension of $F/K$, and is unramified of odd degree,
then $C_v(E/K'')\equiv C_v(E/K')\mod\Q^{\times2}$.
\end{lemma}

\begin{proof}
For $c_v$ this follows from Lemma \ref{lemcv}.
For $\omega$ this is clear.
\end{proof}

\begin{lemma}
\label{DI2}
It suffices to prove Case (3) of Proposition \ref{propcvdv} when $f_{F/K}$ is a power of~2.
\end{lemma}

\begin{proof}
Let $N\normal D$ correspond to the maximal odd degree unramified extension of
$K$ in $F$. As $C_v$ is unchanged in odd degree unramified extensions
(Lemma~\ref{cvunch}), a repeated application of
Lemma \ref{ladvp} with $\phi=\psi=C_v$ shows that $C_v=\tGDI{N,C_v}$.
We claim that $\D_V\sim \tGDI{N,\D_{\Res_N V}}$. Then,
by Theorem~\ref{l2:}$\ell$, it suffices to show that
$\D_{\Res_N V}\sim C_v$ as
functions on the Burnside ring of $N$, as asserted.

By Lemma~\ref{corUDG},
$\tGDI{N,\D_{\Res_N V}}\sim \D_{\Ind_N^G\Res_N V}$. But
$$
  \Ind_N^G\Res_N V \>\iso\> V\tensor\Ind_N^G\triv_N \>\iso\>
    V \oplus J
$$
with $J$ of the form ${\mathcal J}\!\oplus\!{\mathcal J}^*$ over $\bar\Q$. By
Corollary \ref{triviality}, \hbox{$\D_J(\Theta)\!=\!\RC_\Theta(J)\!=\!1$}
for any $D$-relation $\Theta$.
Thus $\D_J\!\sim\! 1$, whence $\D_{\Ind_N^G\Res_N V}\!\sim\!\D_V$, as required.
\end{proof}

\subsubsection{(Cases 3C, 3D) $E/K$ has potentially good reduction}
\label{sssredtyp}

By Lemma~\ref{DI2}, it suffices to prove the following

\begin{claim}
Suppose $f_{F/K}$ is a power of 2.
Then $c_v\sim a$ and $\omega\sim d$, and hence $C_v\sim\D_V$.
\end{claim}

Let $\e=e_{L/K}$ be the ramification degree of $L$ over $K$.
The extension is either cyclic or dihedral (cf. Table \ref{roottable}),
to be precise
\beq
  \text{(Case 3C)}\quad & \Gal(L/K) \iso \Cy_\e    & \e=2,3,4,6, \cr
  \text{(Case 3D)}\quad & \Gal(L/K) \iso \Di_{2\e} & \e=3,4,6.   \cr
\eeq
As $E/L$ has good reduction, $\delta\e\equiv 0\mod 12$,
and $\e$ is the smallest such integer. Moreover,
$q\equiv (-1)^t\mod\e$ with $t=0$ in Case 3C and $t=1$ in Case 3D
(see e.g. \cite{RohG} Thm. 2).

As $V$ is a $\Gal(L/K)$-representation,
both $a(H)$ and the exponent $\dim V^H$ in $d(H)$
depend only on $F^H \cap L$.
In Case 3C, $V\!=\!0$ and $a(H)\!=\!d(H)\!=\!1$.
For Case 3D, we summarise in Table \ref{dihtable}
the subgroups $H$ (up to conjugacy) of
$$
  \Gal(L/K)=D/D'=\Di_{2\e}=\langle g,h|h^\e\!=\!g^2\!=\!(gh)^2\!=\!1\rangle
$$
with the following data:
\stepcounter{equation}%
\begin{table}[h]
\caption{Dihedral quotient}
\label{dihtable}

\begingroup
\baselineskip 8pt
\def\arraystretch{0.8}
\noindent
\begin{picture}(370,155)(25,0)
\put(35,150){$\e\!=\!6$}
\put(345,150){$\e\!=\!4$}
\put(255,0){$\e\!=\!3$}
\pic11(100,38){{\bf D}_{\bf 12}}{g,h}{3\ast}{\II}{1}
\pic12(030,70){\Cy_6}{h}{1}{\II}{1}
\pic13(065,70){\Sym_3^a}{g,h^2}{3\ast}{\IV}{1(3)}
\pic14(126,70){\Sym_3^b}{gh,h^2}{3\ast}{\IV}{3(1)}
\pic15(180,81){\Vi}{g,h^3}{3\ast}{\IZS}{?}
\pic16(030,100){\Cy_3}{h^2}{1}{\IV}{3}
\pic17(077,115){\Cy_2^a}{g}{2}{\IZ}{1}
\pic18(130,115){\Cy_2^b}{gh^3}{6}{\IZ}{1}
\pic19(180,115){\Cy_2}{h^3}{1}{\IZS}{\scriptstyle\square}
\pic20(100,147){1}{}{3}{\IZ}{1}

\pic21(300,55){{\bf D}_{\bf 8}}{g,h}{2*}{\III}{2}
\pic22(245,85){\Vi^a}{g,h^2}{2*}{\IZS}{2(4)}
\pic23(300,85){\Cy_4}{h}{1}{\III}{2}
\pic24(340,85){\Vi^b}{gh,h^2}{2*}{\IZS}{4(2)}
\pic25(253,115){\Cy_2^a}{g}{1}{\IZ}{1}
\pic26(300,115){\Cy_2}{h^2}{1}{\IZS}{4}
\pic27(347,115){\Cy_2^b}{gh}{2}{\IZ}{1}
\pic28(300,145){1}{}{1}{\IZ}{1}

\pic31(210,5){{\bf D}_{\bf 6}}{g,h}{6*}{\IV}{?}
\pic32(170,20){\Cy_3}{h}{1}{\IV}{3}
\pic33(253,30){\Cy_2}{g}{2}{\IZ}{1}
\pic34(210,45){1}{}{3}{\IZ}{1}

\addline 11,12 raise=-4;
\put(100.00,44.86){\line(-5,2){44.73}}
\addline 11,13;
\put(107.35,51.74){\line(-1,1){10.00}}
\addline 11,14;
\put(127.85,51.74){\line(1,1){10.00}}
\addline 11,15 raise=-8;
\put(135.96,41.87){\line(5,3){51.44}}
\addline 12,16;
\put(43.40,83.74){\line(0,1){8.00}}
\addline 12,19 tilt=4 raise=0.5;
\put(46.48,83.74){\line(6,1){138.00}}
\addline 13,16;
\put(71.45,83.74){\line(-3,2){13.11}}
\addline 13,17;
\put(88.40,83.74){\line(0,1){23.00}}
\addline 14,16;
\put(126.00,80.22){\line(-4,1){67.66}}
\addline 14,18;
\put(146.80,83.74){\line(0,1){23.00}}
\addline 15,17;
\put(180.00,88.28){\line(-3,1){76.86}}
\addline 15,18;
\put(180.00,94.42){\line(-3,2){18.48}}
\addline 15,19;
\put(194.75,94.74){\line(0,1){12.00}}
\addline 16,20 raise=5;
\put(52.37,113.74){\line(3,2){47.63}}
\addline 17,20;
\put(97.82,128.74){\line(3,4){7.50}}
\addline 18,20;
\put(134.47,128.74){\line(-1,1){10.00}}
\addline 19,20 raise=3;
\put(180.00,126.22){\line(-5,2){53.86}}
\addline 21,22;
\put(300.00,67.37){\line(-5,3){15.61}}
\addline 21,23;
\put(315.35,68.74){\line(0,1){8.00}}
\addline 21,24;
\put(332.56,68.74){\line(3,2){12.00}}
\addline 22,25;
\put(265.89,98.74){\line(0,1){8.00}}
\addline 22,26;
\put(283.04,98.74){\line(5,3){16.96}}
\addline 23,26;
\put(314.00,98.74){\line(0,1){8.00}}
\addline 24,26;
\put(343.21,98.74){\line(-3,2){15.79}}
\addline 24,27;
\put(360.75,98.74){\line(0,1){8.00}}
\addline 25,28;
\put(279.14,125.78){\line(3,2){20.86}}
\addline 26,28;
\put(313.39,128.74){\line(0,1){8.00}}
\addline 27,28;
\put(347.00,126.70){\line(-5,3){20.86}}
\addline 31,32;
\put(210.00,13.63){\line(-3,1){11.66}}
\addline 31,33;
\put(242.35,17.98){\line(5,3){10.65}}
\addline 32,34;
\put(198.34,31.72){\line(3,2){11.66}}
\addline 33,34;
\put(253.00,37.43){\line(-3,1){16.86}}
\end{picture}
\endgroup
\vskip 3mm

\end{table}%
in the top row is $H$ and its generators;
in the bottom left corner $a(H)$, followed by $*$ when $\dim V^H$ is odd;
in the bottom right corner the Kodaira symbol and $c_v$ for $E/L^H$.
The functions $a$ and $d$ are elementary to compute, and
$c_v$ come from Lemma \ref{lemcv}; $\square$ denotes a square value, $?$
an undetermined value, $\Vi=\Cy_2\!\times\!\Cy_2$.
The entries $1(3)$ and $3(1)$ for the Tamagawa numbers in the table for $\e=6$
mean that there are actually two tables, one with $1$ and $3$ and one with
$3$ and $1$. (Similarly for $2(4)$ and $4(2)$ when $\e=4$.)
There are also identical tables, but with
\II, \III, \IV{} replaced everywhere by \IIS, \IIIS{} and \IVS, respectively.

\begin{remark}
\label{3pictures}
Note from the pictures that $c_v\sim a$ in $\Gal(L/K)$-relations
in Case~D; see Example \ref{exdih} for the list of relations. This is
also true in Case~C, as $\Gal(L/K)$ is cyclic and has no relations.
Now, $a$ is a function lifted from $\Gal(L/K)$. If $c_v$ were such a function
as well, we would have $c_v\sim a$ in general by Theorem \ref{l2:}q.
Unfortunately, life is not that simple, and we will use the full
force of the machinery in \S\ref{smachinery} to establish the result.
\end{remark}

\begin{proposition}
\label{propcva}
$c_v\sim a$.
\end{proposition}

\begin{proof}
We proceed as follows

\Step{1}{Reduction to $D=\Cy_{\e}\rightsemidirect \Cy_{2^k}$}

We claim that $a$ and $c_v$ are both lifted from $\Gal(L^u/K)$,
where we use~${}^u$~to denote the maximal unramified extension in $F$.
Then by Theorem \ref{l2:}q we may replace $F$ by $L^u$,
and we will be left with the case $e_{F/K}=\e$ and
$D=\Cy_{\e}\rightsemidirect \Cy_{2^k}$.

That $a$ lifts from $\Gal(L^u/K)$ is clear. In view of Lemma \ref{lemcv},
to see that $c_v$ has this property it is enough to check that
for every intermediate field $M$ of $F/K$, the extension $M/M\cap L^u$ is
totally ramified with $\gcd(e_{M/K},\e)=\gcd(e_{M\cap L^u/K},\e)$.
By Lemma \ref{blackmagic0}, there is a subfield $K\subset M_\e\subset M$ with
$M/M_\e$ totally ramified and $e_{M_\e/K}=\gcd(e_{M/K},\e)$; so it
suffices to show that $L^u\cap M$ contains $M_\e$,
equivalently that $M_\e\subset L^u$.
But $M_\e^u/K^u$ and $L^u/K^u$ sit inside the (cyclic) maximal tame extension
$F^W/K^u$, so $M_\e^u\subset L^u$ by comparison of degrees.

We now have
$$
  D\>=\>\Cy_{\e}\rightsemidirect \Cy_{2^k}\>=\>\langle x,y\>|\>x^\e=y^{2^k}=1, yxy^{-1}=x^{\pm 1} \rangle
$$
with $x^{+1}$ in Case 3C and $x^{-1}$ in Case 3D,
$I=\Cy_{\e}=\langle x\rangle$ and $\e=2,3,4,6$.

\Step{2}{Proof for $\e=2,3,6$}
Suppose $\e=2, 3$ or $6$. 

By Examples \ref{exoff1}, \ref{exoff2} (note that $4\nmid\e$),
it is enough to prove that
$c_v/a$ is trivial on relations of all the subquotients $H_t/N_t$ where
$$
  H_t=\langle x, y^{2^{t}}\rangle, \qquad N_t = \langle y^{2^{t+1}}\rangle.
$$
In Case 3D, the $t\!=\!0$ quotient is $D/D'\iso \Di_{2\e}$,
where $c_v/a$ does cancel in relations by Remark \ref{3pictures}.
We consider the remaining subquotients \hbox{($\iso \Cy_{\e}\!\times\!\Cy_2$)}
in Cases C and D, according to the value of $\e$.
If $\e=3$, these are cyclic and have no relations, so the result is trivial.
If $\e=6$, $H_t/N_t$ has the following lattice of subgroups.
Here we specify the group name, its generators
($h$ is the image of $x$, and $g$ a suitable element of order $2$),
the value of $a$ ($X\!=\!1$~in Case~3C, and $X\!=\!3$ in Case 3D),
the reduction type and
Tamagawa numbers over the corresponding fields:

\stepcounter{equation}
\begin{table}[h]
\begingroup
\baselineskip 8pt
\def\arraystretch{0.8}
\noindent
\begin{picture}(340,125)(30,-16)
\pic1(242,-04){{\bf C_2\times C_6}}{g,h}{1}{\II}{1}
\pic2(305,027){\Cy_6^c}{gh^2}{1}{\IV}{1(3)}
\pic3(242,027){\Cy_6^b}{gh}{1}{\IV}{3(1)}
\pic4(200,027){\Cy_6^a}{h}{1}{\II}{1}
\pic5(255,060){\Cy_3}{h^2}{1}{\IV}{3}
\pic6(120,035){\Vi}{g,h^3}{1}{\IZS}{?}
\pic7(170,068){\Cy_2^c}{g}{X}{\IZ}{1}
\pic8(120,068){\Cy_2^b}{gh^3}{1}{\IZ}{1}
\pic9(070,068){\Cy_2^a}{h^3}{1}{\IZS}{\square}
\pic10(120,098){1}{}{X}{\IZ}{1}
\addline 1,2;
\put(286.02,9.74){\line(2,1){18.98}}
\addline 1,3 tilt=2;
\put(263.52,9.74){\line(0,1){9.00}}
\addline 1,4;
\put(246.53,9.74){\line(-5,3){21.26}}
\addline 2,5;
\put(306.05,40.74){\line(-5,3){22.71}}
\addline 3,5;
\put(264.01,40.74){\line(1,4){2.75}}
\addline 4,5;
\put(225.27,36.86){\line(5,3){29.73}}
\addline 6,7;
\put(150.29,47.80){\line(3,2){19.71}}
\addline 6,8;
\put(135.51,48.74){\line(0,1){11.00}}
\addline 6,9;
\put(120.00,47.40){\line(-3,2){21.33}}
\addline 7,10;
\put(170.00,79.59){\line(-5,3){20.50}}
\addline 8,10;
\put(135.31,81.74){\line(0,1){8.00}}
\addline 9,10;
\put(98.67,79.21){\line(5,3){21.33}}
\addline 1,6;
\put(242.00,4.33){\line(-3,1){91.71}}
\addline 2,7;
\put(305.00,33.46){\line(-3,1){105.50}}
\addline 3,8;
\put(242.00,35.82){\line(-3,1){90.25}}
\addline 4,9;
\put(200.00,33.07){\line(-3,1){101.33}}
\addline 5,10;
\put(255.00,68.48){\line(-4,1){105.50}}
\end{picture}
\endgroup
\caption{$\Cy_6\!\times\!\Cy_2\,$-$\,$subquotient}
\label{c6c2table}
\vskip-20pt
\end{table}

\noindent The lattice of relations is generated by
$\Cy_3-\Cy_6^a-\Cy_6^b-\Cy_6^c+2(\Cy_6\!\times\!\Cy_2)$
and
$1-\Cy_2^a-\Cy_2^b-\Cy_2^c+2\Vi$, on each of which
$a/c_v$ is trivial by inspection.
Finally if $\e=2$, the subquotients are isomorphic to $\Cy_2\times \Cy_2$,
and the data for their subgroups is the same as for the subgroups of
$\Vi\subset\Cy_6\!\times\!\Cy_2$ in Table~\ref{c6c2table} with $X\!=\!1$.
Once again, $a/c_v$ is trivial on the unique $\Cy_2\!\times\!\Cy_2$ relation, which
completes the proof that $c_v\sim a$ for $\e\ne 4$.

\Step{3}{Proof for $\e=4$}
Suppose $\e=4$ and we are in Case 3D.
By Example \ref{exoff2} and the proof of Lemma \ref{loffchop}(1),
every $D$-relation is a sum of one lifted from $D/D'\iso \Di_8$
and a $D$-relation with terms in $U=\langle x, y^2\rangle$.
By Remark \ref{3pictures}, $c_v/a$ cancels in relations of $D/D'$, so it
suffices to prove cancellation in $D$-relations whose terms lie in $U$.
Subgroups of $U$ project to subgroups of $\Cy_4$ in $D/D'$, so
$a=1$ for these (see Table \ref{dihtable}). Hence
it is enough to show that the following function
cancels in $D$-relations with terms in $U$:
$$
  \tilde c_v(H) = c_v(E/F^H) \cdot \left\{\begin{array}{ll}
     |H|^2, & H \subset \langle y\rangle,\cr
     4, & \langle x^2 \rangle \subset H \subset \langle x^2,y\rangle,\cr
     1, & \text{otherwise}.\cr
  \end{array}\right.
$$
As $\langle x^2\rangle$, $\langle y\rangle$ and $\langle x^2,y\rangle$
are normal in $D$, the ``correction terms'' are the same for conjugate
subgroups, so this is a function on the Burnside ring of~$D$.
Exceptionally,
we view $\tilde c_v$ as a function to $\R_{>0}$, {\em not\/} (!) to $\Q^{\times}/\Q^{\times2}$.
The point is that now it suffices to check that it cancels in $U$-relations,
since a multiple of a $D$-relation with terms in $U$ is an $U$-relation
by Lemma \ref{loffchop}(1),
and taking roots is perfectly ok in $\R_{>0}$.

By Example \ref{exoff1}, $U$-relations are generated by those coming from
the subquotients $H_t/N_t$ ($1\le t\le k-2$) with
$$
  H_t=\langle x, y^{2^t}\rangle, \qquad N_t = \langle y^{2^{t+2}}\rangle,
  \qquad H_t/N_t\,\iso\,\Cy_4\!\times\!\Cy_4\>.
$$
In such a subquotient, let $g$ and $h$ be the images of $x$ and $y^{2^t}$
respectively.
Note that the field $F^{H_t}(\sqrt{\Delta_E})$
corresponds to $\langle g^2,h\rangle$.
(By Lemma \ref{lemcv}, it distinguishes between $c_v=2$ and $c_v=4$
in the \IZS-case; $c_v=1$ cannot occur because $E/K$ has $c_v=2$ and so has
non-trivial 2-torsion in every extension of $K$, see Remark \ref{cv=tors}.)
The subquotient $H_t/N_t$ has a basis of 5 relations, given below with
the corresponding $\tilde c_v$:
\begingroup
\smaller[1]
\def\fourkt{(4^{\text{\tiny $k$-$t$}}){}}
\def\fourktm#1{(4^{\text{\tiny $k$-$t$-$#1$}}){}}
\def\powm{^{-1}}
\def\powm{^{\text{\tiny--}1}}
\beq
\langle g^2\rangle \!-\!\langle g h^2\rangle \!-\!\langle h^2,g^2\rangle \!-\!\langle g\rangle \!+\!2\langle g,h^2\rangle  & \to & 4^{1} 2\powm 4\powm 2\powm 2^{2} =1\cr
\langle h^2,g^2\rangle \!-\!\langle g^2,h\rangle \!-\!\langle g^2,g h\rangle \!-\!\langle g,h^2\rangle \!+\!2\langle g,h\rangle  & \to & 4^{1} 4\powm 2\powm 2\powm 2^{2} =1\cr
\langle h^2 g^2\rangle \!-\!\langle g^3 h\rangle \!-\!\langle h^2,g^2\rangle \!-\!\langle g h\rangle \!+\!2\langle g^2,g h\rangle  & \to & 1^{1} 1\powm 4\powm 1\powm 2^{2} =1\cr
\{1\} \!-\!\langle g^2\rangle \!-\!\langle h^2 g^2\rangle \!-\!\langle h^2\rangle \!+\!2\langle h^2,g^2\rangle  & \to & \fourktm2^{1} 4\powm 1\powm \fourktm1\powm 4^{2} =1\cr
\langle h^2\rangle \!-\!\langle h^2,g^2\rangle \!-\!\langle h\rangle \!-\!\langle g^2 h\rangle \!+\!2\langle g^2,h\rangle  & \to & \fourktm1^{1} 4\powm \fourkt\powm 1\powm 4^{2} =1.\cr
\eeq
\endgroup
As the $\tilde c_v$ cancel in all relations, this proves the claim in Case 3D.

Finally, in Case 3C the $a$ and the $c_v$ are the same as for the subgroup
$H_1$ of Case 3D, so they again cancel in relations.
\end{proof}

\begin{proposition}
\label{omega1}
In Case 3C, $\omega\sim 1$. In Case 3D,
$\omega\sim \tGDI{I,W,\smallchoice{e,}{\e\nmid f}{1,}{\e|f}}$.
\end{proposition}

\begin{proof}
As before, write $\delta$ for the valuation of the minimal discriminant
of $E/K$ and $q$ for the size of the residue field of $K$.
So $q\equiv (-1)^t\mod\e$ with $t=0$ in Case 3C and $t=1$ in Case 3D.
If $q$ is an even power of the residue characteristic $l$, then $t=0$
and $\omega=q^{\ldots}=1\in\Q^{\times}/\Q^{\times2}$ as asserted.
Suppose from now on that $q$ is an odd power of $l$, so
$$
   \omega = \GDI{D,I,q^{\tw{\delta e}f}} = \GDI{D,I,l^{\tw{\delta e}f}}
   \sequald \GDI{I,I,l^{\tw{\delta e}}}
   \sequalr \GDI{I,W,l^{\tw{\delta ef}}}.
$$
The order of $W$ is a power of $l$, so let us define $k$ by
$e=l^k$.

Define $n\in\{0,1\}$ by
$$
  l \delta \equiv q \delta \equiv (-1)^t \delta + 12n \mod 24,
$$
so
$$
  l^k \delta \equiv (-1)^{tk} \delta + 12nk \mod 24.
$$
Then
$$
  \omega\sequal
  \GDI{I,W,l^{\tw{\delta l^kf}}} \sequal
  \GDI{I,W,l^{\lfloor\frac{(-1)^{tk}\delta f}{12}+fkn\rfloor}} \sequal
  \GDI{I,W,l^{\tw{(-1)^{tk}\delta f}}} \GDI{I,W,e^{fn}}.
$$
The second term is trivial:
$$
  \GDI{I,W,e^{fn}} \sequald \GDI{W,W,e^n} \simwwe 1.
$$
As for the first term, for $t=0$ it is a function of $f$ so it is $\sim 1$
by \hbox{Theorem~\ref{l2:}f}, as required.
Finally, $\tw{m}\equiv\tw{-m}\mod 2$ if and only if~$12|m$, so
for $t=1$ we have
$$
  \GDI{I,W,l^{\tw{(-1)^{k}\delta f}}} \sequal
  \GDI{I,W,l^{\tw{\delta f}}}
  \bGDI{I,W,{\Bigl\{\!\textstyle\tiny\begin{array}{ll}
    \!\!1, & 2|k \cr
    \!\!1, & 12|\delta f \cr
    \!\!l, & \text{else}
  \end{array}}\!\!\Bigr\}}
  \simf
  \bGDI{I,W,{\Bigl\{\!\textstyle\tiny\begin{array}{ll}
    \!\!1, & 2|k \cr
    \!\!1, & 12|\delta f \cr
    \!\!l, & \text{else}
  \end{array}}\!\!\Bigr\}}
$$
$$
  \sequal
  \bGDI{I,W,{\Bigl\{\!\textstyle\tiny\begin{array}{ll}
    \!\!l^k, & 2|k \cr
    \!\!1,   & \e|f \cr
    \!\!l^k, & \text{else}
  \end{array}}\!\!\Bigr\}}
  \sequal
  \GDI{I,W,\smallchoice{e,}{\e\nmid f}{1,}{\e|f}},
$$
as asserted.
\end{proof}

\begin{proposition}
\label{omega1.5}
In Case 3C, $d\sim 1$. In Case 3D,
$d\sim \tGDI{I,W,\smallchoice{e,}{\e\nmid f}{1,}{\e|f}}$.
\end{proposition}

\begin{proof}

In Case 3C, the module $V$ is zero and the result is trivial, so suppose we
are in Case 3D.
By definition, $d(H)$ is either 1 or $|H|$ (up to squares),
depending
on the intersection of $F^H$ with the dihedral extension $L/K$. We
\newpage
\noindent
may replace $|H|$ by $[D:H]=e_{F^H/K}f_{F^H/K}$
by Lemma \ref{dimsaddup}. 
Inspecting $\ast$ in Table \ref{dihtable}, we find that for $H\supset D'$
(i.e. corresponding to subfields of $L$),
$$
  (ef)^{\dim V^H} \equiv {\Bigl\{\!\textstyle\small\begin{array}{ll}
    \!\!1,  & 2|f \text{\ or}\>\e|e \cr
    \!\!ef,  & \text{else}
  \end{array}}
  \!\!\!\mod\Q^{\times2}
  \qquad (\text{with\ }e\!=\!e_{F^H/K},\>\>f\!=\!f_{F^H/K}).
$$
By Lemmas \ref{blackmagic0}, \ref{blackmagic1},
the condition ``$2|f \text{\ or}\>\e|e$'' holds for a general $F^H$
if and only if it holds for $F^H\cap L$. Therefore
$$
  d \>\>\sim\>\> (H\mapsto (e_{F^H/K}f_{F^H/K})^{\dim V^H})
  \>\>\sequal\>\>
  \bGDI{D,I,{\Bigl\{\!\textstyle\tiny\begin{array}{ll}
    \!\!1,  & 2|f \cr
    \!\!1,  & \e|e \cr
    \!\!ef,  & \text{else}
  \end{array}}\!\!\Bigr\}}.
$$
Recall that $2,3\nmid|W|$ and that $[D:I]$ is a power of 2
(Lemma \ref{DI2}). Thus
\beq
  \bGDI{D,I,{\Bigl\{\!\textstyle\tiny\begin{array}{ll}
    \!\!1,  & f\!\!\ne\!\!1 \cr
    \!\!1,  & \e|e \cr
    \!\!e,  & \text{else}
  \end{array}}\!\!\Bigr\}}\!\!\!\!\!
    &\sequald\!\!\!\!\!&
  \bGDI{I,I,\choice{1}{\e|e}{e}{\e\nmid e}}
\cr
  &\sequalr\!\!\!\!\!&
  \bGDI{I,W,\choice{1}{\e|f}{ef}{\e\nmid f}}
  \simf
  \bGDI{I,W,\choice{1}{\e|f}{e}{\e\nmid f}}\>.
\eeq
\vskip -2mm
\end{proof}

Combining Propositions \ref{propcva}, \ref{omega1} and \ref{omega1.5}
proves Proposition \ref{propcvdv} in Cases (3C) and (3D).

\subsubsection{(Case 3M) $E/K$ has potentially multiplicative reduction}

In view of Lemma~\ref{DI2}, it suffices to prove the following

\begin{claim}
Suppose $f_{F/K}$ is a power of 2.
Then $\omega\sim 1\sim a$ and $c_v\sim d$, and hence $C_v\sim\D_V$.
\end{claim}

\begin{proposition}
\label{omega2}
$\omega\sim 1\sim a$.
\end{proposition}
\begin{proof}
As $V$ is a $\Cy_2$-representation and $\Cy_2$ has no relations, $a\sim 1$
by \hbox{Theorem}~\ref{l2:}q.
Also $\omega=\tGDI{D,I,q^{\lfloor\frac e2\rfloor f}}\sim 1$,
by the proof of $\e=2, \delta=6$ case
of Proposition \ref{omega1}.
\end{proof}

\begin{proposition}
$c_v\sim d$.
\end{proposition}

\begin{proof}

For $H\< D$, the expression for $c_v(E/F^H)$ is given in
Lemma \ref{lemcv} (note that $L=K(\sqrt{-6B})$ in its notation).
Writing $e=e_{F^H/K}$, $f=f_{F^H/K}$, we have
$$
\begin{array}{l|cccc}
  & E/F^H & c_v & d & c_v/d \cr
\hline
L\subset F^H\vphantom{\int^X}       & \In{ne}\text{ split}     & ne       & \frac{{\scriptscriptstyle|D|}}{ef} & \frac{n}{{\scriptscriptstyle|D|}}e^2f  \cr
L\notsubset F^H, 2|e                & \In{ne}\text{ non-split}  & 2        & 1  & 2         \cr
L\notsubset F^H, 2\nmid e, f\!=\!1  & \InS{ne}                 & c_v(E/K) & 1  & c_v(E/K)  \cr
L\notsubset F^H, 2\nmid e, 2|f      & \InS{ne}                 & 4        & 1  & 4         \cr
\end{array}
$$
Write $L^u$ and $K^u$ for the maximal unramified extensions of $L$ and
$K$ \hbox{inside}~$F$,
respectively.
By Lemma \ref{blackmagic0}, the function $c_v/d$ (to $\Q^{\times}/\Q^{\times2}$) factors through
$$
  \Gal(L^u/K) \,=\, \Gal(L^u/K^u)\times\Gal(L^u/L) \,=\, \Cy_2 \times \Cy_{2^k}.
$$
By Theorem \ref{l2:}q, it suffices to prove that $c_v/d$
is trivial on $\Gal(L^u/K)$-relations.
A list of generating relations is given in
Example \ref{exoff1} (with $u\!=\!1$ and $m\!=\!1$).
These come from $\Cy_2\!\times\!\Cy_2$\hskip 1pt-\hskip 1pt subquotients,
each with one relation (Example \ref{exc2c2})
and it is elementary to verify that $c_v/d$ is trivial on these.
\end{proof}

\subsubsection{Appendix: Local lemmas}

\begin{lemma}
\label{lemcv}
Let $K'/K/\Q_l$ be finite extensions with $l\!\ge\!5$,
and let $E/K$ be an elliptic curve with additive reduction,
$$
  E: y^2 = x^3 + A x + B, \qquad A,B\in K.
$$
Write $\Delta=-16(4A^3+27B^2)$ for the discriminant of this model,
$\delt$ for its $K$-valuation, and $e\!=\!e_{K'/K}$.

If $E$ has potentially good reduction,
then
\beq
\gcd(\delt e,12)=2 & \implies & c_v(E/K')=1 & (\II,\IIS{})\cr
\gcd(\delt e,12)=3 & \implies & c_v(E/K')=2 & (\III,\IIIS)\cr
\gcd(\delt e,12)=4 & \implies & c_v(E/K')=\leftchoice{1}{\sqrt{B}\not\in K'}{3}{\sqrt{B}\in K'} & (\IV,\IVS)\cr
\gcd(\delt e,12)=6 & \implies & c_v(E/K')=\leftchoice{2}{\sqrt{\Delta}\not\in K'}{1\>\text{\rm or}\>4}{\sqrt{\Delta}\in K'} & (\IZS)\cr
\gcd(\delt e,12)=12 & \implies & c_v(E/K')=1 & (\IZ).\cr
\eeq
The extension $K'(\sqrt{B})/K'$ in the \IV, \IVS{} cases and
$K'(\sqrt{\Delta})/K'$ in the \IZS{} case is unramified.
In particular if $K''/K'$ has odd residue degree and
$\gcd(\delt e_{K'/K},12)=\gcd(\delt e_{K''/K},12)$, then $c_v(E/K')=c_v(E/K'')$.

If $E$ has potentially multiplicative reduction of type \InS n over $K$, then
\beq
  2\nmid e, 2\nmid n & \implies &
    c_v(E/K')=\leftchoice{2}{\sqrt{B}\not\in K'}{4}{\sqrt{B}\in K'} & (\InS{ne})\cr
  2\nmid e, 2|n & \implies &
    c_v(E/K')=\leftchoice{2}{\sqrt{\Delta}\not\in K'}{4}{\sqrt{\Delta}\in K'} & (\InS{ne})\cr
  2|e, \sqrt{-6B}\not\in K'  & \implies & c_v(E/K')=2 & (\text{\In{ne} non-split})\cr
  2|e, \sqrt{-6B}\in K'  & \implies & c_v(E/K')=ne & (\text{\In{ne} split}).\cr
\eeq
The extension $K'(\sqrt{B})/K'$, $K'(\sqrt{\Delta})/K'$ or $K'(\sqrt{-6B})/K'$
corresponding to the case is unramified.
\end{lemma}

\begin{proof}
This follows from Tate's algorithm (\cite{Sil2}, IV.9).
\end{proof}

\begin{remark}
\label{cv=tors}
Let $K$ be a finite extension of $\Q_l$ with $l\ge 5$, and $E/K$
an elliptic curve.
Recall that $c_v(E/K)$ is the size of the N\'eron component group
$E(K)/E_0(K)$. When $E/K$ has additive reduction, this group
is of order at most 4, and is isomorphic to the prime-to-$l$ torsion
in $E(K)$. (Use the standard exact sequence
\cite{Sil1} VII.2.1 and note that multiplication-by-$l$ is an isomorphism
both on the formal group of $E$ and on the reduced curve.)
In particular, in the $\IZS$ case of the lemma,
$c_v(E/K')=1$ if and only if $E/K'$ has trivial 2-torsion.
\end{remark}

\begin{lemma}
\label{blackmagic0}
Suppose $L/K/\Q_l$ are finite extensions. For every divisor
$m|e_{L/K}$ with $l\nmid m$,
there exists a subfield $M$ of $L/K$
with $e_{M/K}=m$ and $L/M$ totally ramified.
\end{lemma}
\begin{proof}
Replacing $K$ by its maximal unramified extension in $L$,
we may assume that $L/K$ is totally ramified. Let $F$ be the Galois closure
of $L/K$, and write $G=\Gal(F/K)$, $H=\Gal(F/L)$ and $I$ for the inertia
subgroup of $G$.
Let $N$ be the unique index $m$ subgroup of $I$. We claim that
$M\!=\!F^{NH}$~will~do.

Since $e_{M/K}=[I:I\cap NH]$ and $[I:N]=m$,
it is enough to prove that $N=I\cap NH$. ``$\subset$'' is clear.
Observe that every subgroup $U$ of $I$ whose index is divisible by $m$ is
contained in $N$. (This is clear if $U$ contains the wild inertia
subgroup $W\normal I$, as $I/W$ is cyclic; otherwise replace $U$ by $UW$.)
In particular, $m|e_{L/K}=[I:I\cap H]$
implies that $I\cap H\subset N$. Since $N$ is characteristic in $I\normal G$
and therefore normal in $G$, it follows that $I\cap NH \subset N$ as asserted.
\end{proof}

\begin{lemma}
\label{blackmagic1}
Suppose $K/\Q_l$ is finite, $\e=2,3,4,6$, and the size of the residue field
of $K$ is congruent to $-1$ modulo $\e$.
Then the compositum $F$ of all
totally ramified extensions of $K$ of degree $\e$ is a dihedral extension of
degree $2\e$. Specifically, $F=K'(\sqrt[\e]{\pi})$ with $\pi$ a uniformiser
of $K$ and $K'/K$ quadratic unramified.
\end{lemma}

\begin{proof}

For $\e=2$ this is elementary. Otherwise, $K(\mu_\e)=K'$ and
it suffices to prove that every totally ramified degree $\e$ extension
of $K$ is contained in $K'(\sqrt[\e]{\pi})$.

For $\e=3$ suppose $L/K$ is cubic totally ramified.
It cannot be
Galois by local class field theory, since the units of $K$ have no index
3 subgroups. So its Galois closure is an $\Sym_3$-extension,
which is tame and so contains $K'=K(\mu_3)$. By Kummer theory, $LK'/K'$
is contained in the $\Cy_3\!\times\!\Cy_3$\hskip 1pt-\hskip 1pt extension $M$ of $K'$
obtained by adjoining cube roots of all elements of $K'$.
But it is easy to see that $\Gal(M/K)\iso \Cy_6\leftsemidirect\Cy_3$
has a unique $\Sym_3$-quotient, so $LK'=K'(\sqrt[\e]{\pi})$ and
$L\subset K'(\sqrt[\e]{\pi})$ as asserted.

For $\e=4$, every totally ramified quartic extension of $K$ has a
quadratic subfield by Lemma \ref{blackmagic0}, so there are at most 4 of them by
the $\e=2$ case. Since in this case $K'(\sqrt[\e]{\pi})$ has 4 totally ramified subfields
corresponding to the non-normal subgroups of order 2 in $\Di_8$, they are all
contained in it.

For $\e=6$ the assertion follows from the $\e=2$ and $\e=3$ cases
(apply Lemma~\ref{blackmagic0} for $m=2$ and $m=3$).
\end{proof}

\subsection{Case (4): Semistable abelian varieties}
\label{ss-case4}

\subsubsection{Review of abelian varieties with semistable reduction}
\label{sssavreview}

Let $K$ be a finite extension of $\Q_l$, let $A/K$ be an abelian variety
and take a prime $p\ne l$.
Fix a finite Galois extension $L/K$ where $A$ acquires split
semistable reduction.
By the work of Raynaud (\cite{Ray}, \cite{Gro} \S9),
there is a smooth commutative group scheme $\cA/\O_L$, which
is an extension
$$
  0 \lar \T \lar \cA \lar \B \lar 0,
$$
with $\T/\cO_L$ a split torus and $\B/\cO_L$ an abelian scheme,
and such that $\cA\tensor(\O_L/m_L^i)$ is the identity component
of the N\'eron model of $A$ over $\O_L$ base changed to $\O_L/m_L^i$.
These properties characterise $\cA$ up to a unique isomorphism.
In particular, the group $\Gal(L/K)$ acts naturally on $\T, \cA$ and~$\B$.
The character group of $\T$ is a finite free $\Z$-module
with an action of $\Gal(L/K)$, and we denote it $X(\T)$.

The dual abelian variety $A^t/K$ also has split semistable reduction over $L$,
and there is a sequence as above with $\T^*, \cA^*$ and
$\B^*\iso \B^t$ (\cite{Gro}~Thm.~5.4).
Raynaud constructs a canonical map $X(\T^*)\hookrightarrow\cA(L)$,
inducing an isomorphism of $\Gal(\bar K/K)$-modules
$$
  A(\bar K) \iso \cA(\bar K)/X(\T^*),
$$
which generalises Tate's parametrisation for elliptic curves.
(The $\Gal(\bar K/K)$-action on $\cA(\bar K)$ comes from the
Galois action of $\Gal(\bar K/L)$ and the geometric action of
$\Gal(L/K)$; see \cite{CFKS} \S2.9.)
 From this description, there are exact sequences
for the $p$-adic Tate modules of the generic fibres over $L$,
\beq
  0 \lar T_p(\T_L) \lar T_p(\cA_L) \lar T_p(\B_L) \lar 0 \cr
  0 \lar T_p(\cA_L) \lar T_p(A) \lar X(\T^*)\tensor\Z_p \lar 0. \cr
\eeq
In particular, $T_p(A)$ has a filtration with graded pieces
$$
  \Gr0 = X(\T^*)\tensor\Z_p, \quad
  \Gr1 = T_p(\B_L), \quad
  \Gr2 = \Hom(X(\T),\Z_p(1)).
$$

\smallskip

Now suppose $A/K$ has {\em semistable\/} reduction.
The reduction becomes split semistable over some finite unramified extension
of $K$, and we take $L$ to be the smallest such field;
so now $\Gal(L/K)$ is cyclic, generated by Frobenius.
To describe the Tamagawa number $c_v(A/K)$ and
the action of inertia on $T_p(A)$ we
use the monodromy pairing
$$
  X(\T^*) \times X(\T) \lar \Z.
$$
This is a non-degenerate $\Gal(L/K)$-invariant pairing, and induces
a Galois-equivariant inclusion of lattices
$$
  N: X(\T^*) \longinjects \Hom(X(\T),\Z).
$$
These have the same $\Z$-rank, so $N$ has finite cokernel.
Moreover, $N$ is covariantly functorial with respect to isogenies
of semistable abelian varieties. 
Any polarisation on $A$ gives a map
$X(\T^*) \to X(\T)$, and the induced pairing
$$
  X(\T^*) \times X(\T^*) \lar \Z
$$
is symmetric (\cite{Gro} \S10.2).
In particular, if $A$ is principally polarised, we get a perfect
Galois-equivariant symmetric pairing 
$$
  \coker N \times \coker N \lar \Q/\Z.
$$
If $K'/K$ is a finite extension, then $X(\T)$ and $X(\T^*)$
remain the same modules (restricted to $\Gal(LK'/K')\subset \Gal(L/K)$)
by uniqueness of Raynaud parametrisation.
The map $N$ becomes $e_{K'/K}N$, see \cite{Gro} 10.3.5.

The $\Gal(\bar K/K)$-module $\Gr2\oplus \Gr1\oplus \Gr0$ is unramified
and semisimple, so it is a semisimplification of $T_pA$.
With respect to this filtration, the inertia group acts on $T_pA$ by
$$
  I_{\bar K/K}\ni\sigma\quad\longmapsto\quad
  \begin{pmatrix}
     1 & 0 & t_p(\sigma)N \cr
     0 & 1 & 0 \cr
     0 & 0 & 1 \cr
  \end{pmatrix}
  \in\Aut T_p(A)
$$
with $t_p: I_{\bar K/K}\to \Z_p(1)$ defined by
$\sigma\mapsto \sigma(\pi_K^{1/p^n})/\pi_K^{1/p^n}\in\mu_{p^n}$
for any uniformiser $\pi_K$ of $K$
(\cite{Gro} \S\S 9.1--9.2).

Let $\grphi$ be the group scheme of connected components
of the special fibre of the N\'eron model of $A/\O_K$.
It is an \'etale group scheme over the residue field $k$ of $K$,
so
$\grphi(k)=\grphi(\bar k)^{\Gal(\bar k/k)}$
consists of components defined over $k$.
As $K$ is complete and $k$ is perfect, by \cite{BL} Lemma 2.1
the natural reduction map $A(K)\to\grphi(k)$ is onto, so
$c_v(A/K)=|\grphi(k)|$.
Finally, by \cite{Gro}~Thm.~11.5, $\grphi=\coker N$ as groups with
$\Gal(\bar k/k)=\Gal(K^{un}/K)$-action, so
$$
   c_v(A/K) = |(\coker N)^{\Gal(L/K)}|.
$$

\subsubsection{Local root numbers for twists of semistable abelian varities}

\begin{proposition}
\label{rootabtwist}
Suppose $A/K$ is semistable, let
$F/K$ be a finite Galois extension containing $L$,
and $\tau$ a complex representation of $\Gal(F/K)$. Then
$$
  w(A/K,\tau) =
    w(\tau)^{2\dim A} (-1)^{\blangle \tau,X(\T^*)\brangle}.
$$
\end{proposition}

\begin{proof}
Let $H^1(A)=H^1_{\text{\'et}}(A,\Z_p)\tensor_{\Z_p}\C=\Hom(T_pA\tensor\C,\C)$
for some $p\ne l$, and let $H^1(A)_{\text{ss}}$ be its semisimplification.
Write $V=X(\T^*)\tensor\Q$, and $\sgn z=z/|z|$ for $z\in\C^{\times}$.
By the unramified twist formula \cite{TatN} 3.4.6,
$$
  w((H^1(A)_{\text{ss}})\tensor\tau) =
    w(\tau)^{2\dim A} \sgn\det(\Fr|H^1(A)_{\text{ss}})^{\nu}
$$
for some integer $\nu$ and $\Fr=\Frob_{\bar K/K}^{-1}$.
Since $\det(H^1(A))$ is a power of the cyclotomic character,
this expression is just $w(\tau)^{2\dim A}$. By \cite{TatN} 4.2.4,
\beq
  w(A/K,\tau)=w(H^1(A)\tensor\tau) \cr
  = w((H^1(A)_{\text{ss}})\tensor\tau)
    \frac
    {\sgn\det(-\Fr| ((H^1(A)_{\text{ss}})\tensor\tau)^I )}
    {\sgn\det(-\Fr| ((H^1(A)\tensor\tau)^I )} \cr
  = w(\tau)^{2\dim A}\sgn\det(-\Fr| \Gr2^*\tensor\tau^I ) \cr
  = w(\tau)^{2\dim A}\sgn\det(-\Fr| V\tensor\tau^I ) \cr
  = w(\tau)^{2\dim A}
    (-1)^{\dim V\dim\tau^I}
    {\det(\Fr| V)^{\dim\tau^I}}
    {\det(\Fr| \tau^I)^{\dim V}}, \cr
\eeq
where the penultimate equality again comes from \cite{TatN} 3.4.6.
If $\eta$ denotes the unramified character $\Fr\mapsto -1$, we have
$$
  (-1)^{\dim V} = (-1)^{{\blangle \triv, V\brangle}+{\blangle \eta, V\brangle}},
    \qquad
  \det(\Fr|V)=(-1)^{\blangle \eta, V\brangle}
$$
and similarly for $\tau^I$ in place of $V$, as they are both self-dual and
unramified. Now a trivial computation shows that
$$
  (-1)^{\blangle \tau,V\brangle} =
  (-1)^{\blangle \tau^I,V\brangle} =
  (-1)^{\blangle \triv,\tau^I\brangle \blangle \triv, V\brangle +
        \blangle \eta,\tau^I\brangle \blangle \eta, V\brangle
        }
$$
coincides with $w(A/K,\tau)/w(\tau)^{2\dim A}$, as asserted.
\end{proof}

\subsubsection{Tamagawa numbers for semistable abelian varieties}

\begin{proposition}
\label{avlcv2}
Suppose $A/K$ is semistable and principally polarised,
and set $V=X(\T^*)\tensor\Q$.
Let $F/K$ be a finite Galois extension containing~$L$, and write
$D=\Gal(F/K)$, $I\normal D$ for its inertia subgroup and
$\Fr=\Frob_{L/K}^{-1}$.
As functions from the Burnside ring of $D$ to $\Q^{\times}/\Q^{\times2}$,
$$
  C_v = c_v \sim \GDI{D,I,e^{\dim V^{\Fr^f}}} \sim \D_V
$$
(up to factors of 2 in Case {\rm(4ex)}).
\end{proposition}

\begin{proof}
The N\'eron model of $A/\O_K$ commutes with base change as $A$ is semi\-stable,
so the minimal exterior form $\omega$ remains minimal in
extensions of $K$, and $C_v=c_v$.

As $V$ is unramified,
$$
  \D_V = (H \mapsto \det(\tfrac1{|H|}\lara|V^H))
           = \GDI{D,I,(\tfrac{|D|}{ef})^{\dim V^{\Fr^f}}\det(\lara|V^{\Fr^f})}
$$
which is $\sim\tGDI{D,I,e^{\dim V^{\Fr^f}}}$ by Theorem \ref{l2:}{\rm f}.
This proves the last $\sim$.

As explained above,
$c_v = \tGDI{D,I,(\coker eN)^{\Fr^f}}$.
So it remains to prove that $\tGDI{D,I,\phi(e,f)}\sim 1$, where
$$
  \phi(e,f) = \bigl| (\coker eN)^{\Fr^f} \bigr| \> e^{-\dim V^{\Fr^f}}
    \>\>\in\Q^{\times}/\Q^{\times2}\>.
$$

\begin{claim} The function $\phi$ satisfies:
\begin{enumerate}
\item $\phi(e,pf)=\phi(e,f)$ for $p$ odd.
\item $\phi(e,4f)=\phi(e,2f)$.
\item $\phi(e,2)=|\coker N|$.
\item $\phi(2^k e,f)=\phi(2^k,f)$ for $e$ odd.
\item $\phi(e,f)=2^{\lambda_{e,f}}\phi(1,f)$ for some $\lambda_{e,f}\in \Z$.
\end{enumerate}
\end{claim}

Before verifying the claim, let us use these properties to complete our proof.
Note that the asserted formula already holds up to multiples of 2 by (5) and
Theorem \ref{l2:}f.

Let $W\normal I$ be the wild inertia subgroup. Then
$$\begin{array}{rcl}
  \tGDI{D,I,\phi(e,f)}
     &\tbuildrel{(1,2)}\over{=}&
   \bGDI{D,I,\choice{\phi(e,2)}{2|f}{\phi(e,1)}{2\nmid f}}
     \>\tbuildrel{(3)}\over{=}\>
   \bGDI{D,I,\choice{|\coker N|}{2|f}{\phi(e,1)}{2\nmid f}} \cr
     &\simf&
   \bGDI{D,I,\choice{1}{2|f}{\phi(e,1)}{2\nmid f}}
    \cr
   &\simd&
   \tGDI{I,I,\phi(e,1)}  \simr
   \tGDI{I,W,\phi(ef,1)}.
\end{array}$$
If $K$ has odd residue characteristic,
then the wild inertia group $W$ has odd order and
$\tGDI{I,W,\phi(ef,1)}=\tGDI{I,W,\phi(f,1)}\sim 1$ by (4) and Theorem \ref{l2:}f.
Suppose $K$ has residue characteristic 2. Then $[I:W]$ is odd, and
$$
  \tGDI{I,W,\phi(ef,1)}\>\tbuildrel{(4)}\over{=}\>\tGDI{I,W,\phi(e,1)} \simd \tGDI{W,W,\phi(e,1)}.
$$
If $W$ is cyclic, this is $\sim 1$ as asserted, since $W$ has no relations.
If $2\nmid [L\!:\!K]$, then $\Fr$ and $\Fr^2$ generate the same group, so
$$
\tGDI{W,W,\phi(e,1)}=\tGDI{W,W,\phi(e,2)}
\>\tbuildrel{(3)}\over{=}\> \tGDI{W,W,|\coker N|} \simf 1\>.
$$
Otherwise we are in case (4ex) and we claim nothing about factors of 2.
\end{proof}

\begin{proof}[\refstepcounter{equation}{\bf \theequation.}\label{NMM}\,Proof of claim]

This is a purely module-theoretic statement about the $G$-modules
$X(\T^*)\!\injects\! \Hom(X(\T),\Z)$ and the monodromy pairing.
\hbox{Suppose}
$G\!=\!\langle \Fr\rangle$ is a finite cyclic group, and $N: M'\injects M$
is an inclusion of \hbox{$\Z G$-lattices} of the same (finite) rank.
Furthermore, suppose for every $e\ge 1$ there is a perfect symmetric
$G$-invariant pairing
$$
  \langle,\rangle_e:\>\> M/eM' \times M/eM' \lar \Q/\Z.
$$
Then the function
$$
  \phi(e,f) = \bigl| (\coker eN)^{\Fr^f} \bigr| \> e^{-\rk M^{\Fr^f}}
    \>\>\in\Q^{\times}/\Q^{\times2}
$$
satisfies (1)--(5):

(1)
As $M\tensor\Q$ is a self-dual representation (every rational representation
is self-dual),
$\rk M^{\Fr^f}\!\equiv\!\rk M^{\Fr^{pf}}\!\!\mod 2$.
Write $U\!=\!M/eM'$ and $Z\!=\!\Fr^f$. Then
$$
  (m,n) = \langle Z m - Z^{-1} m, n \rangle_e
$$
defines an alternating pairing $U^{Z^p} \times U^{Z^p} \lar \Q/\Z$
whose left kernel is $U^{Z^2}\cap U^{Z^p}=U^Z$. So it is a perfect
alternating pairing on $U^{Z^p}/U^Z$, hence this group has square order.

(2)
Again $M\otimes\Q$ is self-dual, so
$\rk M^{\Fr^{4f}}\!\!\equiv\! \rk M^{\Fr^{2f}}\!\!\mod 2$.
Next, the above formula for $(,)$ with $Z\!=\!\Fr^f$ defines an
alternating pairing
$U^{\Fr^{4f}}\!\times\! U^{\Fr^{4f}}\! \to\! \Q/\Z$
whose left kernel is $U^{Z^2}\cap U^{Z^4}=U^{Z^2}$, so
$U^{Z^4}/U^{Z^2}$ has square order.

(3)
By (1) and (2), $\phi(e,2)=\phi(e,|G|)=|M/eM'|e^{\rk M}=|M/M'|$.

(4)
Replacing $\Fr$ by $\Fr^f$ we may assume $f=1$.
We may also suppose that $\Fr$ has order a power of 2 by (1).
It suffices to show that in the exact sequence
$$
  0 \lar \Bigl(\frac{{2^kM'}}{{2^keM'}}\Bigr)^\Fr \lar
  \Bigl(\frac{{M}}{{2^keM'}}\Bigr)^\Fr \lar
  \Bigl(\frac{{M}}{{2^kM'}}\Bigr)^\Fr \lar
  H^1\Bigl(G, \frac{{2^kM'}}{{2^keM'}}\Bigr)
$$
the first term has order $e^{\rk M^\Fr}$ and the last one is 0.
Because \hbox{$\frac{{2^kM'}}{{2^keM'}}\!\iso\!\frac{{M'}}{{eM'}}$},
it has odd order and hence has trivial $H^1$.
Next, from the long exact cohomology sequence
for the multiplication by $e$ map on $M'$ we extract
$$
  0 \lar \frac{{M'}^\Fr}{{eM'}^\Fr} \lar
    \Bigl(\frac{{M'}}{{eM'}}\Bigr)^\Fr \lar
    H^1(G,M')[e].
$$
The last term is killed by $|G|$ and $e$, and is therefore 0. So
$$
  \Bigl|\Bigl(\frac{{2^kM'}}{{2^keM'}}\Bigr)^\Fr\Bigr|=\Bigl|\Bigl(\frac{{M'}}{{eM'}}\Bigr)^\Fr\Bigr|=
    \Bigl|\frac{{M'}^\Fr}{{eM'}^\Fr}\Bigr|=e^{\rk M^\Fr},
$$
as required.

(5) By (4), we only need to show that $\phi(2^k,f)$ differs from
$\phi(1,f)$ by a power of 2. But the first and the last term in the
exact sequence
$$
  \Bigl(\frac{{M'}}{{2^kM'}}\Bigr)^{\Fr^f}
    \lar
  \Bigl(\frac{{M}}{{M'}}\Bigr)^{\Fr^f}
    \lar
  \Bigl(\frac{{M}}{{2^kM'}}\Bigr)^{\Fr^f}
    \lar
  H^1\Bigl(\langle F^f\rangle, \frac{{2M'}}{{2^kM'}}\Bigr)
$$
are killed by $2^k$, and the result follows from the definition of $\phi$.
\end{proof}

\begin{remark}
If we could also prove that $\phi(4,1)=\phi(2,1)$ for $\phi$ as in~\ref{NMM},
we would be able to deal with the exceptional case (4ex) of
Theorem \ref{loccompat} using an argument similar to that in \S\ref{ssnonsplit}.
It would then remove the ugly restriction on the reduction type for $p=2$ in
Theorem \ref{tamroot2}. (Embarassingly, this is purely a problem about
$\Z\Cy_{n}$-modules.)
\end{remark}

\noindent Combining Propositions \ref{rootabtwist} and \ref{avlcv2} completes
Case (4) of Proposition~\ref{propcvdv} and the proof of
Theorem \ref{loccompat}.

\section{Applications to the parity conjecture}
\label{stowers}

We now have a machine that, when supplied with a relation $\Theta$
between permutation representations,
confirms the $p$-parity conjecture for the twists of $A/K$
by the representations $\tau\in\Ttp$
coming from regulator constants.
We turn to a class of Galois groups where these are enough
to say something about essentially all twists for some $p$.

Specifically, we concentrate on Galois groups $G=\Gal(F/K)$
that have a normal $p$-subgroup $P$. The type of results that we aim for
is that knowing $p$-parity for all $G/P$-twists is sufficient to
establish it for all $G$-twists.
In particular, we prove Theorems \ref{introthmGPtwist} and \ref{introthm3}.

Apart from the machine itself (Theorem \ref{tamroot})
the proofs 
rely only on group theory and
basic parity properties of Selmer ranks and
root numbers.
Roughly speaking, we may consider any functions, such as
$\tau\mapsto w(A/K,\tau)$ or  $\tau\mapsto(-1)^{\blangle\tau,\Xp(A/F)\brangle}$
that satisfy
``self-duality'' and ``inductivity'' as in Proposition \ref{genparity}(1,2).
If two such functions agree on $G/P$-twists and
on the $\tau\in\Ttp$ for those $\Theta$ that come from dihedral subquotients,
this sometimes forces them to agree on all orthogonal $G$-twists, or
at least on those twists that correspond to intermediate fields.

We will not formulate the results of this section in this language. However,
to be able in principle to extend them to
a larger class of abelian varieties,
we axiomatise the minimal compatibility requirements:

\begin{hypothesis}[Compatibility in dihedral subquotients]
\label{decent}
Let $F/K$ be a Galois extension of number fields,
$A/K$ an abelian variety and $p$ a prime number.
We demand
the following\footnote{
  For odd $p$, the hypothesis may be relaxed
  to subquotients $U/N\iso \Di_{2p}$; this follows from the recent
  invariance result of Rohrlich (Prop. \ref{genparity}(5) or \cite{RohI} Thms. 1,2).
}$\,$:
whenever $N\normal U$ are subgroups of $\Gal(F/K)$ with $U/N\iso \Di_{2p^n}$
and $\tau=\sigma\oplus\triv\oplus\det\sigma$ for some 2-dimensional
representation $\sigma$ of $U/N$,
$$
  (-1)^{\blangle\tau, \Xp(A/F^N)\brangle_U} = w(A/F^U, \tau).
$$
In other words, the $p$-parity conjecture holds for the twists by all such
$\tau$.
(Recall that we regard $\Cy_2\times\Cy_2$ as a dihedral group as well.)
\end{hypothesis}

\pagebreak

\begin{theorem}
\label{hypellthm}
Hypothesis \ref{decent} holds for
\begin{enumerate}
\item (any $p$)
all elliptic curves over $K$
whose primes of additive reduction
above 2 and 3 have cyclic decomposition groups in $F/K$
(e.g. are unramified). 
\item ($p\neq2$)
all principally polarised abelian varieties over $K$
whose primes of unstable
reduction have cyclic decomposition groups in $F/K$
(e.g. all semistable principally polarised abelian varieties). 
\item ($p=2$)
abelian varieties over $K$ with a principal polarisation coming from a $K$-rational
divisor, whose primes of unstable reduction have cyclic decomposition groups in $F/K$,
and with split semistable reduction at those primes above $2$ that have
non-cyclic wild inertia groups in $F/K$.
\end{enumerate}
\end{theorem}

\begin{proof}
Apply Theorem \ref{tamroot} to the relations
in Examples \ref{exdihreg2} and \ref{exdihreg3}.
\end{proof}

Throughout the section we implicitly use that $\Xp$ behaves in an
``\'etale'' fashion: for $K\subset L\subset F$ an intermediate field,
$\Xp(A/L)\!=\!\Xp(A/F)^{\Gal(F/L)}$ (see e.g. \cite{Squarity} Lemma~4.14).
We occasionally say that ``$p$-parity holds'' for $A/L$ or for a twist of
$A$ by $\tau$, referring to Conjectures \ref{pparity}, \ref{pparitytwists}.

\subsection{Parity over fields}

\begin{theorem}
\label{thmGP}
Let $A$, $p\!\ne\!2$ and $F/K$ satisfy Hypothesis \ref{decent}.
Suppose $P\normal \Gal(F/K)$ is a $p$-subgroup.
If the $p$-parity conjecture holds for $A$ over all subfields
of $F^P/K$, then it holds over all subfields of $F/K$.
\end{theorem}

\begin{proof}
Write $G=\Gal(F/K)$
and $V$ for $Z(P)[p]$,
the $p$-elementary part of the centre of $P$.
We may assume $P\ne 1$, so $V$ is non-trivial.
As $V$ is characteristic in $P$, it is normal in $G$.
We need to prove $p$-parity for $A/F^H$ for all subgroups $H$ of $G$,
and it holds when $P\subset H$ by assumption.
We use induction on $|G|$ to reduce $G$ and $H$ to small explicit
groups. Thus, assume the theorem holds for all
proper subquotients $|\Gal(F'/K')|<|\Gal(F/K)|$.

Fix $H\raise1.7pt\hbox{$\scriptstyle\,\lneq\,\,$}G$.
Suppose there is a subgroup $1\ne U\normal G$ with $U\subset P$ and $HU\ne G$.
Applying the theorem to $P/U\normal\,\Gal(F^U\!/K)$,
$p$-parity holds over all subfields of $F^U/K$, including the
intermediate fields of $F^U/F^{HU}$. Applying it again to
$U\normal\Gal(F/F^{HU})$ shows that $p$-parity holds over the subfields
of $F/F^{HU}$, in particular $F^H$.

Hence we may assume that whenever $U\normal G$ is a subgroup of $P$, either
\hbox{$U\!=\!1$} or $HU=G$. In particular, $HV=G$ as $V\normal G$ is non-trivial.
Furthermore, $H\cap V=1$ because it is normal in $HV=G$ and
$H(H\cap V)=H\ne G$. It follows that $G\iso H\ltimes V$.

Moreover, $P=(P\cap H)\ltimes V$, as $P$ contains $V$.
The two constituents commute, so this is a direct product and
$V=Z(P)[p]=Z(P \cap H)[p]\times V$. So $P\cap H=1$, and hence $P=V$.

Finally, we may assume that the action of $H$ on $V$ by conjugation is
faithful. Otherwise let $W=\ker(H\to\Aut V)$ and note that $W\normal H$,
so $W\normal HV=G$. By induction, $p$-parity holds over all subfields
of $F^W/K$, in particular over $F^H$.

We are reduced to the case
$$
  G = H \ltimes \F_p^k\quad\text{with}\quad H\<\GL_k(\F_p)
    \qquad (\text{an affine linear group}),
$$
where we need to show that $p$-parity for $A/F^H$
follows from $p$-parity over the subfields of $F^{\smash{\F_p^k}}/K$.

The group $G$ acts on the one-dimensional complex characters of $\F_p^k$ by
conjugation. Let $\{\chi_i\}$ be a set of representatives for the orbits,
and let $S_i\< G$ be the stabiliser of $\chi_i$.
Extend $\chi_i$ to a character $\tilde\chi_i$ of $S_i$
by $\tilde\chi_i(hv)=\chi_i(v)$ for $h\in S_i\cap H$ and $v\in\F_p^k$.
The representations $\Ind_{S_i}^G \tilde\chi_i$ are irreducible
and distinct (\cite{SerLi} \S8.2), and we observe that
$$
  \C[G/H] \>\iso\> \bigoplus_i \Ind_{S_i}^G \tilde\chi_i.
$$
Indeed, both have dimension $p^k$, so it is enough to check that each term
on
the right is a consituent of $\C[G/H]$; but
\beq
  \blangle\C[G/H] , \Ind_{S_i}^G \tilde\chi_i\brangle\!\! &=\!\!&
  \blangle\triv_H , \Res_H\Ind_{S_i}^G \tilde\chi_i\brangle & \text{(Frobenius reciprocity)} \cr
  &\ge\!\!& \blangle\triv_H , \Ind_{S_i\cap H}^H \Res_{S_i\cap H}\tilde\chi_i\brangle & \text{(Mackey's formula)} \cr
  &=\!\!&\blangle\triv_H , \Ind_{S_i\cap H}^H \triv_{S_i\cap H}\brangle & \text{(definition of $\tilde\chi_i$)} \cr
  &=\!\!&\blangle\Res_{S_i\cap H}\triv_H ,\triv_{S_i\cap H}\brangle = 1 & \text{(Frobenius reciprocity).} \cr
\eeq
Now consider
$$
  \Sigma=\bigl\{i\bigm|\Ind_{S_i}^G \tilde\chi_i\text{\ self-dual}\bigr\}
        =\bigl\{i\bigm|\chi_i^{\pm 1}\text{\ belong to the same $H$-orbit}\bigr\}.
$$
For $i\in\Sigma$ let $M_i$ consist of those elements of $G$
that take $\chi_i$ to $\chi_i^{\pm 1}$,
and let $\psi_i=\Ind_{S_i}^{M_i}\tilde\chi_i$. Computing modulo 2,
\beq
  \blangle\triv_H, \Xp(A/F)\brangle &=&
  \blangle\C[G/H], \Xp(A/F)\brangle  & \text{(Frobenius reciprocity)} \cr
  &\equiv& \sum_{i\in\Sigma} \blangle\Ind_{S_i}^G \tilde\chi_i , \Xp(A/F)\brangle
     & \text{(Self-duality of $\Xp$)} \cr
  &\equiv& \sum_{i\in\Sigma} \blangle\psi_i, \Xp(A/F)\brangle
     & \text{(Frobenius reciprocity)}. \cr
\eeq
The same computation for the root numbers (using \ref{genparity}(1,2))
shows that $ w(A/F^H)=\prod_i w(A/F^{M_i},\psi_i)$.
So, it suffices to prove that
$$
  (-1)^{\blangle\psi_i, \Xp(A/F)\brangle} = w(A/F^{M_i},\psi_i).
$$
If $\chi_i\!=\!\triv$, then $S_i\!=\!G$, $\psi_i\!=\!\triv$ and this
$p$-parity holds by assumption.
Otherwise, $\chi_i\ne\chi_i^{-1}$ as $p$ is odd, and
$\psi_i$ factors through the $\Di_{2p}$-subquotient $M_i/\ker\tilde\chi_i$.
In this case we know
$p$-parity over its two bottom fields $F^{S_i}$ and $F^{M_i}$,
so it also holds for the twist of $A/F^{M_i}$ by $\psi_i$
(Hypothesis \ref{decent}).
\end{proof}

\pagebreak

\begin{theorem}
\label{thmGP2}
Let $A, p=2$ and $F/K$ satisfy Hypothesis \ref{decent}.
Suppose that the Sylow 2-subgroup $P$ of $\Gal(F/K)$ is normal.
If the $2$-parity conjecture holds
for $A$ over $K$ and its quadratic extensions in $F$,
then it holds over all subfields of $F/K$.
\end{theorem}

\begin{proof}
Write $G=\Gal(F/K)$ and pick $H\< G$. We prove $p$-parity for $A/F^H$.

\Step1{Suppose $G$ is a 2-group}

There is a descending chain of subgroups
$G=U_1\supset\ldots\supset U_n=H$ with all inclusions
of index 2.
We show by induction that $2$-parity holds over all quadratic
extensions of $F^{U_i}$ in $F$. For $i=1$ this is true by assumption.
Suppose this is true for $i-1$, and let
$L/F^{U_i}$ be a quadratic extension inside~$F$.
The Galois closure of the quartic extension $L/F^{U_{i-1}}$ has
Galois group $\Cy_4$, $\Cy_2\!\times\!\Cy_2$ or $\Di_8$, as it is a 2-group.
In all cases, 2-parity over quadratic extensions of $F^{U_{i-1}}$
implies 2-parity for all orthogonal twists of this Galois group, in
particular parity over $L$
(for $\Cy_4$ see Corollary~\ref{corbasic}(1,2); for $\Di_8$
this is Hypothesis \ref{decent}.)

\Step2{General case}

As $F^H/F^{H\cap P}$ is Galois of odd degree,
2-parity for $A/F^H$ is equivalent to that for $A/F^{H\cap P}$
by Corollary \ref{corbasic}(3).
Since $P$ is a 2-group, by Step~1 it suffices to establish 2-parity
over $F^P$
and its quadratic extensions in $F$.

Let $\Phi\normal P$ be its Frattini subgroup, so $P/\Phi\iso\F_2^k$ is its
largest 2-elementary quotient. As $\Phi$ is characteristic in $P$,
it is normal in $G$, and $F^\Phi/K$ is Galois.
($F^\Phi$ is the compositum of all quadratic extensions of $F^P$ in $F$.)
Replacing $F$ by $F^\Phi$ we may assume
that $\Phi=0$ and $P=\F_2^k$, so by the Schur-Zassenhaus theorem
$G\iso U\ltimes\F_2^k$ with $U$ of odd order.

We want to prove 2-parity for all twists of $A/F^P$ by characters
\hbox{$\chi: \F_2^k\to \C^{\times}\!\!.$} Write $L_\chi$ for $F^{\ker\chi}$ for such $\chi$;
so $[L_\chi:F^P]\le 2$.

As $F^P/K$ is Galois of odd degree,
2-parity holds over $L_\triv=F^P$,
equivalently for the twist of $A/F^P$ by $\triv$.
More generally,
$G$ acts on characters of~$\F_2^k$ by conjugation,
and if $\chi\ne\triv$ is $G$-invariant,
then $L_\chi/K$ is Galois with Galois group $U\!\times\!\Cy_2$.
In this case, $L_\chi$ is an odd degree Galois extension of a quadratic
extension of $K$ in $F$, so again 2-parity holds over $L_\chi$ and hence
for the twist of $A/F^P$ by $\chi$.

Now pick a general non-trivial $\chi=\chi_1$ and
let $\{\chi_i\}_{1\le i\le n}$ be the complete set of
its conjugates under $G$.
The $L_i$ are conjugate fields, so the 2-parity conjecture
for the twist by $\chi$ is equivalent to that for any of the $\chi_i$.
As the orbit size $n$ is odd, it suffices to check 2-parity for the twist
of $A/F^P$ by~$\oplus_i\chi_i$.

Applying Hypothesis \ref{decent} in $\Cy_2\!\times\!\Cy_2$\hskip 1pt-\hskip 1pt extensions of $F^P$,
2-parity holds for the twist by $\triv\oplus\phi\oplus\psi\oplus\phi\psi$
for any characters $\phi, \psi$ of $\F_2^k$.
Taking a sum of such twists shows that 2-parity for $\oplus_i\chi_i$ is equivalent
to 2-parity for $\triv^{\oplus n}\oplus\prod\chi_i$.
But this is a sum of $G$-invariant characters, for which 2-parity has already
been established.
\end{proof}

\pagebreak

\subsection{Parity for twists}

\begin{theorem}
\label{thmGPtwist}
Let $A, p$ and $F/K$ satisfy Hypothesis \ref{decent}.
Assume that the Sylow $p$-subgroup $P$ of $G=\Gal(F/K)$ is normal
and $G/P$ is abelian.
If the $p$-parity conjecture holds for $A$ over $K$
and its quadratic extensions in $F$,
then it holds for all twists of $A$ by orthogonal representations of $G$.
\end{theorem}

\begin{proof}
Let $\tau$ be an orthogonal representation of $G$.
By the analogue of Brauer's induction theorem for orthogonal representations
\cite{FQ} (2.1),
$$
  \tau = \bigoplus_i \Ind_{H_i}^G \rho_i^{\oplus n_i}
$$
for some $H_i\< G$, $n_i\in\Z$, and with $\rho_i$ either (a) trivial or
(b) $\chi\oplus\bar\chi$ with $\chi\ne\bar\chi$ one-dimensional or
(c) a 2-dimensional irreducible that factors through a dihedral quotient
of $H_i$.

By inductivity (Corollary \ref{corbasic}(2)), it suffices to prove that
$$
  (-1)^{\blangle\rho_i, \Xp(A/F)\brangle_{H_i}} = w(A/F^{H_i},\rho_i).
$$
We distinguish between the three possibilities for $\rho_i$ as above:

Case (a). As $G/P$ is abelian, its only irreducible self-dual
representations are those that factor through a $\Cy_2$\hskip 1pt-\hskip 1pt quotient.
By ``self-duality'' and ``inductivity'' (Corollary~\ref{corbasic}(1,2)),
the assumed parity over $K$ and its quadratic extensions
implies parity in all subfields of $F^P/K$.
By Theorems \ref{thmGP}~and~\ref{thmGP2}, it implies parity in all subfields
of $F/K$, in particular for $A/F^{H_i}$.

Case (b). The formula holds by Corollary \ref{corbasic}(1).

Case (c). Since the commutator of $G$ is a $p$-group, the only dihedral
subquotients it has are $\Di_{2p^k}$.
By case (a), we know parity over $F^{H_i}$ and its quadratic extensions
in $F$, so Hypothesis \ref{decent} implies parity for all
irreducible $2$-dimensional representations of this subquotient.
\end{proof}

\begin{remark}
\label{knowncases}
For elliptic curves, the assumption that the $p$-parity conjecture
holds for $E$ over $K$ and its quadratic extensions in $F$
is known in a number of cases. In particular
(\cite{BS,Gre,Guo,Mon,Nek,Kim,Squarity}, \cite{Isogroot,CFKS})
\begin{enumerate}
\item if $K=\Q$;
\item if $E/K$ admits a rational $p$-isogeny, and
for every prime $v|p$ of $K$,
\begin{itemize}
\item ($p>3$) $E$ is semistable, potentially multiplicative or potentially
ordinary at $v$, or acquires good supersingular reduction over an abelian
extension of $K_v$.
\item ($p=3$) $E$ is semistable at $v$,
\item ($p=2$) $E$ is semistable, and not supersingular at $v$.
\end{itemize}
\end{enumerate}
There are also results for modular abelian varieties over totally real fields
\cite{Nek,Kim2,NekIV} and a generalisation of (2) to abelian varieties with a
suitable $p^g$-isogeny \cite{CFKS}.
\end{remark}

\pagebreak

\begin{remark}
The assumption in Theorem \ref{thmGPtwist} that $G/P$ is abelian
was only used to ensure that (a) $p$-parity holds in all intermediate
fields of $F^P/K$, and (b) dihedral subquotients of $G$
have the form $\Di_{2p^n}$. So the theorem extends to other extensions
that satisfy (a) and (b), e.g. $G$ nilpotent with $p=2$, or
$G/P\>\iso\>$(odd)$\times$(abelian 2-group) with $p\ne 2$.
\end{remark}

\begin{example}
\label{exgl2f3}

Let $E/\Q$ be an elliptic curve, semistable at 2 and 3, and let $F=\Q(E[3])$.
We claim that the $3$-parity conjecture holds
for $E$ over all subfields of $F$, and consequently over all subfields
of $F_n=\Q(E[3^n])$ by Theorems \ref{hypellthm} and \ref{thmGP}. 

If either $F/\Q$ is abelian or $E/\Q$ admits a rational 3-isogeny,
this is true by \cite{Selfduality} Thm. 1.2
and \cite{Isogroot} Thm. 2 respectively.
Otherwise, $G=\Gal(F/\Q)$ is one of the following subgroups of $\GL_2(\F_3)$:
$$
  \GL_2(\F_3), \quad
  \Di_8
  \quad\text{or}\quad
  \Sy_{16}=\text{2-Sylow of\ }\GL_2(\F_3)
  .
$$
It is not hard to verify that in all three cases,
the representations $\Ind_H^G\triv_H$ for subgroups
$H\subset \Cy_2\!\times\!\Cy_2$ (these correspond to fields where $E$ acquires a 3-isogeny)
and those with $G/H$ abelian generate all orthogonal representations. Again,
as the 3-parity conjecture is known for $E/F^H$ for such $H$,
this implies 3-parity for all intermediate fields.

The question whether 3-parity holds for all {\em twists\/} by self-dual
representations of $\Gal(F_n/\Q)$ is more subtle, as we do not have
an analogue of Theorem~\ref{thmGPtwist} in this case. In fact, suppose
that $G_2=\Gal(F_2/\Q)\iso\GL_2(\Z/9\Z)$, i.e. as large as possible.
Then there are precisely two irreducible orthogonal Artin
representations $\tau_1, \tau_2: G_2\to\GL_6(\C)$
that can be realised over $\Q_3(\sqrt{3})$ but not over $\Q_3$.
It turns out that $\RC_\Theta^{\Q_3}(\tau_1\oplus\tau_2)=1$ for
every $G_2$-relation $\Theta$, so the parity of
$\blangle\tau_i, \X_3(E/F_2)\brangle$ cannot be computed
from regulator constants. (It can be computed for all other
$\C G_2$-irreducible orthogonals.)
\end{example}

\def\theequation{A.\arabic{equation}}
\section*{Appendix A: Basic parity properties}
\setcounter{equation}{0}

For the convenience of the reader, we record a few basic facts related to
root numbers and the $p$-parity conjecture.

\begin{lemma}[Determinant formula]
\label{detformula}
Let $K$ be a local field and $\tau$ a continuous representation
of the Weil group of $K$. Then
$$
  w(\tau\oplus\tau^*) = \det(\tau)(\theta(-1))\>,
$$
where $\theta$ is the local reciprocity map on $K^{\times}$.
For an abelian variety $A/K$,
$$
  w(A,\tau\oplus\tau^*) = 1 \>.
$$
\end{lemma}
\begin{proof}
For the first statement see \cite{RohE} p.145 or \cite{TatN} 3.4.7.
The second is an elementary computation using \cite{TatN} 3.4.7, 4.2.4.
\end{proof}

\pagebreak

\begin{proposition}
\label{genparity}
Let $F/K$ be a Galois extension of number fields,
and $A/K$ an abelian variety. Write $\EFC=A(F)\tensor\C$ and $\X=\Xp(A/F)$.
For an Artin representation $\tau: \Gal(F/K)\to\GL_n(\C)$,
\begin{enumerate}
\item (self-duality)
$$
  \blangle\tau,\EFC\brangle = \blangle\tau^*,\EFC\brangle, \qquad
  \blangle\tau,\X\brangle = \blangle\tau^*,\X\brangle, \qquad
  w(A,\tau)=\overline{w(A,\tau^*)}.
$$
\item (inductivity)
If $K\subset L\subset F$ and
$\tau=\Ind_{\Gal(F/L)}^{\Gal(F/K)}\rho$ then
$$
  \blangle\tau,\EFC\brangle =
     \blangle\rho,\EFC\brangle, \quad
  \blangle\tau,\X\brangle =
     \blangle\rho,\X\brangle, \quad
  w(A/K,\tau)=w(A/L,\rho).
$$
\item (odd degree base change)
If $F/K$ is Galois of odd degree, then
\beq
  \rk(A/K)&\equiv&\rk(A/F)&\mod 2, \cr
  \dim\Xp(A/K)&\equiv&\dim\Xp(A/F)& \mod 2, \cr
  w(A/K)&=&w(A/F).
\eeq
\item (orthogonality)
If $\tau$ is symplectic, then $\blangle\tau,\!\EFC\brangle$ is even and
$w(A,\tau)\!=\!1$.
\item (equivariance)
If $\tau$ is self-dual and $\tau'$ is a Galois conjugate of $\tau$
(i.e. their characters are Galois conjugate), then
$$
  \blangle\tau,\EFC\brangle =
     \blangle\tau',\EFC\brangle, \quad
  w(A/K,\tau)=w(A/K,\tau').
$$
\end{enumerate}
\end{proposition}

\begin{proof}
(1) $\EFC$ is a rational representation, hence self-dual; $\X$ is self-dual
as well by \cite{Selfduality} Thm. 1.1. 
The root number formula follows from Lemma \ref{detformula}.\\ 
(2) For $\EFC$ and $\X$ this is Frobenius reciprocity.
The last formula is well-known; it is a consequence of inductivity in
degree 0 of local $\epsilon$-factors for Weil groups,
\cite{TatN} 4.2.4 and a simple determinant computation.\\
(3) Follows from (1), (2) and the fact that the only self-dual irreducible
representation of $\Gal(F/K)$ is trivial. \\
(4) $\EFC$ is a rational representation, so $\blangle\tau,\EFC\brangle$ is even;
$w(A,\tau)=1$ by \cite{RohG} Prop. 8(iii) for elliptic curves,
and by \cite{SabT} Prop. 3.2.3 for abelian varieties.\\
(5) The statement for $\EFC$ is clear, as it is a rational representation.
That for root numbers is clear from (4) in the symplectic case, and follows
from \cite{RohI} Thms 1,2 in the orthogonal case.
\end{proof}

\begin{corollary}
\label{corbasic}
Suppose $F/K$ is a Galois extension, and $A/K$ an abelian variety. Then
the $p$-parity conjecture
\begin{enumerate}
\item holds for twists of $A$ by representations of the form $\tau\oplus\tau^*$.
\item holds for the twist of $A/K$ by $\Ind_{\Gal(F/L)}^{\Gal(F/K)}\rho$
      if and only if it holds for the twist of $A/L$ by $\rho$,
      if $K\subset L\subset F$.
\item holds for $A/F$ if and only if it holds for $A/K$, if $[F:K]$ is odd.
\end{enumerate}
\end{corollary}


\begin{acknowledgements}
We would like to thank the referee for carefully reading the manuscript
and for helpful comments.
\end{acknowledgements}

\pagebreak

\end{document}